\documentclass[12pt, reqno]{amsart}
\usepackage{amsmath}
\usepackage{amstext}
\usepackage{amsbsy}
\usepackage{amsopn}
\usepackage{upref}
\usepackage{amsthm}
\usepackage{amsfonts}
\usepackage{amssymb}
\usepackage{mathrsfs}
\usepackage{times}
\usepackage{fancybox} 
\usepackage{ascmac}
\allowdisplaybreaks
 \newtheorem{theorem}{Theorem}[section]
 \newtheorem{corollary}[theorem]{Corollary}
 \newtheorem{proposition}[theorem]{Proposition}

\theoremstyle{definition}
 
 \theoremstyle{remark} 
 \newtheorem{remark}[theorem]{Remark}

\newcommand{\p}{\partial}
\newcommand{\RR}{\mathbb{R}}

\renewcommand{\Re}{\operatorname{Re}}
\renewcommand{\Im}{\operatorname{Im}}
\numberwithin{equation}{section}
\begin{document}
\title[Structure of a fourth-order dispersive flow equation]
{Structure of a fourth-order dispersive flow equation 
through the generalized Hasimoto transformation
}
\author[E.~Onodera]{Eiji Onodera}
\address[Eiji Onodera]{Department of Mathematics and Physics, 
Faculty of Science and Technology, 
Kochi University, 
Kochi 780-8520, 
Japan}
\email{onodera@kochi-u.ac.jp}
\begin{abstract}
This paper focuses on a one-dimensional fourth-order nonlinear 
dispersive partial differential equation for curve flows on a K\"ahler manifold. 
The equation arises as a fourth-order extension of the one-dimensional 
Schr\"odinger flow equation, 
with physical and geometrical backgrounds.  
First, this paper presents a framework that can transform the equation 
into a system of fourth-order nonlinear dispersive partial differential-integral equations 
for complex-valued functions. 
This is achieved by developing the so-called generalized Hasimoto 
transformation, which enables us to handle general higher-dimensional 
compact K\"ahler manifolds. 
Second, this paper demonstrates the computations to obtain the explicit expression 
of the derived system for three examples of the compact K\"ahler manifolds, 
dealing with the complex Grassmannian as an example in detail.
In particular, the result of the computations 
when the manifold is a Riemann surface or a complex Grassmannian verifies 
that the expression of the system derived by our framework 
actually unifies the ones derived previously.  
Additionally, the computation when the compact K\"ahler manifold has a constant 
holomorphic sectional curvature, 
the setting of which has not been investigated,  
is also demonstrated. 
\end{abstract}
\maketitle
\section{Introduction}
\label{section:introduction}
\subsection{Setting of the problem and previous related results}
\label{subsection:setting}
Let $N$ be a K\"ahler manifold 
of complex dimension $n\in \mathbb{N}$ 
with complex structure $J$ and K\"ahler metric $h$, 
and let $\nabla$ and $R$ denote the Levi-Civita connection associated to $h$ and the 
Riemann curvature tensor respectively.  
\footnote{Throughout this paper, the definition of $R$ is adopted to be   
$
R(X,Y)Z:=\nabla_X\nabla_YZ-\nabla_Y\nabla_XZ-\nabla_{[X,Y]}Z
$ 
for any vector fields $X,Y,Z$ on $N$ where $[X,Y]:=XY-YX$.} 
This paper investigates a one-dimensional fourth-order nonlinear 
dispersive partial differential equation (PDE) for curve flows on $N$,  
which is formulated by 
\begin{alignat}{2}
 & u_t
  =
  a\,J_u\nabla_x^3u_x
  +
  \lambda\, J_u\nabla_xu_x
  +
  b\, R(\nabla_xu_x,u_x)J_uu_x
  +
  c\, R(J_uu_x,u_x)\nabla_xu_x
\label{eq:pde}
\end{alignat}
for a smooth map $u=u(t,x):(-T,T)\times \RR\to N$. 
Here, $0<T\leqslant \infty$, 
and $a\ne 0$, $b$, $c$, $\lambda$ are real constants,  
$u_t=du\left(\frac{\p}{\p t}\right)$, $u_x=du\left(\frac{\p}{\p x}\right)$, 
$\nabla_t$ and $\nabla_x$ denote the covariant derivatives 
along $u$ with respect to $t$ and $x$ respectively, 
and $J_u$ denotes the complex structure at $u=u(t,x)\in N$. 
Geometrically,  \eqref{eq:pde} describes
the relationship among elements of $\Gamma(u^{-1}TN)$, 
where $\Gamma(u^{-1}TN)$ denotes the set of smooth sections of $u^{-1}TN$. 
\par 
The equation \eqref{eq:pde} 
for curve flows on the canonical two-sphere $\mathbb{S}^2$ 
with additional assumptions $\lambda=1$ and $c=3(a-b)/2$ 
is the typical example with physical backgrounds, 
which is known to coincide with the following 
three-component system of PDEs  
\begin{equation}
\label{eq:ss3}
u_t=u\wedge 
\left\{
a\,u_{xxxx}+ u_{xx}
+(5a-b)(u_{xx},u_x)u_x
+\dfrac{5a-b}{2}(u_x,u_x)u_{xx}
\right\}
\end{equation}
for $u=u(t,x):(-T,T)\times \RR\to \mathbb{S}^2\hookrightarrow \RR^3$. 
Here, $\wedge$ and $(\cdot,\cdot)$ denote  
the exterior and the inner product in $\RR^3$ respectively, 
$J_u=u\wedge$ on $T_u\mathbb{S}^2$, 
and $h$ corresponds to the metric induced from $(\cdot,\cdot)$.   
The system \eqref{eq:ss3} was 
proposed by Lakshmanan et al.(\cite{LPD,PDL}), 
modeling the continuum limit of a one-dimensional isotropic Heisenberg 
ferromagnetic spin chain systems with nearest neighbor bilinear 
and bi-quadratic exchange interaction. 
It was proved in \cite{LPD,PDL} that \eqref{eq:ss3}  
can be reduced (equivalently if $\gamma_2=-\frac{5}{2}\gamma_1$) to the following fourth-order nonlinear Schr\"odinger equation 
\begin{align}
&\sqrt{-1}q_t
+\gamma_1q_{xxxx}
+(q_{xx}+2|q|^2q)
-4\delta_1|q|^2q_{xx}
-4\delta_2q^2\overline{q}_{xx}
\nonumber
\\
&\quad
-4\delta_3q|q_x|^2
-4\delta_4q_x^2\overline{q}
-24\delta_5|q|^4q
=0
\label{eq:4shro}
\end{align}
for $q=q(t,x):(-T,T)\times \RR\to \mathbb{C}$, 
where  
$\delta_1=3\gamma_1+2\gamma_2$, 
$\delta_2=2\gamma_1+\gamma_2$, 
$\delta_3=9\gamma_1+4\gamma_2$,
$\delta_4=\frac{7}{2}\gamma_1+2\gamma_2$,  
$\delta_5=\gamma_1+\frac{1}{2}\gamma_2$, 
and 
$(\gamma_1,\gamma_2)=(a,-(5a-b)/2)$. 
It was also revealed in \cite{LPD,PDL} that 
\eqref{eq:4shro} is completely integrable 
in the sense of Painlev\'e singularity structure analysis
if and only if $\gamma_2=-\frac{5}{2}\gamma_1$. 
Additionally, \eqref{eq:4shro} with
$\gamma_2=-\frac{5}{2}\gamma_1$ arises in other contexts as well, such as 
the dynamics of a one-dimensional anisotropic Heisenberg ferromagnetic spin chain with octupole-dipole interaction in the continuum limit(\cite{DKA}) and    
the molecular excitations along the hydrogen bonding spine in an alpha helical protein with higher-order molecular interactions under specific parameter choices(\cite{DL}).
\par 
The equation \eqref{eq:pde} for curve flows 
on $\mathbb{S}^2$ with $\lambda=1$ and $c=3(a-b)/2$
can be also derived from 
the so-called Fukumoto-Moffatt model for a vortex filament (\cite{fukumoto,FM}). 
The equation proposed in \cite{fukumoto,FM} describes the motion of 
an arc-length-parameterized space curve in $\RR^3$,  
denoted by 
$\vec{\bf{X}}=\vec{\bf{X}}(t,x):(-T,T)\times \RR\to \RR^3$
here, 
which models the three-dimensional dynamics of a vortex 
filament in an incompressible viscous fluid, considering
the deformation effect of the vortex core due to the self-induced strain. 
If $\vec{\bf{X}}$ satisfies the vortex filament equation, 
then the velocity vector $u:=\vec{\bf{X}}_x$ 
takes values in $\mathbb{S}^2$ and satisfies  
\eqref{eq:ss3}. 
Furthermore, it is proved by \cite{fukumoto,FM}) that 
the equation for $\vec{\bf{X}}$ is transformed to  
\eqref{eq:4shro} via the so-called Hasimoto transformation(\cite{Hasimoto}). 
\par 
In the context of geometric dispersive PDEs, 
\eqref{eq:pde} is regarded as a fourth-order extension of the 
so-called one-dimensional Schr\"odinger flow equation 
\begin{equation}
u_t=J_u\nabla_xu_x. 
\label{eq:SF}
\end{equation}
In fact, \eqref{eq:pde} can be regarded as 
the so-called generalized bi-Schr\"odinger flow equation 
if $(N,J,h)$ is a locally Hermitian symmetric space and $c=3(a-b)/2$: 
The generalized bi-Schr\"odinger flow equation was originally 
introduced by Ding and Wang in \cite{DW2018} 
as a PDE for time-dependent maps $u=u(t,x):(-T,T)\times M\to N$ 
where $M$ is a Riemannian manifold and 
$N$ is a K\"ahler or para-K\"ahler manifold. 
When $M=\RR$ with the Euclidean metric and $(N,J,h)$ is a K\"ahler manifold, 
it is defined as the following Hamiltonian gradient flow equation 
\begin{equation}
u_t=J_u\nabla E_{\alpha,\beta,\gamma}(u), 
\label{eq:bibi} 
\end{equation}  
where $\beta\ne 0$ and $\alpha, \beta$ are real constants 
and $\nabla E_{\alpha,\beta,\gamma}(u)$ denotes 
the gradient (not the Levi-Civita connection only here) of the energy functional
\begin{align}
E_{\alpha,\beta,\gamma}(u)
&:=
\frac{\alpha}{2}\int_{\RR}h(u_x,u_x)\,dx
+\frac{\beta}{2}\int_{\RR}h(\nabla_xu_x,\nabla_xu_x)\,dx
\nonumber
\\
&\quad
+
\gamma
\int_{\RR}h(R(u_x,J_uu_x)J_uu_x,u_x)\,dx.
\nonumber
\end{align}   
As is pointed out in \cite{onodera4}, 
if $(N,J,h)$ is a locally Hermitian symmetric space, 
then the explicit expression of 
\eqref{eq:bibi} coincides with \eqref{eq:pde} under the setting  
\begin{equation}
a=\beta, b=\beta+8\gamma, c=\frac{3(a-b)}{2}, \lambda=-\alpha
\label{eq:11021}
\end{equation}
\par 
Additionally, another fourth-order extension of \eqref{eq:SF}, 
also generalizing \eqref{eq:ss3},  
has been investigated in \cite{chihara2,CO2,onodera0,onodera3}, 
which is formulated by 
\begin{align}
u_t&=a_1J_u\nabla_x^3u_x+a_2J_u\nabla_xu_x
+a_3h(u_x,u_x)J_u\nabla_xu_x+a_4h(\nabla_xu_x,u_x)J_uu_x
\label{eq:4ono}
\end{align}
for curve flow $u=u(t,x)$ on a K\"ahler manifold $(N,J,h)$. 
As is shown in \cite{onodera4},
\eqref{eq:4ono} is also regarded as  \eqref{eq:bibi}
provided that $(N,J,h)$ is a Riemann surface with 
constant Gaussian curvature. 
However, the assumption seems to be a bit strong geometrically. 
It can be said that \eqref{eq:pde} modifies \eqref{eq:4ono} to be 
geometrically more reasonable by considering some kind of symmetry and 
curvature on $(N,J,h)$ as a 
K\"ahler manifold. 
\par 
This paper is concerned with correspondences between 
geometric dispersive PDEs for curve flows 
and systems of nonlinear PDEs for complex-valued functions
 (or equations for complex-matrix-valued functions), 
such as that between \eqref{eq:ss3} for curve flows on $\mathbb{S}^2$ 
and \eqref{eq:4shro}. 
These correspondences 
have attracted much attention from researchers in mathematical physics, 
differential geometry, and theory of PDEs.
Understanding them has the potential to promote 
the studies in both directions complementarily each other. 
In this connection, we next focus on two seemingly different methods 
developed in previous studies of the geometric dispersive PDEs 
(except of the fourth-order equations \eqref{eq:pde}, \eqref{eq:ss3}, and \eqref{eq:4ono} 
which is stated later in Introduction).   
\par
The first type of method is based on the development map 
acting on a space of smooth curves on $N$,  
embedding $N$ as an adjoint orbit in an associated Lie algebra.  
Notably, the method essentially applies the properties of 
Hermitian symmetric spaces as $(N,J,h)$. 
It is to be stated first that Zakharov and Takhtadzhyan \cite{ZT} showed  
that the Schr\"odinger flow equation \eqref{eq:SF} for curve flows on $\mathbb{S}^2$
is equivalent to the cubic nonlinear Schr\"odinger equation(NLS) for complex-valued functions. 
As the generalization, it was established by Terng and Uhlenbeck \cite{TU}
that \eqref{eq:SF} for curve flows on 
a compact complex Grassmannian is equivalent to the matrix NLS which was first 
studied by  Fordy and Kulish \cite{FK}. 
It is also pointed out in \cite{TU} that the existence of a time-global solution 
to the initial value problem of \eqref{eq:SF} for curve flows on the compact complex Grassmannian follows from the correspondence.    
The above equivalence with respect to \eqref{eq:SF} 
was investigated further by Terng and Thorbergsson \cite{TT} 
for the other three types of compact Hermitian symmetric spaces as $(N,J,h)$.      
Additionally, the equivalence with respect to \eqref{eq:SF} for curve flows 
on $\mathbb{S}^2$ in the periodic setting in $x$
has been obtained by Liu \cite{liu} recently. 
The interested readers can also refer to, e.g., \cite{AA,DW2018,DZ2021}, 
for more details related to the method. 
\par 
The second type of method, called the generalized Hasimoto transformation,   
is to transform a geometric dispersive PDE for curve flows into a nonlinear 
dispersive PDE for complex-valued functions or a system of them, 
by constructing a parallel (in $x$-direction) orthonormal frame along a curve flow $u=u(t,x)$ 
and then by expressing the equation satisfied by the components of $u_x$ 
with respect to the frame. 
We expect that this method can handle Riemann surfaces or more general
K\"ahler manifolds as $(N,J,h)$ without imposing any symmetry or curvature conditions, 
and the expressions of the derived equations or systems are simpler, 
in that they are semilinear ones without any second-highest order derivatives in spatial variable. 
In addition, we expect that the derived expressions present insights on the 
essential structure of the original geometric dispersive PDEs,    
although constructing the inverse of the transformation remains further discussion. 
Indeed, these insights have been sometimes applicable to the time-global solvability of 
the initial value problem for 
original geometric dispersive PDEs. 
See, e,g, the work by Chang et al. (\cite{CSU}), 
Nahmod et al. (\cite{NSVZ}), 
Rodnianski et al.(\cite{RRS}) 
for \eqref{eq:SF}. 
Moreover, also for some analogous third-order geometric dispersive PDEs for curve flows 
on compact Riemann surfaces, 
equations derived from 
the generalized Hasimoto transformation have been investigated 
in \cite{onodera0} and in \cite{SW2011,SW2013}, 
independently on whether any direct applications to the time-global solvability of 
their initial value problem exist or not. 
Notably, many results based on the method seem to handle 
``open" curve flows, where 
the spacial domain of $x$ is the real line $\RR$. 
On the other hand, if a geometric dispersive PDE for closed curve flows is considered, 
then the method requires some modifications, 
involving holonomy corrections along the closed curves to transform 
into PDEs for complex-valued functions that are periodic in $x$,   
which becomes rather complicated. 
Nonetheless, this case has been also investigated for \eqref{eq:SF}. 
See \cite{Koiso1995,Koiso1997}  for closed curve flows on locally Hermitian symmetric spaces, 
and \cite{RRS} for those on Riemann surfaces. 
\par 
In contrast, investigating our fourth-order  PDE \eqref{eq:pde} 
in this context still remains unexplored. 
Related previous results on the correspondences are 
limited to when $N$ is any one of 
$G_{n_0,k_0}$  (including  $\mathbb{S}^2\cong G_{2,1}$), 
$G_{n_0}^{k_0}${\footnote{Here and hereafter, $G_{n_0,k_0}$ (resp. $G_{n_0}^{k_0}$) for integers 
$n_0,k_0$ satisfying $1\leqslant k_0<n_0$
denotes the compact (resp. noncompact) complex Grassmannian 
inherited with the structure as a Hermitian symmetric space.},
 and a Riemann surface, 
which are stated 
more concretely in the next two paragraphs.  
\par 
When $N$ is either of $G_{n_0,k_0}$
or $G_{n_0}^{k_0}$, 
our equation \eqref{eq:pde} with $c=3(a-b)/2$ for $u=u(t,x):(-T,T)\times \RR\to N$
is proved to be equivalent to a fourth-order matrix nonlinear 
(Schr\"odinger-like) differential-integral equation, which 
follows from the results by Ding and Wang (\cite{DW2018}):  
Taking the compact case $N=G_{n_0,k_0}$ of complex dimension 
$n=k_0(n_0-k_0)$ as the example,  
the authors in \cite{DW2018} investigated \eqref{eq:bibi} 
(not as \eqref{eq:pde}), 
and equivalently transformed it to 
\begin{align}
&\sqrt{-1}q_t
-
\alpha
\biggl\{
q_{xx}+2qq^{*}q\biggr\}
+\beta
\biggl\{
q_{xxxx}+4q_{xx}q^{*}q+2qq^{*}_{xx}q+4qq^{*}q_{xx}
\nonumber
\\
&\quad
+2q_xq^{*}_xq+6q_xq^{*}q_x+2qq^{*}_xq_x+6qq^{*}qq^{*}q
\biggr\}
-2(\beta+8\gamma)
\biggl\{
(qq^{*}q)_{xx} 
\nonumber
\\
&\quad
+2qq^{*}qq^{*}q
+q\left(\int_{0}^x
q^{*}(qq^{*})_sq\,ds
\right)
+
\left(
\int_{0}^x
q(q^{*}q)_sq^{*}\,ds
\right)q
\biggr\}=0
\label{eq:mns}
\end{align}
for $q=q(t,x):(-T,T)\times \RR\to \mathcal{M}_{k_0\times (n_0-k_0)}$,  
where $\mathcal{M}_{k_0\times (n_0-k_0)}$ stands for the space of  
$k_0\times (n_0-k_0)$ complex-matrices and 
$q^{*}=\overline{q}^t$ is the transposed conjugate matrix-valued function of $q$. 
The proof basically employs the first type of method established by \cite{TU} 
based on the development map, 
combined with the idea of PDEs with given (non-zero) 
curvature representation in the category of Yang-Mills theory. 
They succeed to establish the results, 
even though the considered equation 
is not completely integrable unless $\beta+8\gamma=0$, 
revealing also that the nonlocal terms of integral type in \eqref{eq:mns} 
vanish under the integrable condition.  
It is also pointed out in \cite{DW2018} that, if $N=G_{2,1}\cong \mathbb{S}^2$, 
then the nonlocal terms of \eqref{eq:mns} vanish without the integrable condition, 
and \eqref{eq:mns} reduces to \eqref{eq:4shro} under the setting \eqref{eq:11021}. 
\par 
When $(N,J,h)$ is a Riemann surface with constant Gaussian curvature, 
our equation \eqref{eq:pde} with $c=3(a-b)/2$ for $u=u(t,x):(-T,T)\times \RR\to N$
is proved to be transformed to a fourth-order nonlinear 
dispersive PDE without non-local integral terms 
for complex-valued functions, 
which follows from the results by Ding and Zhong (\cite{DZ2021}): 
The authors in \cite{DZ2021} investigated \eqref{eq:bibi} 
for $u=u(t,x):(-T,T)\times \RR\to N$ where $(N,J,h)$ is a Riemann surface. 
Assuming the existence of a time-independent edge point $\displaystyle\lim_{x\to\infty}u(t,x)\in N$, 
they employed the second type of method based on the generalized Hasimoto transformation, 
which transformed the equation into the following differential-integral equation 
\begin{align}
&\sqrt{-1}q_t
-
\alpha\left(
q_{xx}+\frac{\kappa(u)}{2}|q|^2q
\right)
\nonumber
\\
&
+\beta
\biggl\{
q_{xxxx}
-\frac{\kappa(u)}{2}
\left(
2|q_x|^2q-2|q|^2q_{xx}-\overline{q}q_x^2
\right)
-\frac{(\kappa(u))_x}{2}
\left(
q^2\overline{q}_x-|q|^2q_x
\right)
\biggl\}
\nonumber
\\
&
-\gamma
\biggl[
3(\kappa(u))^2|q|^4q
+\left\{
3(\kappa(u))_x|q|^2q+4\kappa(u)(|q|^2q)_x
\right\}_x
\biggl]
-qW(t,x)=0
\label{eq:DZ}
\end{align}
for complex-valued function $q=q(t,x)$, 
where  
$$
W(t,x)=\frac{1}{2}\int_x^{\infty}
(\kappa(u))_{\tilde{x}}\left(
\alpha|q|^2
-\beta
(q\overline{q}_{\tilde{x}\tilde{x}}+\overline{q}q_{\tilde{x}\tilde{x}}
-q_{\tilde{x}}\overline{q}_{\tilde{x}}
) 
+6\gamma \kappa(u)|q|^4
\right)
d\tilde{x}
$$
and $(\kappa(u))(t,x):=\kappa(u(t,x))$ is the Gaussian curvature at $u(t,x)\in N$.  
Notably,  
the expression is informative for our equation \eqref{eq:pde} with $c=3(a-b)/2$
only if $\kappa(u)$ is constant and hence the non-local term vanishes. 
This is because \eqref{eq:pde} is not the same as 
\eqref{eq:bibi} in general unless $(N,J,h)$ is a locally Hermitian symmetric space. 
Additionally, the method of transformation
has been developed to investigate 
\eqref{eq:4ono} (not \eqref{eq:pde}) for closed curve flows on compact Riemann surfaces 
by Chihara (\cite{chihara2}), which gives essential insights on the 
structure of equations satisfied by higher-order $x$-directional 
covariant derivatives of a solution to \eqref{eq:4ono}. 
The obtained results are valid to our equation \eqref{eq:pde} only if the compact 
Riemann surface $(N,J,h)$ has a constant Gaussian curvature.      
\subsection{Main results in this paper}
\label{subsection:main}
The main results in this paper are the following two. 
\par
First, for general compact K\"ahler manifold $(N,J,h)$ 
of complex dimension $n\in \mathbb{N}$, 
we present a framework that can transform \eqref{eq:pde} 
for $u=u(t,x):(-T,T)\times \RR\to N$
into an $n$-component system of fourth-order nonlinear dispersive partial differential-integral 
equations for complex-valued functions. 
To state our results precisely, let $u^{\infty}$ be a fixed point in $N$, 
and let $C_{u^{\infty}}((-T,T)\times \RR; N)$ denote the set of smooth maps 
$u=u(t,x):(-T,T)\times \RR\to N$ such that 
$\displaystyle\lim_{x\to -\infty}u(t,x)=u^{\infty}$ 
and $u_x(t,\cdot):\RR\to (u(t,\cdot))^{-1}TN$ is in the Schwartz class for any $t\in (-T,T)$.  
(The assumption on the set of maps is the same as that in \cite{TU}.) 
\begin{theorem}
\label{theorem:momo}
Under the above setting, 
\eqref{eq:pde} for  $u\in C_{u^{\infty}}((-T,T)\times \RR; N)$ 
is transformed into the system 
for $Q_1,\ldots,Q_n: (-T,T)\times \RR\to \mathbb{C}$ of the form  
\begin{align}
&\sqrt{-1}\p_tQ_j+(a\p_x^4+\lambda \p_x^2)Q_j
\nonumber
\\
&=
   d_1\sum_{p,q,r=1}^n
    S_{p,q,r}^{j}
    \p_x^2Q_p\overline{Q_q}Q_r
       +d_2
            \sum_{p,q,r=1}^n
        S_{p,q,r}^{j}
        Q_{p}\overline{\p_x^2Q_q}Q_r
    +d_3\sum_{p,q,r=1}^n
        S_{p,q,r}^{j}
        \p_xQ_p\overline{\p_xQ_q}Q_r
        \nonumber
        \\
    &\quad
       +d_4
        \sum_{p,q,r=1}^n
        S_{p,q,r}^{j}
        \p_xQ_p\overline{Q_q}\p_xQ_r
        +d_5
            \sum_{p,q,r=1}^n
            \p_x(
            S_{p,q,r}^{j})
            \p_xQ_p\overline{Q_q}Q_r
          \nonumber
          \\
          &\quad
            +d_6
                \sum_{p,q,r=1}^n
                \p_x(
            S_{p,q,r}^{j})
            Q_{p}\overline{\p_xQ_q}Q_r
    -\lambda
    \sum_{p,q,r=1}^n
    S_{p,q,r}^{j}Q_p\overline{Q_q}Q_r
\nonumber 
\\
&\quad 
+\sum_{r=1}^{n}
\left(\int_{-\infty}^x
f_{j,r}^1(Q,\p_xQ)(t,y)dy
\right)
Q_r
+\sum_{r=1}^{n}
\left(\int_{-\infty}^x
f_{j,r}^2(Q,\p_xQ)(t,y)dy
\right)
Q_r
\label{eq:r51251}
\end{align}
for each $j\in \left\{1,\ldots,n\right\}$.  
Here,  
$d_1=-a-b-2c$, 
$d_2=-a+b$, 
$d_3=a+b-2c$, 
$d_4=-b-2c$, 
$d_5=a-b-2c$, 
$d_6=a+b$, 
and $S_{p,q,r}^j$ for $p,q,r,j\in \left\{1,\ldots,n\right\}$ 
are complex-valued functions depending on $u$ 
which is defined by \eqref{eq:711}, and  
\begin{align}
f_{j,r}^1(Q,\p_xQ)
&=
-(b+2c)
\sum_{p,q,\alpha,\beta,\gamma=1}^n
S_{p,q,r}^{j}\overline{S_{\alpha,\beta,\gamma}^q}
\overline{\p_xQ_{\alpha}}Q_{\beta}
\overline{Q_{\gamma}}Q_p
\nonumber
\\
&\quad
-(b+2c)
\sum_{p,q,\alpha,\beta,\gamma=1}^n
S_{p,q,r}^{j}S_{\alpha,\beta,\gamma}^p
\p_xQ_{\alpha}\overline{Q_{\beta}}
Q_{\gamma}\overline{Q_q}
\nonumber
\\
&\quad
+b
\sum_{p,q,\alpha,\beta,\gamma=1}^n
          S_{p,q,r}^{j}\overline{S_{\alpha,\beta,\gamma}^q}
          \overline{Q_{\alpha}}\p_xQ_{\beta}
          \overline{Q_{\gamma}}Q_p
\nonumber
\\
&\quad
+b
\sum_{p,q,\alpha,\beta,\gamma=1}^n
          S_{p,q,r}^{j}S_{\alpha,\beta,\gamma}^p
          Q_{\alpha}\overline{\p_xQ_{\beta}}
          Q_{\gamma}\overline{Q_q}, 
\label{eq:r6201}
\\
f_{j,r}^2(Q,\p_xQ)
&=
-a
\sum_{p,q=1}^n
  \p_x^2(S_{p,q,r}^{j}) 
  (\p_xQ_p\overline{Q_q}+Q_p\overline{\p_xQ_q})
\nonumber
\\
&\quad 
-3a
\sum_{p,q=1}^n 
  \p_x(S_{p,q,r}^{j})
  \p_xQ_p\overline{\p_xQ_q}
+\lambda
 \sum_{p,q=1}^n  
 \p_x(S_{p,q,r}^{j})
 Q_p\overline{Q_q}. 
\label{eq:r6202}
\end{align}
for each $j,r\in \left\{1,\ldots,n\right\}$. 
\end{theorem}
The proof of Theorem~\ref{theorem:momo} 
is based on the generalized Hasimoto transformation, 
mainly following the argument of \cite{CSU} and \cite{RRS} 
to handle general compact K\"ahler manifolds as $(N,J,h)$ 
and to work under the assumption of the 
existence of the fixed edge point $u^{\infty}\in N$.
Importantly, investigating higher-order geometric PDEs such as \eqref{eq:pde} 
in this framework involves basic properties of the Riemann curvature tensor 
$R$ in relation with the parallel (in $x$-direction) 
orthonormal frame along $u$.  
This paper develops an understanding of these properties 
by introducing $S_{p,q,r}^j$ 
for $p,q,r,j\in \left\{1,\ldots,n\right\}$, 
and by deriving useful properties among them, 
which is a key ingredient of our proof.
These functions are introduced by \eqref{eq:711}
and their properties are gathered in Section~\ref{subsection:frame}. 
Propositions~\ref{proposition:Rloc}, \ref{prop:tensor}, and 
\ref{proposition:R4} are inherited from basic properties of $R$ on a 
general K\"ahler manifold stated in Proposition~\ref{prop:fact_R}, 
and using them enables us to arrive at the expression \eqref{eq:r51251} 
comprehensively. 
These properties established in Section~\ref{subsection:frame} 
are independent of the equation \eqref{eq:pde}, and thus
may be applicable in future for investigating other geometric PDEs. 
We should mention that our framework heavily relies on the K\"ahlerity of $N$ 
in order to construct the parallel moving frame.  
\par
Second, aiming at checking that \eqref{eq:r51251} with \eqref{eq:11021} 
actually unifies \eqref{eq:mns} and \eqref{eq:DZ} obtained 
previously by \cite{DW2018,DZ2021}, 
we take the following three examples of compact K\"ahler manifolds as $(N,J,h)$, 
and demonstrate the computations of   
\eqref{eq:r51251}: 
\begin{enumerate}
\item[(1)] compact Riemann surface
\item[(2)] compact K\"ahler manifold with constant holomorphic sectional curvature
\item[(3)] compact complex Grassmannian $G_{n_0,k_0}$ 
\end{enumerate}
The results of the computations for these examples are respectively 
presented as 
\eqref{eq:ns51261}, 
\eqref{eq:512132}, 
and 
\eqref{eq:Smomo3}. 
What we find from the results is outlined in 
the following paragraphs just below.  
\par 
If $(N,J,h)$ is a compact Riemann surface with constant Gaussian curvature, 
then the computation of example (1) shows that 
\eqref{eq:r51251} for $Q=Q_1$ under \eqref{eq:11021} 
coincides with \eqref{eq:DZ} for $q$ where $Q_1=q$. 
Basically, no originality is claimed here because the orthonormal frame 
$\left\{e,Je\right\}$ we use is essentially the same as that used in \cite{DZ2021}.  
If we were to add something, we note that 
a difference between them can be observed explicitly  
when the Gaussian curvature of $N$ is not a constant. 
See Remark~\ref{remark:rs} for the detail. 
\par 
If $(N,J,h)$ is an $n$-dimensional compact K\"ahler manifold with constant 
holomorphic sectional curvature $K$ and if $n\geqslant 2$, 
the setting of which has no been presented so far,  
then the computation of example (2) shows that    
nonlocal terms of integral type actually remain in the derived system 
of \eqref{eq:512132} unless $bK^2=0$. 
See Remark~\ref{remark:0802} also. 
Moreover, we point out that the computation of examples (2) is not so difficult.  
This is because the well-known formula \eqref{eq:nawa2} 
is available for associating $R$ explicitly with $h$ and the parallel orthonormal frame 
we use, the same of which is true for example (1).  
Notably, the typical example of $N$ here is   
the $n$-dimensional complex projective space 
with the Fubini-Study metric. 
This example can be handled also in the framework of example (3), 
since it is identified with $G_{n+1,1}$. 
However, the computation of example (2) is worth demonstrating, 
because it is rather easier and faster compared with that of example (3).
\par 
If $(N,J,h)$ is a compact complex Grassmannian, 
then the computation of example (3) shows the following. 
\begin{corollary}
\label{cor:coro} 
Let $(N,J,h)$ be the compact complex Grassmannian $G_{n_0,k_0}$ of 
complex dimension $n=k_0(n_0-k_0)$ as a 
Hermitian symmetric space. 
Then, under the same setting as in Theorem~\ref{theorem:momo} with 
additional setting \eqref{eq:11021}, 
the system of equations 
\eqref{eq:r51251} for $Q_1\ldots,Q_n$ 
coincides with \eqref{eq:mns} for $q$ up to a gauge transformation, 
where  the $(i,j)$-component of $q$ 
for $i\in \left\{1,\ldots, k_0\right\}$ 
and $j\in \left\{1,\ldots,n_0-k_0\right\}$ 
is identified with $Q_{(j-1)k_0+i}$.
\end{corollary}
This verifies that the first type of method based on the development map and 
the second one based on the parallel orthonormal frame 
are essentially same in this setting.  
We show Corollary~\ref{cor:coro} by a direct computation of \eqref{eq:r51251}
in Section~\ref{subsection:comGr}, 
Interestingly, the computation of  example (3)  is more challenging than 
those of examples (1) and (2) 
especially in the case where 
$G_{n_0,k_0}$ is a higher-rank symmetric space with $\min(k_0,n_0-k_0)>1$.
(It is known that $G_{n_0,k_0}$ in the case does not have a constant 
holomorphic sectional curvature, and thus does not fall into the scope of example (2).)
Although a formula \eqref{eq:Rgr} to express $R$ is available,  
it does not directly provide explicit relation between $R$ and $h$, 
which is not compatible with the parallel orthonormal frame we use.    
To overcome the difficulty, we construct a parallel orthonormal frame concretely
by taking a suitable orthogonal basis at a point in $G_{n_0,k_0}$, 
which makes \eqref{eq:Rgr} applicable for the computation.
Additionally, we also provide more theoretical explanation for the reason 
why seemingly different two types of methods lead to the same result for 
$G_{n_0,k_0}$ by comparing them in Section~\ref{subsection:revisit}. 
More precisely, we show Proposition~\ref{proposition:revisit}, 
which presents another proof of Corollary~\ref{cor:coro} 
without doing the computation in Section~\ref{subsection:comGr}.    
\par  
Finally, two additional comments on our results in this paper 
are in order.
\par  
First, recent studies on the initial value problem 
for \eqref{eq:pde} and \eqref{eq:4ono} 
(\cite{chihara2,CO2,GZS,onodera2,onodera3,onodera4,onodera5})
have handled the case where $(N,J,h)$ is 
any of the canonical $\mathbb{S}^2$, a compact Riemann surface, 
a compact locally Hermitian symmetric space, or more general compact 
K\"ahler manifold. 
Through the studies, some results on time-local existence of a 
unique solution in a Sobolev space with high regularity 
have been established. 
This motivated the present author to establish Theorem~\ref{theorem:momo} 
in order to understand the structure of the geometric equations in the level of 
systems of nonlinear equations for complex valued functions. 
We expect that the derived expression in Theorem~\ref{theorem:momo} 
will be informative in future work 
to establish fruitful results on   
the conditions for the solution to \eqref{eq:pde} 
to exist globally in time, or 
on the conditions for \eqref{eq:pde} to be completely integrable. 
\par 
Second, as is stated above, 
Ding and Wang investigated \eqref{eq:bibi} in \cite{DW2018}. 
In fact, they proposed some unclear questions in their paper, 
one of which is commented as follows: 
\begin{quotation}
Except the Hermitian symmetric spaces $G_{n,k}$ of compact 
type (i.e, A III), there are C I, D III and BD I-types of 
symmetric spaces (and also two exceptional Hermitian 
symmetric spaces E III and E VII) (refer to [23, 26]). 
Can we establish similar results on these symmetric spaces?
(\cite{DW2018}, p.~192)
\end{quotation}
Our Theorem~\ref{theorem:momo} handles 
general compact K\"ahler manifolds as $(N,J,h)$, 
including these symmetric spaces. 
Moreover, we expect that 
the computation of example (3) for $G_{n_0,k_0}$ 
in Section~\ref{subsection:comGr} 
can be proceeded in similar way also for these other symmetric spaces, 
only by modifying the setting in Section~\ref{subsection:PreGr} 
suitably. 
If this is true, then the result of the computation will provide a 
partial answer to the above proposed question, 
although the approach may not be what was intended 
by the authors in \cite{DW2018} and the way to construct 
the inverse of the transformation which verifies the equivalence  
is still unclear in this approach.   
\par 
The organization of this paper is as follows: 
In Section~\ref{section:reduction}, 
the parallel orthonormal frame is constructed and associated basic properties are provided  
in Section~\ref{subsection:frame}, 
and then Theorem~\ref{theorem:momo} is proved 
in Section~\ref{subsection:reduction}.
In Section~\ref{section:Examples}, the computations of 
\eqref{eq:r51251} for examples (1) and (2) of $(N,J,h)$ 
are demonstrated.  
In Section~\ref{section:Example3}, 
the setting of $G_{n_0,k_0}$ are reviewed in Section~\ref{subsection:PreGr}, 
and then the computation of \eqref{eq:r51251} for example (3) of $(N,J,h)$ 
is demonstrated to show Corollary~\ref{cor:coro}
in Section~\ref{subsection:comGr}, 
with additional explanation in Section~\ref{subsection:revisit}. 
Supplemental materials are stated in Appendix.
%
%
%
\section{Reduction to a system of PDEs}
\label{section:reduction}
In this section, 
suppose that 
$(N,J,h)$ is a compact K\"ahler manifold of complex dimension $n$ 
and $u\in C_{u^{\infty}}((-T,T)\times \RR;N)$ is a 
solution to \eqref{eq:pde}
as in Theorem~\ref{theorem:momo}.   
\subsection{Parallel moving frame and the associated properties 
of the Riemann curvature tensor}
\label{subsection:frame}  
For fixed $u^{\infty}\in N$,  
let $\left\{e^{\infty}_1, \ldots, e^{\infty}_n, 
J_{u^{\infty}}e^{\infty}_{1},\ldots, J_{u^{\infty}}e^{\infty}_{n}\right\}$ 
be an orthonormal basis
for $T_{u^{\infty}}N$ with respect to $h$. 
Following \cite{CSU} and \cite{RRS}, 
we take the orthonormal frame 
$\left\{e_1, \ldots, e_n, e_{n+1},\ldots, e_{2n}\right\}$
for $u^{-1}TN$ that satisfies 
\begin{align}
\nabla_xe_p(t,x)&=0,  
\label{eq:mf01} 
\\
\lim_{x\to -\infty}e_p(t,x)&=e_p^{\infty}, 
\label{eq:mfini}
\end{align} 
and $e_{p+n}=J_ue_p$ for any $p\in \left\{1,\ldots,n\right\}$. 
The K\"ahlerity $\nabla J=0$ ensures 
$\nabla_xe_{p+n}=0$ for any $p\in \left\{1,\ldots,n\right\}$. 
(In what follows, 
$\left\{e_1, \ldots, e_n, e_{n+1},\ldots, e_{2n}\right\}$
is written by $\left\{e_j,Je_j\right\}_{j=1}^n$ for simplicity.) 
In the setting, we can write 
 \begin{align}
 \nabla_te_p
 &=
 \sum_{j=1}^{n}a_j^pe_j
 +
 \sum_{j=1}^{n}b_j^pJ_ue_j
 \quad 
 (\forall p\in \left\{1,\ldots,n\right\}), 
 \label{eq:tep}
 \end{align}
where for each $p,j\in \left\{1,\ldots,n\right\}$, 
$a_j^p=a_j^p(t,x)$ and 
$b_j^p=b_j^p(t,x)$ 
denote real-valued functions of $(t,x)$.   
We will seek conditions on $a_j^p$ and $b_j^p$ later. 
(See \eqref{eq:s13R1}.) 
The components of $R$ 
associated with $\left\{e_j,Je_j\right\}_{j=1}^n$
are expressed by the following:
\begin{align}
R(e_p,e_{q})e_r
&=
\sum_{j=1}^n
R_{p,q,r}^je_j
+
\sum_{j=1}^n
R_{p,q,r}^{j+n}J_ue_j, 
\label{eq:61}
\\
R(e_p,J_ue_{q})e_r
&=
\sum_{j=1}^n
R_{p,q+n,r}^je_j
+
\sum_{j=1}^n
R_{p,q+n,r}^{j+n}J_ue_j
\label{eq:62}
\end{align}
for $p,q,r\in \left\{1,\ldots,n\right\}$, 
where
$R_{p,q,r}^j, R_{p,q,r}^{j+n},R_{p,q+n,r}^j,R_{p,q+n,r}^{j+n}$
are real-valued functions of $(t,x)$. 
Furthermore, let us set 
\begin{align}
R_{p,q,r}^{A,j}
&:=
R_{p,q,r}^j
+
\sqrt{-1}\,R_{p,q,r}^{j+n},
\quad 
R_{p,q,r}^{B,j}
:=
R_{p,q+n,r}^j
+
\sqrt{-1}\,R_{p,q+n,r}^{j+n}, 
\label{eq:71}
\\
S_{p,q,r}^j&:=\frac{1}{2}\left(R_{p,q,r}^{A,j}+\sqrt{-1}R_{p,q,r}^{B,j}\right), 
\quad
T_{p,q,r}^j:=\frac{1}{2}\left(
-R_{p,q,r}^{A,j}+\sqrt{-1}R_{p,q,r}^{B,j}
\right)
\label{eq:711}
\end{align}
for $p,q,r,j\in \left\{1,\ldots,n\right\}$, 
all of which are then complex-valued functions of $(t,x)$. 
In fact, these functions at $(t,x)$ depend on $u(t,x)$.  
In particular, by definition, it follows that 
\begin{align}
S_{p,q,r}^j
&=
\dfrac{1}{2}
\left\{
h(R(e_p,e_q)e_r,e_j)
+\sqrt{-1}h(R(e_p,e_q)e_r,J_ue_j)
\right\}
\nonumber
\\
&\quad
+
\dfrac{\sqrt{-1}}{2}
\left\{
h(R(e_p,J_ue_q)e_r,e_j)
+\sqrt{-1}h(R(e_p,J_ue_q)e_r,J_ue_j)
\right\}. 
\label{eq:add11171}
\end{align}
\par 
Here, we recall the following basic properties for $R$ 
on the K\"ahler manifold $(N,J,h)$:  
\begin{proposition}
\label{prop:fact_R}
For any $Y_1, \ldots, Y_4\in \Gamma(u^{-1}TN)$, 
the following properties hold: 
\begin{enumerate}
\item[(i)] $R(Y_1,Y_2)=-R(Y_2,Y_1)$,
\item[(ii)] $R(Y_1,Y_2)Y_3+R(Y_2,Y_3)Y_1+R(Y_3,Y_1)Y_2=0$,
\item[(iii)] $h(R(Y_1,Y_2)Y_3,Y_4)=h(R(Y_3,Y_4)Y_1,Y_2)=h(R(Y_4,Y_3)Y_2,Y_1)$,
\item[(iv)] $R(Y_1,Y_2)J_uY_3=J_u\,R(Y_1,Y_2)Y_3$, 
\item[(v)]
$R(J_uY_1,Y_2)Y_3=-R(Y_1,J_uY_2)Y_3$, \quad 
$R(J_uY_1,Y_2)Y_3=R(J_uY_2,Y_1)Y_3$.
\end{enumerate} 
\end{proposition}
The properties (i)-(iii) follow from the definition of $R$. 
The property (iv) holds since $(N,J,h)$ is a K\"ahler manifold.
The property (v) follows from (i) and $J_u^2=-\text{id}$ on $\Gamma(u^{-1}TN)$. 
\par 
The following properties for $R_{p,q,r}^{A,j}$ and $R_{p,q,r}^{B,j}$ 
are inherited from Proposition~\ref{prop:fact_R} for $R$: 
\begin{proposition}
\label{proposition:R1} 
The following properties hold: 
\begin{align}
&R_{p,q,r}^{A,j}=-R_{q,p,r}^{A,j}
\quad
(\forall p,q,r,j\in \left\{1,\ldots,n\right\}),
\label{eq:R1}
\\
&R_{p,q,r}^{B,j}=R_{q,p,r}^{B,j}
\quad
(\forall p,q,r,j\in \left\{1,\ldots,n\right\}), 
\label{eq:R2}
\\
&R_{p,q,r}^{A,j}+R_{q,r,p}^{A,j}+R_{r,p,q}^{A,j}=0
\quad
(\forall p,q,r,j\in \left\{1,\ldots,n\right\}), 
\label{eq:R3}
\\
&R_{p,q,r}^{A,j}
=\sqrt{-1}(R_{q,r,p}^{B,j}-R_{r,p,q}^{B,j})
\quad
(\forall p,q,r,j\in \left\{1,\ldots,n\right\}),
\label{eq:R4}
\\
&
\operatorname{Re}
[R_{p,q,r}^{A,j}]
=\operatorname{Re}
[R_{r,j,p}^{A,q}]
=\operatorname{Re}
[R_{j,r,q}^{A,p}]
\quad
(\forall p,q,r,j\in \left\{1,\ldots,n\right\}),
\label{eq:R6}
\\
&\operatorname{Im}
[R_{p,q,r}^{B,j}]
=\operatorname{Im}
[R_{r,j,p}^{B,q}]
=\operatorname{Im}
[R_{j,r,q}^{B,p}]
\quad
(\forall p,q,r,j\in \left\{1,\ldots,n\right\}),
\label{eq:R7}
\\
&\operatorname{Im}
[R_{p,q,r}^{A,j}]
=\operatorname{Re}
[R_{r,j,p}^{B,q}]
=-\operatorname{Re}
[R_{j,r,q}^{B,p}]
\quad
(\forall p,q,r,j\in \left\{1,\ldots,n\right\}).
\label{eq:R8}
\end{align}
\end{proposition}
In particular, if $\nabla R=0$ holds, that is, 
$(N,J,h)$ is locally Hermitian symmetric, 
the following also holds:   
\begin{proposition}
\label{proposition:Rloc}
Under the additional assumption $\nabla R=0$, 
the following also holds: 
\begin{align}
&\p_x R_{p,q,r}^{A,j}
=\p_x R_{p,q,r}^{B,j}=0
\quad
(\forall p,q,r,j\in \left\{1,\ldots,n\right\}).
\label{eq:R5}
\end{align}
\end{proposition}
\begin{proof}[Proof of Proposition~\ref{proposition:Rloc}]
Let $p,q,r\in \left\{1,\ldots,n\right\}$ be any given. 
From $\nabla R=0$ and  \eqref{eq:mf01},  
\begin{align}
&\nabla_x
\{
R(e_p,e_{q})e_r
\}
\nonumber
\\
&=(\nabla_xR)(e_p,e_{q})e_r
+R(\nabla_xe_p,e_{q})e_r
+R(e_p,\nabla_xe_{q})e_r
+R(e_p,e_{q})\nabla_xe_r
\nonumber
\\
&=0.
\nonumber
\end{align}
On the other hand, from \eqref{eq:61}, the K\"ahlerity 
$\nabla J=0$, and \eqref{eq:mf01},  
\begin{align}
\nabla_x
\{
R(e_p,e_{q})e_r
\}
&=
\sum_{j=1}^n
(\p_x R_{p,q,r}^j)e_j
+
\sum_{j=1}^n
(\p_xR_{p,q,r}^{j+n})J_ue_j.
\nonumber
\end{align}
Comparing both equations, and noting 
$\left\{e_j,Je_j\right\}_{j=1}^n$ is the orthonormal frame, 
we have
$$
0=\p_x R_{p,q,r}^j=\operatorname{Re}[\p_xR_{p,q,r}^{A,j}]
\quad
\text{and}
\quad 
0=\p_xR_{p,q,r}^{j+n}=\operatorname{Im}[\p_xR_{p,q,r}^{A,j}]$$
for any $j\in \left\{1,\ldots,n\right\}$,
which shows $\p_xR_{p,q,r}^{A,j}=0$ for any $p,q,r,j\in \left\{1,\ldots,n\right\}$. 
In the same way as above, we use \eqref{eq:62} and then obtain  
\begin{align}
0&=\nabla_x\left\{R(e_p,J_ue_{q})e_r\right\}
=
\sum_{j=1}^n
(\p_x R_{p,q+n,r}^j)e_j
+
\sum_{j=1}^n
(\p_xR_{p,q+n,r}^{j+n})J_ue_j.
\nonumber
\end{align}
From this, it follows that 
$$
0=\p_x R_{p,q+n,r}^j=\operatorname{Re}[\p_xR_{p,q,r}^{B,j}]
\quad 
\text{and}
\quad
0=
\p_xR_{p,q+n,r}^{j+n}=\operatorname{Im}[\p_xR_{p,q,r}^{B,j}],$$
and thus $\p_xR_{p,q,r}^{B,j}=0$ for any $p,q,r,j\in \left\{1,\ldots,n\right\}$.
\end{proof}
\begin{proof}[Proof of Proposition~\ref{proposition:R1}]
The property \eqref{eq:R1} follows from (i) in Proposition~\ref{prop:fact_R}: 
Indeed,  (i) for $(Y_1,Y_2)=(e_p,e_q)$ and \eqref{eq:61} shows
\begin{align}
0&=R(e_p,e_{q})e_r+R(e_{q},e_p)e_r
\nonumber 
\\
&
=\sum_{j=1}^n
(R_{p,q,r}^j+R_{q,p,r}^j)e_j
+
\sum_{j=1}^n
(R_{p,q,r}^{j+n}+R_{q,p,r}^{j+n})J_ue_j, 
\nonumber 
\end{align}
which implies 
\begin{align}
&R_{p,q,r}^j=-R_{q,p,r}^j,
\quad 
R_{p,q,r}^{j+n}=-R_{q,p,r}^{j+n}
\quad 
(\forall p,q,r,j\in \left\{1,\ldots,n\right\}).
\label{eq:35ab}
\end{align}
Recalling \eqref{eq:71}, 
we see \eqref{eq:35ab} is equivalent to \eqref{eq:R1}. 
\par 
The property \eqref{eq:R2} follows from (i) and the second part of (v)
in Proposition~\ref{prop:fact_R}: 
Indeed,  
$R(e_p,J_ue_{q})e_r=R(e_{q},J_ue_p)e_r$ follows from them. 
This combined with \eqref{eq:62} yields  
\begin{align}
R_{p,q+n,r}^j=R_{q,p+n,r}^j,
\quad 
R_{p,q+n,r}^{j+n}=R_{q,p+n,r}^{j+n}
\quad 
(\forall p,q,r,j\in \left\{1,\ldots,n\right\}).
\label{eq:35cd}
\end{align}
Recalling \eqref{eq:71}, we see \eqref{eq:35cd} is equivalent to \eqref{eq:R2}. 
\par 
The properties \eqref{eq:R3} and \eqref{eq:R4} 
follow from (ii), (iv) and (v)
in Proposition~\ref{prop:fact_R}: 
Indeed, (ii) for 
$(Y_1, Y_2,Y_3)=(e_p,e_q,e_r)$ shows 
$$
0=
\sum_{j=1}^n
(R_{p,q,r}^j+R_{q,r,p}^j+R_{r,p,q}^j)e_j
+
\sum_{j=1}^n
(R_{p,q,r}^{j+n}+R_{q,r,p}^{j+n}+R_{r,p,q}^{j+n})J_ue_j, 
$$
which implies 
\begin{align}
&R_{p,q,r}^j+R_{q,r,p}^j+R_{r,p,q}^j=0, 
\ 
R_{p,q,r}^{j+n}+R_{q,r,p}^{j+n}+R_{r,p,q}^{j+n}=0
\quad 
(\forall p,q,r,j\in \left\{1,\ldots,n\right\}).
\label{eq:3158}
\end{align}
Taking the summation of the first equality of \eqref{eq:3158} 
and the second one multiplied by $\sqrt{-1}$, 
we obtain \eqref{eq:R3}. 
On the other hand, using (ii) for 
$(Y_1,Y_2,Y_3)=(e_p, e_{q}, J_ue_r)$, (iv) and the first part of (v), 
we deduce 
\begin{align}
0&=R(e_p,e_{q})J_ue_r
+R(e_{q},J_ue_r)e_p+R(J_ue_r,e_p)e_{q}
\nonumber
\\
&=
J_uR(e_p,e_{q})e_r
+R(e_{q},J_ue_r)e_p-R(e_r,J_ue_p)e_{q}
\nonumber
\\
&=\sum_{j=1}^n
J_u(R_{p,q,r}^je_j
+
R_{p,q,r}^{j+n}J_ue_j)
+
\sum_{j=1}^n
(R_{q,r+n,p}^je_j
+
R_{q,r+n,p}^{j+n}J_ue_j)
\nonumber
\\
&\quad 
-\sum_{j=1}^n
(R_{r,p+n,q}^je_j
+
R_{r,p+n,q}^{j+n}J_ue_j)
\nonumber
\\
&=
\sum_{j=1}^{n}
\left\{
(-R_{p,q,r}^{j+n}+R_{q,r+n,p}^j-R_{r,p+n,q}^j)
+
(R_{p,q,r}^j+R_{q,r+n,p}^{j+n}-R_{r,p+n,q}^{j+n})J_u
\right\}
e_j.
\nonumber
\end{align}
From this, for any $p,q,r,j\in \left\{1,\ldots,n\right\}$, it follows that 
\begin{align}
&-R_{p,q,r}^{j+n}+R_{q,r+n,p}^j-R_{r,p+n,q}^j=0, 
\quad 
R_{p,q,r}^j+R_{q,r+n,p}^{j+n}-R_{r,p+n,q}^{j+n}=0. 
\label{eq:3112}
\end{align}
Two equalities in \eqref{eq:3112} can be written respectively as follows: 
\begin{align}
&\operatorname{Im}[R_{p,q,r}^{A,j}]
=
\operatorname{Re}[R_{q,r,p}^{B,j}]
-\operatorname{Re}[R_{r,p,q}^{B,j}], 
\label{eq:31b}
\\
&\operatorname{Re}[R_{p,q,r}^{A,j}]
=
-\operatorname{Im}[R_{q,r,p}^{B,j}]
+\operatorname{Im}[R_{r,p,q}^{B,j}]. 
\label{eq:31c}
\end{align}
Using \eqref{eq:31b} and \eqref{eq:31c}, 
we obtain the desired \eqref{eq:R4} as follows:
\begin{align}
R_{p,q,r}^{A,j}
&=
-\operatorname{Im}[R_{q,r,p}^{B,j}]
+\operatorname{Im}[R_{r,p,q}^{B,j}]
+
\sqrt{-1}(\operatorname{Re}[R_{q,r,p}^{B,j}]
-\operatorname{Re}[R_{r,p,q}^{B,j}])
\nonumber
\\
&=
\sqrt{-1}\left\{
(\operatorname{Re}[R_{q,r,p}^{B,j}]+\sqrt{-1}\operatorname{Im}[R_{q,r,p}^{B,j}])
-(\operatorname{Re}[R_{r,p,q}^{B,j}]+\sqrt{-1}\operatorname{Im}[R_{r,p,q}^{B,j}])
\right\}
\nonumber
\\
&=
\sqrt{-1}(
R_{q,r,p}^{B,j}-R_{r,p,q}^{B,j}
).
\nonumber
\end{align}
\par
The properties \eqref{eq:R6}-\eqref{eq:R8} follow from 
(iii) and (iv) in Proposition~\ref{prop:fact_R}: 
First, noting $R_{p,q,r}^j=h(R(e_p,e_q)e_r,e_j)$ and using (iii) for $(Y_1,Y_2,Y_3,Y_4)=(e_p,e_q,e_r,e_j)$, 
we obtain
\begin{align}
&R_{p,q,r}^j
=R_{r,j,p}^{q}
=R_{j,r,q}^p
\quad
(\forall p,q,r,j\in \left\{1,\ldots,n\right\}), 
\label{eq:3321}
\end{align}
which by \eqref{eq:71} is equivalent to \eqref{eq:R6}. 
Second, noting 
$R_{p,q+n,r}^{j+n}=h(R(e_p,J_ue_q)e_r,J_ue_j)$ and using (iii), 
we obtain 
$$
R_{p,q+n,r}^{j+n}
=h(R(e_r,J_ue_j)e_p,J_ue_q)=h(R(J_ue_j,e_r)J_ue_q,e_p).
$$ 
Here, $h(R(e_r,J_ue_j)e_p,J_ue_q)=R_{r,j+n,p}^{q+n}$ follows from 
\eqref{eq:62}, and 
$$
h(R(J_ue_j,e_r)J_ue_{q},e_p)=h(R(e_j,J_ue_r)e_{q},J_ue_p)
=R_{j,r+n,q}^{p+n}
$$
follows from 
 (iii), the first part of (v), and \eqref{eq:62}.
Combining them, we obtain 
\begin{align}
&R_{p,q+n,r}^{j+n}
=R_{r,j+n,p}^{q+n}
=R_{j,r+n,q}^{p+n}
\quad
(\forall p,q,r,j\in \left\{1,\ldots,n\right\}), 
\label{eq:3324}
\end{align}
which is equivalent to \eqref{eq:R7}.    
Third, noting $R_{p,q,r}^{j+n}=h(R(e_p,e_q)e_r, J_ue_j)$ 
and using (iii) for $(Y_1,Y_2,Y_3,Y_4)=(e_p,e_q,e_r,J_ue_j)$, 
we have
$$
R_{p,q,r}^{j+n}
=h(R(e_r,J_ue_j)e_p, e_q)
=h(R(J_ue_j,e_r)e_q, e_p).
$$
Here, $h(R(e_r,J_ue_j)e_p, e_q)=R_{r,j+n,p}^q$ follows from \eqref{eq:62}, 
and 
$$h(R(J_ue_j,e_r)e_{q},e_p)=-h(R(e_j,J_ue_r)e_{q},e_p)=-R_{j,r+n,q}^p$$
follows from the first part of (v) and \eqref{eq:62}. 
Combining them, we obtain 
\begin{align}
&R_{p,q,r}^{j+n}
=R_{r,j+n,p}^{q}
=-R_{j,r+n,q}^p
\quad
(\forall p,q,r,j\in \left\{1,\ldots,n\right\}), 
\label{eq:3322}
\end{align}
which is equivalent to \eqref{eq:R8}. 
 \end{proof}
 \begin{remark} 
 In the above proof, 
 the conditions \eqref{eq:R3} and \eqref{eq:R4} 
 are obtained by using (ii) in Proposition~\ref{prop:fact_R} 
 for $(Y_1,Y_2, Y_3)=(e_p,e_q, e_r)$ and $(e_p,e_q,J_ue_r)$.  
 Additionally, no other conditions can be obtained from (ii) 
 even if we choose 
 $(e_p,J_ue_q, e_r)$, $(e_p,J_ue_q, J_ue_r)$
 $(J_ue_p,e_q, e_r)$,  $(J_ue_p,e_q, J_ue_r)$, 
 $(J_ue_p,J_ue_q, e_r)$, or $(J_ue_p,J_ue_q,J_u e_r)$ 
 as $(Y_1,Y_2,Y_3)$, 
 which is due to the properties of $R$ stated above. 
 \end{remark}
\begin{remark} 
In the above proof, the conditions \eqref{eq:R6}-\eqref{eq:R8} are 
obtained by investigating  
\begin{equation}
h(R(K(p)e_p, K(q)e_q)K(r)e_r, K(j)e_j)
\label{eq:236152}
\end{equation}
where $(K(p), K(q), K(r),K(j))=(I_d,I_d,I_d,I_d)$, 
$(I_d,J_u,I_d,J_u)$
or $(I_d,I_d, I_d,J_u)$. 
Although there are seemingly $2^4$-types of expression of \eqref{eq:236152} 
depending on the choice of $I_d$ or $J_u$ as $K(\cdot)$, 
no other conditions can be obtained even if we 
investigate the rest $(2^4-3)$-types: 
The most curious may be the case where $(K(p), K(q), K(r),K(j))=(I_d,J_u,I_d,I_d)$, 
in that $R_{p,q+n,r}^{j}=h(R(e_p,J_ue_q)e_r, e_j)$ appearing in \eqref{eq:62} is not 
investigated in this context. 
As for the case, we use (iii) in the same way as above to see  
$$
R_{p,q+n,r}^{j}
=h(R(e_r, e_j)e_p,J_ue_q)
=h(R(e_j,e_r)J_ue_q, e_p). 
$$
Here, $h(R(e_r, e_j)e_p,J_ue_q)=R_{r,j,p}^{q+n}$ follows from \eqref{eq:61}, 
and 
$$h(R(e_j,e_r)J_ue_{q},e_p)=-h(R(e_j,e_r)e_{q},J_ue_p)=-R_{j,r,q}^{p+n}$$
follows from (iii) combined with the first part of (v) and \eqref{eq:61}. 
Combining them, we get  
\begin{align}
&R_{p,q+n,r}^{j}
=R_{r,j,p}^{q+n}
\quad 
\text{and}
\quad 
R_{p,q+n,r}^{j}
=-R_{j,r,q}^{p+n}
\quad
(\forall p,q,r,j\in \left\{1,\ldots,n\right\}).  
\label{eq:3323}
\end{align}
However, by replacing indexes 
$(p,q,r,j)\to (r,j,p,q)$ for the first part of \eqref{eq:3323} and 
$(p,q,r,j)\to (j,r,q,p)$ for the second part, 
\eqref{eq:3323} becomes 
$$
R_{r,j+n,p}^{q}
=R_{p,q,r}^{j+n}, 
\quad 
\text{and}
\quad
R_{j,r+n,q}^{p}
=-R_{p,q\,r}^{j+n}
\quad
(\forall p,q,r,j\in \left\{1,\ldots,n\right\}), 
$$
which is nothing but \eqref{eq:3322}. 
Hence, \eqref{eq:3323} is now meaningless. 
In the same way, it turns out that 
the conditions obtained from
the rest $(2^4-4)$-expressions of \eqref{eq:236152}
are reduced to either of 
\eqref{eq:3321}, \eqref{eq:3324}, \eqref{eq:3322}, and \eqref{eq:3323}. 
We omit the detail. 
\end{remark}
\par 
Next, let us see the functions $S_{p,q,r}^j$ play the crucial role to express $R$. 
\begin{proposition}
\label{prop:tensor}
Under the same assumption as that in Proposition~\ref{proposition:R1},  
\begin{align}
\langle
R(U,V)W
\rangle_j
&=
\sum_{p,q,r=1}^n
	  S_{p,q,r}^{j}\left(
	  \langle U \rangle_p\overline{\langle V \rangle_q}
	 -\langle V \rangle_p\overline{\langle U \rangle_q}
	  \right)
	  \langle W \rangle_r
	  \quad
	  (\forall j\in \left\{1,\ldots,n\right\})
	  \label{eq:23613}	 
\end{align}
holds for any $U,V,W\in \Gamma(u^{-1}TN)$. 
Here, for any $\Xi\in \Gamma(u^{-1}TN)$ and $j\in \left\{1,\ldots,n\right\}$, 
$\langle\Xi\rangle_j$ 
denotes a complex-valued function defined by 
 \begin{equation}
\langle\Xi\rangle_j:=h(\Xi,e_j)+\sqrt{-1}\,h(\Xi,J_ue_j).
\label{eq:momo}
 \end{equation}
\end{proposition}
If we write  
$\Xi=\displaystyle\sum_{k=1}^n
   (\Xi_k^R+J_u\Xi_k^I)e_k$ 
for $\Xi\in \Gamma(u^{-1}TN)$ by using real-valued functions 
$\Xi_k^R$ and $\Xi_k^I$ and substitute it into \eqref{eq:momo},   
then we see
\begin{align}
\langle\Xi\rangle_j
&=
\Xi_j^R+\sqrt{-1}\Xi_j^I. 
\label{eq:momo2}
\end{align}
\begin{proof}[Proof of Proposition~\ref{prop:tensor}]
 To begin with, let us write $U,V,W\in \Gamma(u^{-1}TN)$ as 
 \begin{align}
  U&=
  \sum_{p=1}^n
  (U_p^R+J_uU_p^I)e_p, 
  V=\sum_{q=1}^n
    (V_{q}^R+J_uV_{q}^I)e_{q}, 
  W=\sum_{r=1}^n
      (W_r^R+J_uW_r^I)e_r, 
  \label{eq:51}
  \end{align}
where $U_p^R, U_p^I$, $V_q^R, V_q^I$, $W_r^R, W_r^I$ for all $p,q,r\in \left\{1,\ldots,n\right\}$ are real-valued functions of $(t,x)$.  
As $R$ is trilinear, 
\begin{align}
&R(U,V)W
\nonumber
\\
&=
\sum_{p, q,r=1}^n
R(U_p^Re_p+J_uU_p^Ie_p, V_{q}^Re_{q}+J_uV_{q}^Ie_{q}) (W_r^Re_r+J_uW_r^Ie_r)
\nonumber
\\
&=\sum_{p, q,r=1}^n
\left\{
	  \begin{tabular}{l}
	    $U_p^RV_{q}^RW_r^RR(e_p,e_{q})e_r
	    +U_p^RV_{q}^RW_r^IR(e_p,e_{q})J_ue_r$
	    \\
	    $+U_p^RV_{q}^IW_r^RR(e_p,J_ue_{q})e_r
	    +U_p^RV_{q}^IW_r^IR(e_p,J_ue_{q})J_ue_r$
	    \\
	    $+U_p^IV_{q}^RW_r^RR(J_ue_p,e_{q})e_r
	    +U_p^IV_{q}^RW_r^IR(J_ue_p,e_{q})J_ue_r$
	    \\
	    $+U_p^IV_{q}^IW_r^RR(J_ue_p,J_ue_{q})e_r
	    +U_p^IV_{q}^IW_r^IR(J_ue_p,J_ue_{q})J_ue_r$
	  \end{tabular}
	  \right\}.
\nonumber
\end{align}
By (iv) and the first part of (v) in Proposition~\ref{prop:fact_R}, this becomes 
\begin{align}
R(U,V)W
&=
\sum_{p, q,r=1}^n
\left\{
	  \begin{tabular}{l}
	   $(U_p^RV_{q}^I-U_p^IV_{q}^R)W_r^RR(e_p,J_ue_{q})e_r$
	   \\
	   $+(U_p^RV_{q}^I-U_p^IV_{q}^R)W_r^IJ_uR(e_p,J_ue_{q})e_r$
	   \\
	   $+(U_p^RV_{q}^R+U_p^IV_{q}^I)W_r^RR(e_p,e_{q})e_r$
	   	   \\
	   	   $+(U_p^RV_{q}^R+U_p^IV_{q}^I)W_r^IJ_uR(e_p,e_{q})e_r$
	  \end{tabular}
	  \right\}.
\nonumber
\end{align}
Substitution of \eqref{eq:61} and \eqref{eq:62} into the above yields  
\begin{align}
&R(U,V)W
\nonumber
\\
&=
\sum_{p,q,r,k=1}^n
(U_p^RV_{q}^I-U_p^IV_{q}^R)(W_r^R+J_uW_r^I)
(R_{p,q+n,r}^k+J_uR_{p,q+n,r}^{k+n})e_k
\nonumber
\\
&\quad 
+
\sum_{p,q,r,k=1}^n
(U_p^RV_{q}^R+U_p^IV_{q}^I)(W_r^R+J_uW_r^I)
(R_{p,q,r}^k+J_uR_{p,q,r}^{k+n})e_k.
\nonumber
\end{align}
By \eqref{eq:momo2} and \eqref{eq:71}, 
this expression yields
\begin{align}
&\langle
R(U,V)W
\rangle_j
\nonumber
\\
&=
\sum_{p,q,r=1}^n
(U_p^RV_{q}^I-U_p^IV_{q}^R)
\langle
\sum_{k=1}^n
(W_r^R+J_uW_r^I)
(R_{p,q+n,r}^k+J_uR_{p,q+n,r}^{k+n})e_k
\rangle_j
\nonumber
\\
&\quad 
+
\sum_{p,q,r=1}^n
(U_p^RV_{q}^R+U_p^IV_{q}^I)
\langle
\sum_{k=1}^n
(W_r^R+J_uW_r^I)
(R_{p,q,r}^k+J_uR_{p,q,r}^{k+n})e_k
\rangle_j
\nonumber
\\
&=
\sum_{p,q,r=1}^n
(U_p^RV_{q}^I-U_p^IV_{q}^R)
(W_r^R+\sqrt{-1}W_r^I)
(R_{p,q+n,r}^j+\sqrt{-1}R_{p,q+n,r}^{j+n})
\nonumber
\\
&\quad 
+
\sum_{p,q,r=1}^n
(U_p^RV_{q}^R+U_p^IV_{q}^I)
(W_r^R+\sqrt{-1}W_r^I)
(R_{p,q,r}^j+\sqrt{-1}R_{p,q,r}^{j+n})
\nonumber
\\
&=
\sum_{p,q,r=1}^n
R_{p,q,r}^{B,j}(U_p^RV_{q}^I-U_p^IV_{q}^R)
(W_r^R+\sqrt{-1}W_r^I)
\nonumber
\\
&\quad 
+
\sum_{p,q,r=1}^n
R_{p,q,r}^{A,j}(U_p^RV_{q}^R+U_p^IV_{q}^I)
(W_r^R+\sqrt{-1}W_r^I).
\nonumber
\end{align}
By using an elementary calculation for complex numbers, 
\eqref{eq:momo2} for $U,V,W$, 
and using \eqref{eq:711}, 
we deduce  
\begin{align}
&\langle
R(U,V)W
\rangle_j
\nonumber
\\
&= 
\sum_{p,q,r=1}^n
R_{p,q,r}^{A,j}
\Re\left[\overline{(U_p^R+\sqrt{-1}U_p^I)}(V_{q}^R+\sqrt{-1}V_{q}^I)\right]
(W_r^R+\sqrt{-1}W_r^I)
\nonumber
\\
&\quad 
+
\sum_{p,q,r=1}^n
R_{p,q,r}^{B,j}
\Im\left[\overline{(U_p^R+\sqrt{-1}U_p^I)}(V_{q}^R+\sqrt{-1}V_{q}^I)\right]
(W_r^R+\sqrt{-1}W_r^I)
\nonumber
\\
&=
\sum_{p,q,r=1}^n
\left\{
	  R_{p,q,r}^{A,j}\Re\left[\overline{\langle U \rangle_p}\langle V \rangle_q\right]
	  +R_{p,q,r}^{B,j}\Im\left[\overline{\langle U \rangle_p}\langle V \rangle_q\right]
	  \right\}
	  \langle W \rangle_r
\nonumber
\\
&=
\frac{1}{2}
\sum_{p,q,r=1}^n
\left\{
	  (R_{p,q,r}^{A,j}+\sqrt{-1}R_{p,q,r}^{B,j})
	  \langle U \rangle_p\overline{\langle V \rangle_q}
	  -(-R_{p,q,r}^{A,j}+\sqrt{-1}R_{p,q,r}^{B,j})
	  \overline{\langle U \rangle_p}\langle V \rangle_q
	  \right\}
	  \langle W \rangle_r
\nonumber
\\
&=
\sum_{p,q,r=1}^n
\left\{
	  S_{p,q,r}^{j}
	  \langle U \rangle_p\overline{\langle V \rangle_q}
	  -T_{p,q,r}^{j}
	  \overline{\langle U \rangle_p}\langle V \rangle_q
	  \right\}
	  \langle W \rangle_r.
	  \label{eq:81-}
\end{align} 
Here, note that  
\begin{align}
T_{p,q,r}^j=S_{q,p,r}^j
\quad 
(\forall p,q,r,j\in \left\{1,\ldots,n\right\}), 
\label{eq:TtoS}
\end{align}
which immediately follows from  \eqref{eq:R1}, \eqref{eq:R2} in Proposition~\ref{proposition:R1}. 
Applying \eqref{eq:TtoS} to \eqref{eq:81-}, we derive 
\begin{align}
\langle
R(U,V)W
\rangle_j
&=
\sum_{p,q,r=1}^n
\left\{
	  S_{p,q,r}^{j}
	  \langle U \rangle_p\overline{\langle V \rangle_q}
	  -T_{q,p,r}^{j}
	  \overline{\langle U \rangle_q}\langle V \rangle_p
	  \right\}
	  \langle W \rangle_r
	  \nonumber
\\
&=
\sum_{p,q,r=1}^n
	  S_{p,q,r}^{j}\left(
	  \langle U \rangle_p\overline{\langle V \rangle_q}
	 -\langle V \rangle_p\overline{\langle U \rangle_q}
	  \right)
	  \langle W \rangle_r, 
	  \label{eq:81}	  
\end{align}
which is the desired result. 
\end{proof}
The next proposition also follows from Proposition~\ref{proposition:R1}. 
\begin{proposition}
\label{proposition:R4}
Under the same assumption as that in Proposition~\ref{proposition:R1}, 
\begin{align}
S_{p,q,r}^j
&=
S_{r,q,p}^j
\quad
(\forall p,q,r,j\in \left\{1,\ldots,n\right\}). 
\label{eq:tsu3}
\end{align}
\end{proposition}
\begin{proof}[Proof of Proposition~\ref{proposition:R4}]
For any $p,q,r,j\in \left\{1,\ldots,n\right\}$, 
it follows that 
\begin{align}
-\sqrt{-1}(R_{p,q,r}^{A,j}
-R_{p,r,q}^{A,j})
&=
-\sqrt{-1}(
R_{p,q,r}^{A,j}
+R_{r,p,q}^{A,j})
\quad 
(\because \eqref{eq:R1})
\nonumber
\\
&=
\sqrt{-1}R_{q,r,p}^{A,j}
\quad
(\because \eqref{eq:R3})
\nonumber
\\
&=
-(R_{r,p,q}^{B,j}-R_{p,q,r}^{B,j})
\quad 
(\because \eqref{eq:R4})
\nonumber
\\
&=R_{p,q,r}^{B,j}-R_{p,r,q}^{B,j}.
\quad
(\because \eqref{eq:R2})
\label{eq:tsu1}
\end{align}
By multiplying both sides of \eqref{eq:tsu1} by $\sqrt{-1}$
and by transposing the terms,   
\eqref{eq:tsu1} reads 
\begin{equation}
R_{p,q,r}^{A,j}
-\sqrt{-1}R_{p,q,r}^{B,j}
=
R_{p,r,q}^{A,j}
-\sqrt{-1}R_{p,r,q}^{B,j}.
\label{eq:ntsu2}
\end{equation}
This shows 
\begin{align}
T_{p,q,r}^j
&=
T_{p,r,q}^j
\quad
(\forall p,q,r,j\in \left\{1,\ldots,n\right\}). 
\label{eq:tsu2}
\end{align}
By combining \eqref{eq:tsu2} and \eqref{eq:TtoS}, 
we obtain 
\begin{align}
S_{p,q,r}^j
&=
T_{q,p,r}^j
=T_{q,r,p}^j
=S_{r,q,p}^j
\quad
(\forall p,q,r,j\in \left\{1,\ldots,n\right\}).  
\end{align}
\end{proof} 
Propositions~\ref{prop:tensor} and \ref{proposition:R4} 
and sometimes Proposition~\ref{proposition:Rloc} 
will be sufficient to show the claims 
in Section~\ref{subsection:reduction} and 
Section~\ref{section:Examples}. 
\subsection{Proof of Theorem~\ref{theorem:momo}}
\label{subsection:reduction}
In this subsection,  we complete the proof of Theorem~\ref{theorem:momo}.
\begin{proof}[Proof of Theorem~\ref{theorem:momo}]
Let $\left\{e_j, Je_j\right\}_{j=1}^n$  be the orthonormal frame 
for $u^{-1}TN$ introduced in Section~\ref{subsection:frame}. 
We represent $u_x, u_t\in \Gamma(u^{-1}TN)$ by 
\begin{align}
u_x&=\sum_{p=1}^{n}(\xi_p+\eta_pJ_u)e_p, 
\quad
u_t=\sum_{p=1}^{n}(\mu_p+\nu_pJ_u)e_p,  
\label{eq:uxt}
\end{align}
where $\xi_p$, $\eta_p$, $\mu_p$, $\nu_p$ for 
$p\in \left\{1,\ldots,n\right\}$
are real-valued functions of $(t,x)$. 
Set $Q_j:=\langle u_x\rangle_j$ and $P_j:=\langle u_t\rangle_j$ for 
$j\in \left\{1,\ldots,n\right\}$. 
By \eqref{eq:momo2}, they satisfy  
 \begin{align}
 Q_j&=\langle u_x\rangle_j=\xi_j+\sqrt{-1}\eta_j, 
 \quad
 P_j=\langle u_t\rangle_j=\mu_j+\sqrt{-1}\nu_j.
\label{eq:req10181}
 \end{align}  
Substitution of \eqref{eq:pde} into $P_j=\langle u_t\rangle_j$ yields   
\begin{align}
P_j&=
\langle (a\,J_u\nabla_x^3+\lambda\,J_u\nabla_x)u_x\rangle_j
\nonumber
\\
&\quad
+b\langle R(\nabla_xu_x,u_x)J_uu_x\rangle_j
+c\langle R(J_uu_x,u_x)\nabla_xu_x\rangle_j. 
\label{eq:236181}
\end{align}
We compute the right hand side of \eqref{eq:236181}. 
First, it follows from \eqref{eq:uxt} 
\begin{align}
& \left(
 a\,J_u\nabla_x^3
 +\lambda\,J_u\nabla_x
 \right)
 u_x
=
\left(
 a\,J_u\nabla_x^3
 +\lambda\,J_u\nabla_x
 \right)
 \sum_{j=1}^n(\xi_j+J_u\eta_j)e_j
 \nonumber
 \\
 &=
 \sum_{j=1}^n
 (
 a\,J_u\p_x^3\xi_j-a\p_x^3\eta_j+\lambda\,J_u\p_x\xi_j-\lambda\,\p_x\eta_j
 )e_j, 
 \nonumber
 \end{align}
 which by \eqref{eq:momo2} shows 
 \begin{align}
 \langle
 \left(
  a\,J_u\nabla_x^3
  +\lambda\,J_u\nabla_x
  \right)
  u_x
 \rangle_j
 &=
  a \sqrt{-1}\p_x^3\xi_j-a\p_x^3\eta_j+\lambda \sqrt{-1}\p_x\xi_j-\lambda\,\p_x\eta_j
  \nonumber
  \\
  &=
   \sqrt{-1}(a\p_x^3+\lambda \p_x)Q_j. 
   \label{eq:s21}
 \end{align} 
 Second, noting 
 $\langle \nabla_xu_x\rangle_p=\p_xQ_p$, 
 $\langle u_x\rangle_q=Q_q$, 
 and $\langle J_uu_x\rangle_r=\sqrt{-1}Q_r$, 
 we apply \eqref{eq:81} in Proposition~\ref{prop:tensor} 
 for $(U,V,W)=(\nabla_xu_x, u_x,J_uu_x)$ to deduce   
 \begin{align}
 &\langle 
 R(\nabla_xu_x,u_x)J_uu_x
 \rangle_j
 =
 \sum_{p,q,r=1}^n
 S_{p,q,r}^j
 \left(
 \p_xQ_p\overline{Q_q}
 -Q_p\overline{\p_xQ_q}
 \right)
 \sqrt{-1}Q_r
 	  \nonumber
 \\
 &=
  \sqrt{-1}\sum_{p,q,r=1}^n
  S_{p,q,r}^j\p_xQ_p\overline{Q_q}Q_r
  - \sqrt{-1}\sum_{p,q,r=1}^n
    S_{p,q,r}^jQ_p\overline{\p_xQ_q}Q_r.
  \label{eq:s41}
  \end{align}
Third, in the same way as above, 
noting $\langle J_uu_x\rangle_p=\sqrt{-1}Q_p$, 
$\langle u_x\rangle_q=Q_q$, 
and $\langle \nabla_xu_x\rangle_r=\p_xQ_r$, 
we apply \eqref{eq:81} for $(U,V,W)=(J_uu_x,u_x,\nabla_xu_x)$ to deduce 
\begin{align}
\langle
R(J_uu_x,u_x)\nabla_xu_x
\rangle_j
&=
\sum_{p,q,r=1}^n
 S_{p,q,r}^j
 \left(
 \sqrt{-1}Q_p\overline{Q_q}
 -
 Q_p\overline{\sqrt{-1}Q_q}
 \right)
 \p_xQ_r
 \nonumber
 \\
 &=
2\sqrt{-1}\sum_{p,q,r=1}^n
  S_{p,q,r}^j
  Q_p\overline{Q_q}
  \p_xQ_r.
  \nonumber
  \end{align}
Furthermore, replacing indexes $(p,q,r)\to (r,q,p)$ in the summation 
and using \eqref{eq:tsu3} in Proposition~\ref{proposition:R4}, 
we find     
 \begin{align}
 \langle
 R(J_uu_x,u_x)\nabla_xu_x
 \rangle_j
 &= 
 2\sqrt{-1}\sum_{p,q,r=1}^n
   S_{p,q,r}^j
   \p_xQ_p\overline{Q_q}
   Q_r.
\label{eq:s42}
\end{align}
Substituting \eqref{eq:s21}, \eqref{eq:s41}, and \eqref{eq:s42} 
into \eqref{eq:236181}, 
we have  
\begin{align}
P_j&=
\sqrt{-1} (a\p_x^3+\lambda \p_x)Q_j
+
 (b+2c)\sqrt{-1}\sum_{p,q,r=1}^n
  S_{p,q,r}^j\p_xQ_p\overline{Q_q}Q_r
 \nonumber
 \\&\quad
-b\sqrt{-1}\sum_{p,q,r=1}^n
S_{p,q,r}^jQ_p\overline{\p_xQ_q}Q_r.
\label{eq:s51}
\end{align} 
\par 
Next, we seek the condition obtained from the fact $\nabla_xu_t=\nabla_tu_x$. 
By \eqref{eq:uxt} and \eqref{eq:mf01},    
\begin{align}
\langle 
\nabla_xu_t
\rangle_j
&=
\left\langle 
 \sum_{p=1}^{n}(\p_x\mu_p+J_u\p_x\nu_p)e_p
 \right\rangle_j
 = 
 \p_x\mu_j+\sqrt{-1}\p_x\nu_j
 =\p_xP_j.
 \label{eq:s61}
 \end{align}
On the other hand, by \eqref{eq:uxt} and \eqref{eq:tep}, 
\begin{align}
 \nabla_tu_x
 &=
 \sum_{p=1}^{n}(\p_t\xi_p+J_u\p_t\eta_p)e_p
  +
  \sum_{p,r=1}^{n}(\xi_r+J_u\eta_r)(a_p^r+J_ub_p^r)e_p, 
  \nonumber 
\end{align}
which shows 
\begin{align}
\langle 
\nabla_tu_x
\rangle_j
   &= 
   \p_t\xi_j+\sqrt{-1}\p_t\eta_j
   +\sum_{r=1}^{n}\left\{
   (a_j^r\xi_r-b_j^r\eta_r)+\sqrt{-1}(a_j^r\eta_r+b_j^r\xi_r)
    \right\}
   \nonumber
   \\
   &=
      \p_tQ_j
      +\sum_{r=1}^{n}
      (a_j^r+\sqrt{-1}b_j^r)Q_r. 
      \label{eq:s62}
 \end{align}
Since $\nabla_xu_t=\nabla_tu_x$ holds, 
\eqref{eq:s61} and \eqref{eq:s62} show   
\begin{align}
 \p_tQ_j
 &=
 \p_xP_j-\sum_{r=1}^{n}
      (a_j^r+\sqrt{-1}b_j^r)Q_r 
\quad
(j\in \left\{1,\ldots,n\right\}).
      \label{eq:s63}
\end{align}
\par
Next, we seek the condition obtained from the fact  $\nabla_x\nabla_te_r=\nabla_t\nabla_xe_r+R(u_x,u_t)e_r=R(u_x,u_t)e_r$. 
By \eqref{eq:mf01} and \eqref{eq:tep}, 
\begin{align}
\langle 
\nabla_x\nabla_te_r
\rangle_j
=
\left\langle
\sum_{p=1}^{n}
\p_x(a_p^r+J_ub_p^r)e_p
\right\rangle_j
=
\p_x(a_j^r+\sqrt{-1}b_j^r). 
\label{eq:s71}  
\end{align} 
On the other hand, 
noting 
$\langle u_x\rangle_p=Q_p$, 
$\langle u_t\rangle_p=P_p$ 
and 
$\langle e_r\rangle_p=\delta_{pr}$, 
we apply \eqref{eq:81} for $(U,V,W)=(u_x,u_t,e_r)$, 
which yields 
\begin{align}
\langle R(u_x,u_t)e_r\rangle_j
&=
\sum_{p,q,r^{\prime}=1}^n
S_{p,q,r^{\prime}}^j
(Q_p\overline{P_q}-P_p\overline{Q_q})
\delta_{rr^{\prime}}
=
\sum_{p,q=1}^n
S_{p,q,r}^j
(Q_p\overline{P_q}-P_p\overline{Q_q}).
\label{eq:s72}
\end{align}
Comparing \eqref{eq:s71} and \eqref{eq:s72}, 
we have  
\begin{align}
\p_x(a_j^r+\sqrt{-1}b_j^r)
&=
\sum_{p,q=1}^n
S_{p,q,r}^j
(Q_p\overline{P_q}-P_p\overline{Q_q})
\quad
(j,r\in \left\{1,\ldots,n\right\}).
\label{eq:s73}
\end{align}
Furthermore, using \eqref{eq:s51} with the 
replacement of indexes $(p,q,r)\to (\alpha,\beta,\gamma)$ in the summation,  
we have 
\begin{align}
Q_p\overline{P_q}
&=
  -\sqrt{-1} a Q_p\overline{\p_x^3Q_q}-\sqrt{-1} \lambda Q_p\overline{\p_xQ_q}
  \nonumber
  \\
  &\quad
     -(b+2c)\sqrt{-1}\sum_{\alpha,\beta,\gamma=1}^n
      \overline{S_{\alpha,\beta,\gamma}^q}
      \overline{\p_xQ_{\alpha}}Q_{\beta}\overline{Q_{\gamma}}Q_p
\nonumber
\\
&\quad 
+b\sqrt{-1}
\sum_{\alpha,\beta,\gamma=1}^n
      \overline{S_{\alpha,\beta,\gamma}^q}
\overline{Q_{\alpha}}\p_xQ_{\beta}\overline{Q_{\gamma}}Q_p
\nonumber
\\     
&=
  \p_x(-\sqrt{-1} a Q_p\overline{\p_x^2Q_q})
  +\sqrt{-1} a \p_xQ_p\overline{\p_x^2Q_q}
  -\sqrt{-1}\lambda Q_p\overline{\p_xQ_q}
  \nonumber
  \\
  &\quad
     -(b+2c)\sqrt{-1}\sum_{\alpha,\beta,\gamma=1}^n
      \overline{S_{\alpha,\beta,\gamma}^q}
      \overline{\p_xQ_{\alpha}}Q_{\beta}\overline{Q_{\gamma}}Q_p
\nonumber
\\
&\quad 
+b\sqrt{-1}
\sum_{\alpha,\beta,\gamma=1}^n
      \overline{S_{\alpha,\beta,\gamma}^q}
\overline{Q_{\alpha}}\p_xQ_{\beta}\overline{Q_{\gamma}}Q_p
    \nonumber
\end{align}
for any $p,q\in \left\{1,\ldots,n\right\}$. 
Substituting this into \eqref{eq:s73} multiplied by $\sqrt{-1}$, we deduce  
\begin{align}
&\sqrt{-1}\p_x(a_j^r+\sqrt{-1}b_j^r)
\nonumber
\\
&=
  a\sum_{p,q=1}^n
  S_{p,q,r}^{j} 
  \p_x\left(
  \p_x^2Q_p\overline{Q_q}
  +Q_p\overline{\p_x^2Q_q}
  \right)
  \nonumber
  \\
  &\quad
  -a\sum_{p,q=1}^n 
  S_{p,q,r}^{j}
  \left(
  \p_x^2Q_p\overline{\p_xQ_q}+\p_xQ_p\overline{\p_x^2Q_q}
  \right)
  \nonumber
  \\
  &\quad
 +\lambda 
 \sum_{p,q=1}^n  
 S_{p,q,r}^{j}
 \left(
 \p_xQ_p\overline{Q_q}
 +Q_p\overline{\p_xQ_q}
 \right)
 \nonumber
\\&\quad
     +(b+2c)\sum_{p,q,\alpha,\beta,\gamma=1}^n 
     \left(
     S_{p,q,r}^{j}\overline{S_{\alpha,\beta,\gamma}^{q}}
     \overline{\p_xQ_{\alpha}}Q_{\beta}\overline{Q_{\gamma}}
     Q_p
     +S_{p,q,r}^{j}S_{\alpha,\beta,\gamma}^{p}
          \p_xQ_{\alpha}\overline{Q_{\beta}}Q_{\gamma}
          \overline{Q_q}
     \right)
     \nonumber
     \\
     &\quad
     -b\sum_{p,q,\alpha,\beta,\gamma=1}^n
     \left(
     S_{p,q,r}^{j}\overline{S_{\alpha,\beta,\gamma}^{q}}
     \overline{Q_{\alpha}}\p_xQ_{\beta}\overline{Q_{\gamma}}
     Q_p
     +S_{p,q,r}^{j}S_{\alpha,\beta,\gamma}^{p}
     Q_{\alpha}\overline{\p_xQ_{\beta}}Q_{\gamma}
          \overline{Q_q}
     \right)
\nonumber
\\
&=:
\p_x\left\{
  a\sum_{p,q=1}^n
  S_{p,q,r}^{j} 
  \left(
  \p_x^2Q_p\overline{Q_q}
  +Q_p\overline{\p_x^2Q_q}
  \right)
  \right\}
  -\p_x\left\{
  a\sum_{p,q=1}^n 
  S_{p,q,r}^{j}
  \p_xQ_p\overline{\p_xQ_q}
  \right\}
    \nonumber
    \\
    &\quad
 +\p_x\left\{
 \lambda 
 \sum_{p,q=1}^n  
 S_{p,q,r}^{j}
 Q_p\overline{Q_q}
 \right\}
     -S_{j,r}^R
     -S_{j,r}^{\nabla R}
 \label{eq:s83}
\end{align}
for any $j,r\in \left\{1,\ldots,n\right\}$,  
where 
\begin{align}
&S_{j,r}^R:=
-(b+2c)(S_1+S_2)+b(S_3+S_4),  
\label{eq:211}
\\
&\qquad S_1:=\sum_{p,q,\alpha,\beta,\gamma=1}^n 
     S_{p,q,r}^{j}\overline{S_{\alpha,\beta,\gamma}^{q}}
     \overline{\p_xQ_{\alpha}}Q_{\beta}\overline{Q_{\gamma}}
     Q_p,
     \label{eq:S1}
     \\
&\qquad S_2:=\sum_{p,q,\alpha,\beta,\gamma=1}^n 
     S_{p,q,r}^{j}S_{\alpha,\beta,\gamma}^{p}
          \p_xQ_{\alpha}\overline{Q_{\beta}}Q_{\gamma}
          \overline{Q_q},
          \label{eq:S2}
          \\
&\qquad S_3:=\sum_{p,q,\alpha,\beta,\gamma=1}^n 
     S_{p,q,r}^{j}\overline{S_{\alpha,\beta,\gamma}^{q}}
     \overline{Q_{\alpha}}\p_xQ_{\beta}\overline{Q_{\gamma}}
     Q_p,
     \label{eq:S3}
     \\
&\qquad S_4:=\sum_{p,q,\alpha,\beta,\gamma=1}^n 
S_{p,q,r}^{j}S_{\alpha,\beta,\gamma}^{p}
     Q_{\alpha}\overline{\p_xQ_{\beta}}Q_{\gamma}
          \overline{Q_q}, 
          \label{eq:S4}
\end{align}
and 
\begin{align}
S_{j,r}^{\nabla R}
&:=
 a\sum_{p,q=1}^n
  \p_x(S_{p,q,r}^{j}) 
  \left(
  \p_x^2Q_p\overline{Q_q}
  +Q_p\overline{\p_x^2Q_q}
  \right)
  \nonumber
  \\
  &\quad
- a\sum_{p,q=1}^n 
  \p_x(S_{p,q,r}^{j})
  \p_xQ_p\overline{\p_xQ_q}
+\lambda 
 \sum_{p,q=1}^n  
 \p_x(S_{p,q,r}^{j})
 Q_p\overline{Q_q}.
 \label{eq:05201}
\end{align}
Here, it follows that 
\begin{align}
|a_j^r+\sqrt{-1}b_j^r|
&=
|h(\nabla_te_r,e_j)+\sqrt{-1}h(\nabla_te_r,J_ue_j)|
\leqslant 
2|\nabla_te_r|_h, 
\nonumber
\end{align}
where $|\cdot|_h=\sqrt{h(\cdot,\cdot)}$.
Moreover, $\nabla_te_r=O(|u_t|_h)$ holds, since $N$ is compact.
In addition,
$|u_t(t,x)|_h=
|(aJ_u\nabla_x^3u_x+\cdots)(t,x)|_h\to 0$ as $x\to -\infty$ for each $t\in (-T,T)$,  
since $u_x(t,\cdot)$ is in the Schwartz class. 
Combining them, we see 
\begin{equation}
\displaystyle\lim_{x\to -\infty}(a_j^r+\sqrt{-1}b_j^r)(t,x)
=0 
\quad 
(t\in (-T,T)).
\label{eq:add1015}
\end{equation} 
Integrating both sides of \eqref{eq:s83} with respect to $x$,  
and using \eqref{eq:add1015}, 
we obtain 
\begin{align}
&\sqrt{-1}(a_j^r+\sqrt{-1}b_j^r)(t,x)
\nonumber
\\
&=
a\sum_{p,q=1}^n
S_{p,q,r}^{j}
\p_x^2Q_p\overline{Q_q}
+a\sum_{p,q=1}^n
S_{p,q,r}^{j}
Q_p\overline{\p_x^2Q_q}
-a\sum_{p,q=1}^n
S_{p,q,r}^{j}\p_xQ_p\overline{\p_xQ_q}
\nonumber
\\
&\quad
+\lambda
\sum_{p,q=1}^n
S_{p,q,r}^{j}Q_p\overline{Q_q}
-\int_{-\infty}^{x}
S_{j,r}^R(t,y)
dy
-\int_{-\infty}^{x}
S_{j,r}^{\nabla R}(t,y)
dy. 
\label{eq:s13R1}
\end{align}
\par 
By substitution of \eqref{eq:s51} and \eqref{eq:s13R1} into \eqref{eq:s63} 
multiplied by $\sqrt{-1}$, we deduce  
\begin{align}
&\sqrt{-1}\p_tQ_j+(a\p_x^4+\lambda \p_x^2)Q_j
\nonumber
\\
&=
-(b+2c)
    \sum_{p,q,r=1}^n
    \p_x\left\{
    S_{p,q,r}^{j}
    \p_xQ_p\overline{Q_q}Q_r
    \right\}
+b\sum_{p,q,r=1}^n
        \p_x\left\{
    S_{p,q,r}^{j}
    Q_{p}\overline{\p_xQ_q}Q_r
\right\}
\nonumber
\\
&\quad
-a
\sum_{p,q,r=1}^n
S_{p,q,r}^{j}
\p_x^2Q_p\overline{Q_q}Q_r
-a\sum_{p,q,r=1}^n
S_{p,q,r}^{j}
Q_p\overline{\p_x^2Q_q}Q_r
\nonumber
\\
&\quad
+a
\sum_{p,q,r=1}^n
S_{p,q,r}^{j}\p_xQ_p\overline{\p_xQ_q}Q_r
-\lambda
\sum_{p,q,r=1}^n
S_{p,q,r}^{j}Q_p\overline{Q_q}Q_r
\nonumber
\\
&\quad
+\sum_{r=1}^{n}
\left(\int_{-\infty}^{x}
S_{j,r}^R(t,y)
dy\right)Q_r
+\sum_{r=1}^{n}
\left(
\int_{-\infty}^{x}
S_{j,r}^{\nabla R}(t,y)
dy\right)
Q_r
\nonumber
\\
&=
   \left(-a-b-2c\right)
    \sum_{p,q,r=1}^n
    S_{p,q,r}^{j}
    \p_x^2Q_p\overline{Q_q}Q_r
    \nonumber
    \\
    &\quad
    +\left(a-b-2c+b\right)
        \sum_{p,q,r=1}^n
        S_{p,q,r}^{j}
        \p_xQ_p\overline{\p_xQ_q}Q_r
        \nonumber
        \\
    &\quad
       +(-b-2c)
        \sum_{p,q,r=1}^n
        S_{p,q,r}^{j}
        \p_xQ_p\overline{Q_q}\p_xQ_r
    +(-a+b)
        \sum_{p,q,r=1}^n
    S_{p,q,r}^{j}
    Q_{p}\overline{\p_x^2Q_q}Q_r
\nonumber
        \\
        &\quad
    +b \sum_{p,q,r=1}^n
    S_{p,q,r}^{j}
    Q_{p}\overline{\p_xQ_q}\p_xQ_r
    +(-b-2c)
        \sum_{p,q,r=1}^n
        \p_x(
        S_{p,q,r}^{j})
        \p_xQ_p\overline{Q_q}Q_r
        \nonumber
        \\
        &\quad
        +b
            \sum_{p,q,r=1}^n
            \p_x(
        S_{p,q,r}^{j})
        Q_{p}\overline{\p_xQ_q}Q_r
            -\lambda
            \sum_{p,q,r=1}^n
            S_{p,q,r}^{j}Q_p\overline{Q_q}Q_r
            \nonumber
            \\
            &\quad
            +
            \sum_{r=1}^{n}
            \left(\int_{-\infty}^{x}
            S_{j,r}^R(t,y)
            dy\right)Q_r
            +
            \sum_{r=1}^{n}
            \left(
            \int_{-\infty}^{x}
            S_{j,r}^{\nabla R}(t,y)
            dy\right)
            Q_r.
    \nonumber
\end{align}
Moreover, replacing indexes $(p,q,r)\to (r,q,p)$ in the summation and 
using \eqref{eq:tsu3} shows
\begin{align}
\sum_{p,q,r=1}^n
    S_{p,q,r}^{j}
    Q_{p}\overline{\p_xQ_q}\p_xQ_r
 &=\sum_{p,q,r=1}^n
                S_{p,q,r}^{j}
                \p_xQ_p\overline{\p_xQ_q}Q_r.
\nonumber
\end{align}
Using this, we have
\begin{align}
&\sqrt{-1}\p_tQ_j+(a\p_x^4+\lambda \p_x^2)Q_j
\nonumber
\\
&=
   d_1
    \sum_{p,q,r=1}^n
    S_{p,q,r}^{j}
    \p_x^2Q_p\overline{Q_q}Q_r
    +d_2
            \sum_{p,q,r=1}^n
        S_{p,q,r}^{j}
        Q_{p}\overline{\p_x^2Q_q}Q_r
    +d_3
        \sum_{p,q,r=1}^n
        S_{p,q,r}^{j}
        \p_xQ_p\overline{\p_xQ_q}Q_r
            \nonumber
             \\
             &\quad
       +d_4
        \sum_{p,q,r=1}^n
        S_{p,q,r}^{j}
        \p_xQ_p\overline{Q_q}\p_xQ_r
        +(-b-2c)
            \sum_{p,q,r=1}^n
            \p_x(
            S_{p,q,r}^{j})
            \p_xQ_p\overline{Q_q}Q_r
      \nonumber
        \\
        &\quad
            +b
                \sum_{p,q,r=1}^n
                \p_x(
            S_{p,q,r}^{j})
            Q_{p}\overline{\p_xQ_q}Q_r
    -\lambda
    \sum_{p,q,r=1}^n
    S_{p,q,r}^{j}Q_p\overline{Q_q}Q_r
 \nonumber
 \\
 &\quad
    +
    \sum_{r=1}^{n}
    \left(\int_{-\infty}^{x}
    S_{j,r}^R(t,y)
    dy\right)Q_r
    +
    \sum_{r=1}^{n}
    \left(
    \int_{-\infty}^{x}
    S_{j,r}^{\nabla R}(t,y)
    dy\right)
    Q_r, 
    \label{eq:s51251}
\end{align}
where $d_1,d_2,d_3,d_4$ are the same constants as those 
in the statement of  Theorem~\ref{theorem:momo}. 
\par 
Furthermore, we compute the last two terms 
of the right hand side of \eqref{eq:s51251}.
First, recalling \eqref{eq:211} with \eqref{eq:S1}-\eqref{eq:S4}, 
we see $S_{j,r}^R$ is equal to $f_{j,r}^1(Q,\p_xQ)$ in \eqref{eq:r6201}, 
and hence
\begin{align}
    \sum_{r=1}^{n}
    \left(\int_{-\infty}^{x}
    S_{j,r}^R(t,y)
    dy\right)Q_r
&=
    \sum_{r=1}^{n}
    \left(\int_{-\infty}^{x}
    f_{j,r}^1(Q,\p_xQ)(t,y)
    dy\right)Q_r.
\label{eq:236201}
\end{align}
Second, it follows from \eqref{eq:05201}
\begin{align}
\int_{-\infty}^{x}
    S_{j,r}^{\nabla R}(t,y)
    dy
&=
a
\int_{-\infty}^{x}
\sum_{p,q=1}^n
  \left(
  \p_x(S_{p,q,r}^{j}) 
  (\p_x^2Q_p\overline{Q_q}
  +Q_p\overline{\p_x^2Q_q})
  \right)
  (t,y)
dy
\nonumber
\\
&\quad 
-a
\int_{-\infty}^{x}
\sum_{p,q=1}^n
\left( 
  \p_x(S_{p,q,r}^{j})
  \p_xQ_p\overline{\p_xQ_q}
\right)(t,y)
dy
\nonumber
\\
&\quad 
+\lambda
\int_{-\infty}^{x} 
 \sum_{p,q=1}^n  
 \left(
 \p_x(S_{p,q,r}^{j})
 Q_p\overline{Q_q}
 \right)
 (t,y)
dy.
\label{eq:05202}
\end{align}
For the first term of the right hand side, 
we rewrite as 
\begin{align}
\p_x(S_{p,q,r}^{j}) 
  (\p_x^2Q_p\overline{Q_q}
  +Q_p\overline{\p_x^2Q_q})
  &=
\p_x\left\{
\p_x(S_{p,q,r}^{j}) 
  (\p_xQ_p\overline{Q_q}
  +Q_p\overline{\p_xQ_q})
\right\}
\nonumber 
\\
&\quad 
-\p_x^2(S_{p,q,r}^{j}) 
  (\p_xQ_p\overline{Q_q}
  +Q_p\overline{\p_xQ_q})
  \nonumber
 \\
 &\quad 
 -2\p_x(S_{p,q,r}^{j}) 
    \p_xQ_p\overline{\p_xQ_q}. 
\label{eq:add11172}
\end{align}
Here, 
we see there exists a positive constant $C(N)$ depending only on $N$ 
such that 
$$
|\p_x\left(S_{p,q,r}^j\right)|\leqslant C(N)|Q|, 
\quad 
|\p_x^2\left(S_{p,q,r}^j\right)|\leqslant C(N)(|\p_xQ|+|Q|),  
$$ 
since $S_{p,q,r}^j(t,x)$ depends on $u(t,x)\in N$ 
and $N$ is compact. 
(This can be also proved by taking partial derivatives of 
the right hand of \eqref{eq:add11171} with respect to $x$.) 
This ensures 
$\displaystyle\lim_{x\to -\infty}\p_x(S_{p,q,r}^{j})(\p_xQ_p\overline{Q_q}
+Q_p\overline{\p_xQ_q})(t,x)=0$,  
since $Q(t,\cdot):\RR\to \mathbb{C}^n$ is in the Schwartz class. 
Noting this and substituting \eqref{eq:add11172} into \eqref{eq:05202} leads to
\begin{align}
    \sum_{r=1}^{n}
    \left(\int_{-\infty}^{x}
    S_{j,r}^{\nabla R}(t,y)
    dy\right)Q_r
&=
a\sum_{p,q,r=1}^n
  \p_x(S_{p,q,r}^{j}) 
  (\p_xQ_p\overline{Q_q}Q_r
  +Q_p\overline{\p_xQ_q}Q_r)
  \nonumber
  \\
  &\quad
  +
    \sum_{r=1}^{n}
    \left(\int_{-\infty}^{x}
    f_{j,r}^2(Q,\p_xQ)(t,y)
    dy\right)Q_r, 
    \label{eq:236202}
\end{align}
where $f_{j,r}^2(Q,\p_xQ)$ is given by \eqref{eq:r6202}.
Substituting \eqref{eq:236201} and \eqref{eq:236202} 
into \eqref{eq:s51251}, we derive the desired expression 
\eqref{eq:r51251} with \eqref{eq:r6201} and \eqref{eq:r6202}, 
which completes the proof of Theorem~\ref{theorem:momo}. 
\end{proof}
\section{Examples (1) and (2)}
\label{section:Examples}
In this section, taking two examples of $(N,J,h)$,  
we formulate \eqref{eq:r51251} for $Q$ 
in Theorem~\ref{theorem:momo} more explicitly. 
\subsection{Example~(1).}
\label{subsection:ex1}
Let $(N,J,h)$ be a compact Riemann surface.
Since $n=1$ in this setting,  
the orthonormal frame introduced by \eqref{eq:mf01} 
in Section~\ref{subsection:frame} is $\left\{e_1, J_ue_1\right\}$, 
and thus only $S_{1,1,1}^1$ is required to compute \eqref{eq:r51251}. 
We see 
\begin{equation}
S_{1,1,1}^1=\dfrac{\kappa(u)}{2}, 
\label{eq:s1111}
\end{equation}
where $(\kappa(u))(t,x):=\kappa(u(t,x))$ denotes the Gaussian curvature at $u(t,x)\in N$ 
which is known to be characterized by
\begin{equation}
\kappa(u)=h(R(e_1,J_ue_1)J_ue_1,e_1).
\label{eq:nawa1}
\end{equation} 
To see this, recall that \eqref{eq:add11171} for $p,q,r,j=1$ yields 
\begin{align}
2S_{1,1,1}^1
&=
h(R(e_1,e_1)e_1,e_1)
+\sqrt{-1}h(R(e_1,e_1)e_1,J_ue_1)
\nonumber
\\
&\quad
+
\sqrt{-1}
\left\{
h(R(e_1,J_ue_1)e_1,e_1)
+\sqrt{-1}h(R(e_1,J_ue_1)e_1,J_ue_1)
\right\}.
\nonumber
\end{align} 
Moreover, by (i) and (iii) in Proposition~\ref{prop:fact_R}, 
\begin{align}
&h(R(e_1,e_1)e_1,e_1)
=\sqrt{-1}h(R(e_1,e_1)e_1,J_ue_1)
=h(R(e_1,J_ue_1)e_1,e_1)=0, 
\nonumber
\\
&h(R(e_1,J_ue_1)e_1,J_ue_1)
=-h(R(e_1,J_ue_1)J_ue_1,e_1)
=-\kappa(u),  
\nonumber
\end{align}
which shows \eqref{eq:s1111}.
\par
Next, we compute 
\eqref{eq:r6201} and 
\eqref{eq:r6202}.
We write $Q=Q_1$ for simplicity. 
As for \eqref{eq:r6201} in this setting, it follows that 
$
f_{1,1}^1(Q,\p_xQ)
=
-(b+2c)(S_1+S_2)
+b(S_3+S_4)
$
where 
\begin{align}
S_1&=
S_{1,1,1}^{1}\overline{S_{1,1,1}^1}
\overline{\p_xQ}Q
\overline{Q}Q
=\frac{(\kappa(u))^2}{4}
\overline{\p_xQ}Q
|Q|^2, 
\nonumber
\\
S_2
&=
S_{1,1,1}^{1}S_{1,1,1}^1
\p_xQ\overline{Q}
Q\overline{Q}
=\frac{(\kappa(u))^2}{4}
\p_xQ\overline{Q}
|Q|^2,
\nonumber
\\
S_3&=
S_{1,1,1}^{1}\overline{S_{1,1,1}^1}
\overline{Q}\p_xQ
\overline{Q}Q
=
\frac{(\kappa(u))^2}{4}
\p_xQ\overline{Q}
|Q|^2,
\nonumber
\\
S_4
&=
S_{1,1,1}^{1}S_{1,1,1}^1
Q\overline{\p_xQ}
Q\overline{Q}
=
\frac{(\kappa(u))^2}{4}
\overline{\p_xQ}Q
|Q|^2.
\nonumber
\end{align}
Therefore
$$
S_1+S_2=S_3+S_4=\frac{(\kappa(u))^2}{4}\p_x(|Q|^2)|Q|^2
=\frac{(\kappa(u))^2}{8}\p_x(|Q|^4). 
$$
This yields $f_{1,1}^1(Q,\p_xQ)=-\dfrac{c}{4}(\kappa(u))^2\p_x(|Q|^4)$, 
and thus   
\begin{align}
&\sum_{r=1}^{1}
\left(
\int_{-\infty}^x
f_{1,r}^1(Q,\p_xQ)(t,y)\,dy
\right)Q_r
\nonumber\\
&=
-\frac{c}{4}
\left(
\int_{-\infty}^{x}
\left((\kappa(u))^2\p_x(|Q|^4)\right)(t,y)
dy
\right)
Q
\nonumber
\\
&=
-\frac{c}{4}(\kappa(u))^2
|Q|^4Q
+\frac{c}{2}
\left(
\int_{-\infty}^{x}
\left(\kappa(u)(\kappa(u))_x|Q|^4\right)(t,y)
dy
\right)
Q.
\label{eq:05203}
\end{align}
As for \eqref{eq:r6202}, it follows from the definition and \eqref{eq:s1111} 
\begin{align}
&f_{1,1}^2(Q,\p_xQ)
\nonumber
\\&=
-a\p_x^2(S_{1,1,1}^{1}) 
  (\p_xQ\overline{Q}+Q\overline{\p_xQ})
-3a 
  \p_x(S_{1,1,1}^{1})
  \p_xQ\overline{\p_xQ}
+\lambda
 \p_x(S_{1,1,1}^{1})
 Q\overline{Q}
\nonumber
\\
&=
-\frac{a}{2}
  (\kappa(u))_{xx} 
  (\p_xQ\overline{Q}+Q\overline{\p_xQ})
-\frac{3a}{2} 
  (\kappa(u))_x
  |\p_xQ|^2
+\frac{\lambda}{2}
 (\kappa(u))_x
 |Q|^2.
\label{eq:052032}
\end{align}
Substituting \eqref{eq:s1111}-\eqref{eq:052032} 
into \eqref{eq:r51251}, 
we obtain 
\begin{align}
&\sqrt{-1}\p_tQ+(a\p_x^4+\lambda \p_x^2)Q
\nonumber
\\
&=
   \frac{d_1}{2}
    \kappa(u)
    \p_x^2Q|Q|^2
    +\frac{d_2}{2}
        \kappa(u)
        \overline{\p_x^2Q}Q^2
    +\frac{d_3}{2}
        \kappa(u)
        |\p_xQ|^2Q
    +\frac{d_4}{2}
        \kappa(u)
        (\p_xQ)^2\overline{Q}
           \nonumber
            \\
            &\quad
            +\frac{d_5}{2}
                (\kappa(u))_x
                \p_xQ|Q|^2
                +\frac{d_6}{2}
                    (\kappa(u))_x
                Q^2\overline{\p_xQ}
    -\frac{\lambda}{2}
    \kappa(u)|Q|^2Q
    -\frac{c}{4}\kappa^2(u)
    |Q|^4Q
    \nonumber
    \\
    &\quad 
    +\frac{1}{2}
    \left(
    \int_{-\infty}^{x}
      \mathcal{W}_1(Q,\p_xQ)
      (t,y)
    dy
    \right)
    Q, 
    \label{eq:ns51261}
\end{align}
where 
\begin{align}
\mathcal{W}_1(Q,\p_xQ)&=
-a (\kappa(u))_{xx} 
  (\p_xQ\overline{Q}+Q\overline{\p_xQ})
-3a  (\kappa(u))_x|\p_xQ|^2
\nonumber
\\
&\quad
+c\,\kappa(u)(\kappa(u))_x|Q|^4
+\lambda\, (\kappa(u))_x|Q|^2.
\nonumber
\end{align}
\begin{remark}
\label{remark:rs} 
If the Gaussian curvature of $(N,J,h)$ is constant, 
then the nonlocal term in \eqref{eq:ns51261} vanishes, 
and it is easy to check that \eqref{eq:ns51261} under the setting \eqref{eq:11021}
actually coincides with \eqref{eq:DZ} which is transformed from \eqref{eq:bibi} in \cite{DZ2021}. 
This is natural because our orthonormal frame for $n=1$ is essentially the same as that used in \cite{DZ2021} and because
the constancy of the sectional curvature on $(N,J,h)$ ensures $\nabla R=0$. 
In contrast, without the constancy of the curvature,  
\eqref{eq:ns51261} under the setting \eqref{eq:11021} 
includes the nonlocal term  
and does not coincide with \eqref{eq:DZ}, even though 
we rewrite the nonlocal term 
by using the fundamental theorem of calculus. 
This is not strange, because the definitions of \eqref{eq:pde} and \eqref{eq:bibi} for curve flows are originally not the same unless $\nabla R=0$.   
\end{remark}
\subsection{Example~(2).}
\label{subsection:ex2}
Let $(N,J,h)$ be a compact K\"ahler manifold of complex dimension $n$ 
with constant holomorphic sectional curvature $K$. 
It is known that 
\begin{align}
R(U,V)W&=\dfrac{K}{4}
\biggl\{
h(V,W)U-h(U,W)V
+h(U,J_uW)J_uV
\nonumber
\\
&\qquad \qquad \qquad
-h(V,J_uW)J_uU
+2h(U,J_uV)J_uW
\biggr\}
\label{eq:nawa2}
\end{align}
for any $U,V,W\in \Gamma(u^{-1}TN)$, 
and $\nabla R=0$ holds.
In particular, 
\begin{align}
R(e_p,e_{q})e_r
&=
\frac{K}{4}(\delta_{q r}e_p-\delta_{p r}e_{q}), 
\nonumber
\\
R(e_p,J_ue_{q})e_r
&=
-\frac{K}{4}(\delta_{pr}J_ue_{q}
+\delta_{q r}J_ue_p+2\delta_{p q}J_ue_{r}) 
\nonumber
\end{align}
hold for all $p,q,r\in \left\{1,\ldots,n\right\}$. 
Applying them to \eqref{eq:61}-\eqref{eq:71}, 
we see
\begin{align}
(\operatorname{Re}
[R_{p,q,r}^{A,j}]
&=)
h(R(e_p,e_q)r_r,e_j)
=\frac{K}{4}(\delta_{q r}\delta_{p j}-\delta_{p r}\delta_{q j}),
\nonumber
\\
(\operatorname{Im}
[R_{p,q,r}^{A,j}]
&=)
h(R(e_p,e_q)r_r,J_ue_j)
=0, 
\nonumber
\\
(\operatorname{Re}
[R_{p,q,r}^{B,j}]
&=)
h(R(e_p,J_ue_q)r_r,e_j)
=0, 
\nonumber
\\
(\operatorname{Im}
[R_{p,q,r}^{B,j}]
&=)
h(R(e_p,J_ue_q)r_r,J_ue_j)
=
-\frac{K}{4}(\delta_{pr}\delta_{q j}
+\delta_{q r}\delta_{pj}
+2\delta_{p q}\delta_{rj})
\nonumber
\end{align}
for all $p,q,r,j\in \left\{1,\ldots,n\right\}$. 
Substituting them into \eqref{eq:add11171}, we obtain
\begin{align}
S_{p,q,r}^j
&=
\frac{K}{8}(\delta_{qr}\delta_{pj}-\delta_{pr}\delta_{qj})
+\frac{K}{8}(\delta_{pr}\delta_{qj}+\delta_{qr}\delta_{pj}
+2\delta_{pq}\delta_{rj})
\nonumber
\\
&=
\frac{K}{4}(\delta_{qr}\delta_{pj}+\delta_{pq}\delta_{rj})
\quad (\in \RR)
\qquad 
(\forall p,q,r,j\in \left\{1,\ldots,n\right\}).
\label{eq:Spqrj2}
\end{align}
From this, we also see $\p_x(S_{p,q,r}^j)\equiv 0$. 
This does not conflict with Proposition~\ref{proposition:Rloc}.
\par 
We use \eqref{eq:Spqrj2} to compute the right hand side of \eqref{eq:r51251} 
with \eqref{eq:r6201} and \eqref{eq:r6202}.  
It follows that 
\begin{align}
\sum_{p,q,r=1}^n
S_{p,q,r}^j
\p_x^2Q_p\overline{Q_q}Q_r
&=
\frac{K}{4}
\sum_{p,q,r=1}^n
(\delta_{qr}\delta_{pj}+\delta_{pq}\delta_{rj})
\p_x^2Q_p\overline{Q_q}Q_r
\nonumber
\\
&=
\frac{K}{4}
\sum_{q=1}^n\sum_{r=1}^n\delta_{qr}\p_x^2Q_j\overline{Q_q}Q_r
+
\frac{K}{4}
\sum_{p=1}^n\sum_{q=1}^n
\delta_{pq}
\p_x^2Q_p\overline{Q_q}Q_j
\nonumber
\\
&=
\frac{K}{4}|Q|^2\p_x^2Q_j
+
\frac{K}{4}
\sum_{p=1}^n
\p_x^2Q_p\overline{Q_p}Q_j, 
\nonumber
\\
\sum_{p,q,r=1}^n
S_{p,q,r}^j
\p_xQ_p\overline{\p_xQ_q}Q_r
&=
\frac{K}{4}
\sum_{p,q,r=1}^n
(\delta_{qr}\delta_{pj}+\delta_{pq}\delta_{rj})
\p_xQ_p\overline{\p_xQ_q}Q_r
\nonumber
\\
&=
\frac{K}{4}\sum_{q=1}^n
\overline{\p_xQ_q}Q_q
\p_xQ_j
+\frac{K}{4}|\p_xQ|^2Q_j,
\nonumber
\\
\sum_{p,q,r=1}^n
S_{p,q,r}^j
\p_xQ_p\overline{Q_q}\p_xQ_r
&=
\frac{K}{4}
\sum_{p,q,r=1}^n
(\delta_{qr}\delta_{pj}+\delta_{pq}\delta_{rj})
\p_xQ_p\overline{Q_q}\p_xQ_r
\nonumber
\\
&=
\frac{K}{2}\sum_{q=1}^n
\p_xQ_q\overline{Q_q}
\p_xQ_j,
\nonumber
\\
\sum_{p,q,r=1}^n
S_{p,q,r}^j
Q_p\overline{\p_x^2Q_q}Q_r
&=
\frac{K}{2}
\sum_{q=1}^n
\overline{\p_x^2Q_q}Q_q
Q_j,
\nonumber
\\
\sum_{p,q,r=1}^n
S_{p,q,r}^j
Q_p\overline{Q_q}Q_r
&=
\frac{K}{2}|Q|^2Q_j.
\nonumber
\end{align}
Since $\p_x(S_{p,q,r}^j)=0$ for all $p,q,r,j\in \left\{1,\ldots,n\right\}$, 
it is immediate to see 
$$
\sum_{r=1}^{n}
\left(
\int_{-\infty}^x
f_{j,r}^2(Q,\p_xQ)(t,y)\,dy
\right)Q_r
=0.
$$ 
On the other hand, by a lengthy computation, we can show 
\begin{align}
&\sum_{r=1}^{n}
\left(
\int_{-\infty}^x
f_{j,r}^1(Q,\p_xQ)(t,y)\,dy
\right)Q_r
\nonumber
\\
&=
-\frac{(b+4c)K^2}{16}|Q|^4Q_j
+
\frac{bK^2}{8}
\sum_{r=1}^n
\left(
\int_{-\infty}^{x}
(Q_j\overline{Q_r}\p_x(|Q|^2))(t,y)
dy
\right)
Q_r.
\label{eq:512131}
\end{align}
We demonstrate the computation here. 
Recall that $f_{j,r}^1(Q,\p_xQ)
=
-(b+2c)(S_1+S_2)
+b(S_3+S_4)$, 
where $S_1,\ldots,S_4$ are given by 
\eqref{eq:S1}-\eqref{eq:S4}. 
Obtaining the exact expressions of 
$S_{p,q,r}^jS_{\alpha,\beta,\gamma}^q$
and 
$S_{p,q,r}^jS_{\alpha,\beta,\gamma}^p$ is sufficient to
compute $S_1,\ldots,S_4$, since 
$\overline{S_{p,q,r}^j}=S_{p,q,r}^j$ holds for any $p,q,r,j$ 
in this example. 
(We need to compute them separately, since $S_{p,q,r}^j\neq S_{q,p,r}^j$.)
The result of computation is as follows: 
\begin{align}
S_{p,q,r}^jS_{\alpha,\beta,\gamma}^q
&=\frac{K^2}{16}
(\delta_{qr}\delta_{pj}+\delta_{pq}\delta_{rj})
(\delta_{\beta \gamma}\delta_{\alpha q}+\delta_{\alpha \beta}\delta_{\gamma q})
\nonumber
\\
&=
\frac{K^2}{16}
(\delta_{qr}\delta_{pj}\delta_{\beta \gamma}\delta_{\alpha q}
+\delta_{qr}\delta_{pj}\delta_{\alpha \beta}\delta_{\gamma q}
+\delta_{pq}\delta_{rj}\delta_{\beta \gamma}\delta_{\alpha q}
+\delta_{pq}\delta_{rj}\delta_{\alpha \beta}\delta_{\gamma q}), 
\nonumber
\\
S_{p,q,r}^jS_{\alpha,\beta,\gamma}^p
&=\frac{K^2}{16}
(\delta_{qr}\delta_{pj}+\delta_{pq}\delta_{rj})
(\delta_{\beta \gamma}\delta_{\alpha p}+\delta_{\alpha \beta}\delta_{\gamma p})
\nonumber
\\
&=
\frac{K^2}{16}
(\delta_{qr}\delta_{pj}\delta_{\beta \gamma}\delta_{\alpha p}
+\delta_{qr}\delta_{pj}\delta_{\alpha \beta}\delta_{\gamma p}
+\delta_{pq}\delta_{rj}\delta_{\beta \gamma}\delta_{\alpha p}
+\delta_{pq}\delta_{rj}\delta_{\alpha \beta}\delta_{\gamma p}).
\nonumber
\end{align}
Hence, 
\begin{align}
S_1&=
\frac{K^2}{16}
\Biggl(
\sum_{\beta=1}^{n}
\overline{\p_xQ_r}Q_{\beta}\overline{Q_{\beta}}Q_j
+
\sum_{\beta=1}^{n}
\overline{\p_xQ_{\beta}}Q_{\beta}\overline{Q_{r}}Q_j
\nonumber
\\
&\qquad
\phantom{-i\frac{K^2}{4}}
+\delta_{jr}
\sum_{\alpha=1}^{n}\sum_{\beta=1}^{n}
\overline{\p_xQ_{\alpha}}Q_{\beta}\overline{Q_{\beta}}Q_{\alpha}
+\delta_{jr}
\sum_{\beta=1}^{n}\sum_{\gamma=1}^{n}
\overline{\p_xQ_{\beta}}Q_{\beta}\overline{Q_{\gamma}}Q_{\gamma}
\Biggl)
\nonumber
\\
&=
\frac{K^2}{16}
\Biggl(
\overline{\p_xQ_r}Q_j|Q|^2
+
\overline{Q_{r}}Q_j
\sum_{\beta=1}^{n}
\overline{\p_xQ_{\beta}}Q_{\beta}
+2\delta_{jr}
\sum_{\alpha=1}^{n}
\overline{\p_xQ_{\alpha}}Q_{\alpha}|Q|^2
\Biggl), 
\label{eq:51281}
\\
S_3&=
\frac{K^2}{16}
\Biggl(
\sum_{\beta=1}^{n}
\overline{Q_r}\p_xQ_{\beta}\overline{Q_{\beta}}Q_j
+
\sum_{\beta=1}^{n}
\overline{Q_{\beta}}\p_xQ_{\beta}\overline{Q_{r}}Q_j
\nonumber
\\
&\qquad
\phantom{-i\frac{K^2}{4}}
+\delta_{jr}
\sum_{\alpha=1}^{n}\sum_{\beta=1}^{n}
\overline{Q_{\alpha}}\p_xQ_{\beta}\overline{Q_{\beta}}Q_{\alpha}
+\delta_{jr}
\sum_{\beta=1}^{n}\sum_{\gamma=1}^{n}
\overline{Q_{\beta}}\p_xQ_{\beta}\overline{Q_{\gamma}}Q_{\gamma}
\Biggl)
\nonumber
\\
&=
\frac{K^2}{16}
\Biggl(
2\overline{Q_r}Q_j
\sum_{\beta=1}^{n}
\p_xQ_{\beta}\overline{Q_{\beta}}
+2\delta_{jr}
\sum_{\alpha=1}^{n}
\p_xQ_{\alpha}\overline{Q_{\alpha}}|Q|^2
\Biggl), 
\label{eq:51282}
\\
S_2&=
\frac{K^2}{16}
\Biggl(
\sum_{\beta=1}^{n}
\p_xQ_j\overline{Q_{\beta}}Q_{\beta}\overline{Q_r}
+
\sum_{\beta=1}^{n}
\p_xQ_{\beta}\overline{Q_{\beta}}Q_j\overline{Q_r}
\nonumber
\\
&\qquad
\phantom{-i\frac{K^2}{4}}
+\delta_{jr}
\sum_{\alpha=1}^{n}\sum_{\beta=1}^{n}
\p_xQ_{\alpha}\overline{Q_{\beta}}Q_{\beta}\overline{Q_{\alpha}}
+\delta_{jr}
\sum_{\beta=1}^{n}\sum_{\gamma=1}^{n}
\p_xQ_{\beta}\overline{Q_{\beta}}Q_{\gamma}\overline{Q_{\gamma}}
\Biggl)
\nonumber
\\
&=
\frac{K^2}{16}
\Biggl(
\p_xQ_j\overline{Q_r}|Q|^2
+
Q_j\overline{Q_{r}}
\sum_{\beta=1}^{n}
\p_xQ_{\beta}\overline{Q_{\beta}}
+2\delta_{jr}
\sum_{\alpha=1}^{n}
\p_xQ_{\alpha}\overline{Q_{\alpha}}|Q|^2
\Biggl), 
\label{eq:51291}
\\
S_4&=
\frac{K^2}{16}
\Biggl(
\sum_{\beta=1}^{n}
Q_j\overline{\p_xQ_{\beta}}Q_{\beta}\overline{Q_r}
+
\sum_{\beta=1}^{n}
Q_{\beta}\overline{\p_xQ_{\beta}}Q_j\overline{Q_r}
\nonumber
\\
&\qquad
\phantom{-i\frac{K^2}{4}}
+\delta_{jr}
\sum_{\alpha=1}^{n}\sum_{\beta=1}^{n}
Q_{\alpha}\overline{\p_xQ_{\beta}}Q_{\beta}\overline{Q_{\alpha}}
+\delta_{jr}
\sum_{\beta=1}^{n}\sum_{\gamma=1}^{n}
Q_{\beta}\overline{\p_xQ_{\beta}}Q_{\gamma}\overline{Q_{\gamma}}
\Biggl)
\nonumber
\\
&=
\frac{K^2}{16}
\Biggl(
2\overline{Q_r}Q_j
\sum_{\beta=1}^{n}
\overline{\p_xQ_{\beta}}Q_{\beta}
+2\delta_{jr}
\sum_{\alpha=1}^{n}
\overline{\p_xQ_{\alpha}}Q_{\alpha}|Q|^2
\Biggl).
\label{eq:51210}
\end{align}
Combining them, we have
\begin{align}
S_1+S_2
&=
\frac{K^2}{16}
\left\{
\p_x(Q_j\overline{Q_r})|Q|^2
+
Q_j\overline{Q_{r}}\p_x(|Q|^2)
+2\delta_{jr}
|Q|^2\p_x(|Q|^2)
\right\}
\nonumber
\\
&=
\frac{K^2}{16}
\left\{
\p_x(Q_j\overline{Q_r}|Q|^2)
+\delta_{jr}
\p_x(|Q|^4)
\right\}
\nonumber
\\
&=
\p_x\left\{
\frac{K^2}{16}
\left(
Q_j\overline{Q_r}|Q|^2
+\delta_{jr}|Q|^4
\right)
\right\}, 
\nonumber
\\
S_3+S_4
&=
\frac{K^2}{16}
\Biggl(
2Q_j\overline{Q_r}
\p_x(|Q|^2)
+
2\delta_{jr}
|Q|^2\p_x(|Q|^2)
\Biggl)
\nonumber
\\
&=
\frac{K^2}{16}
\Biggl(
2Q_j\overline{Q_r}
\p_x(|Q|^2)
+\delta_{jr}
\p_x(|Q|^4)
\Biggl)
\nonumber
\\
&=
\p_x\left\{
\frac{K^2}{16}
\delta_{jr}|Q|^4
\right\}
+\frac{K^2}{8}
Q_j\overline{Q_r}
\p_x(|Q|^2).
\nonumber
\end{align}
It follows that 
\begin{align}
\sum_{r=1}^{n}
\left(
\int_{-\infty}^x
(S_1+S_2)(t,y)\,dy
\right)Q_r
&=
\frac{K^2}{16}
\sum_{r=1}^{n}
(Q_j\overline{Q_r}|Q|^2
+\delta_{jr}|Q|^4)
Q_r
\nonumber
\\
&=
\frac{K^2}{8}|Q|^4Q_j, 
\label{eq:512111}
\end{align}
\begin{align}
&\sum_{r=1}^{n}
\left(
\int_{-\infty}^x
(S_3+S_4)(t,y)\,dy
\right)Q_r
\nonumber
\\
&=
\frac{K^2}{16}\sum_{r=1}^{n}
\delta_{jr}|Q|^4Q_r
+\frac{K^2}{8}
\sum_{r=1}^{n}
\left(
\int_{-\infty}^x
\left(Q_j\overline{Q_r}
\p_x(|Q|^2)\right)(t,y)
\,dy
\right)Q_r
\nonumber
\\
&=
\frac{K^2}{16}|Q|^4Q_j
+\frac{K^2}{8}
\sum_{r=1}^{n}
\left(
\int_{-\infty}^x
\left(Q_j\overline{Q_r}
\p_x(|Q|^2)
\right)(t,y)
\,dy
\right)Q_r.
\label{eq:512121}
\end{align}
Combining  
\eqref{eq:512111} and \eqref{eq:512121},  
we get the desired \eqref{eq:512131}. 
\par 
Finally, substituting the result of computation into \eqref{eq:r51251}, 
we arrived at 
\begin{align}
&\sqrt{-1}\p_tQ_j+(a\p_x^4+\lambda \p_x^2)Q_j
\nonumber
\\
&=\frac{K}{4}
   d_1
   \left(
   |Q|^2\p_x^2Q_j
   +
   \sum_{r=1}^n
   \p_x^2Q_r\overline{Q_r}Q_j
   \right)
       +\frac{K}{2}
       d_2
       \sum_{r=1}^n
       \overline{\p_x^2Q_r}Q_r
       Q_j
    \nonumber
   \\
  &\quad
    +\frac{K}{4}
    d_3
    \left(
    \sum_{r=1}^n
    \overline{\p_xQ_r}Q_r
    \p_xQ_j
    +|\p_xQ|^2Q_j
    \right)
    +\frac{K}{2}
    d_4
    \sum_{r=1}^n
    \p_xQ_r\overline{Q_r}
    \p_xQ_j
    -\frac{K}{2}\lambda
    |Q|^2Q_j
    \nonumber
    \\
    &\quad
  -\frac{(b+4c)K^2}{16}|Q|^4Q_j
  +\frac{bK^2}{8}
  \sum_{r=1}^n
  \left(
  \int_{-\infty}^{x}
  (Q_j\overline{Q_r}\p_x(|Q|^2))(t,y)
  dy
  \right)
  Q_r.  
  \label{eq:512132}
\end{align}
\begin{remark}
\label{remark:0802}
If $n=1$, then the final nonlocal term 
of the right hand of \eqref{eq:512132} simply becomes a local one, 
in that 
\begin{align}
\left(
    \int_{-\infty}^{x}
    (Q_1\overline{Q_1}\p_x(|Q_1|^2))(t,y)
    dy
    \right)
    Q_1
&=\frac{1}{2}
\left(
    \int_{-\infty}^{x}
    (\p_x(|Q_1|^4))(t,y)
    dy
    \right)
    Q_1
\nonumber
\\
&=
\frac{1}{2}|Q_1|^4Q_1. 
\nonumber
\end{align}
In this case, the derived equation   
\eqref{eq:512132} for $Q=Q_1$ turns out to coincide with \eqref{eq:ns51261} where $\kappa(u)\equiv K$.  
This does not conflict with the fact that 
the holomorphic sectional curvature for the Riemann surface
coincides with the Gaussian curvature.
\end{remark}
\section{Example (3)}
\label{section:Example3}
We investigate the case $(N,J,h)$ is a compact complex Grassmannian as a Hermitian 
symmetric space. 
Looking at many famous literature(e.g., 
\cite{Arvanitoyeorgos}, 
\cite{helgason}, 
\cite{KN}, 
\cite{oneill}, 
\cite{petersen} 
), 
there are some models of complex Grassmannians. 
To avoid the confusion, following \cite{BZA, Dimitric} mainly, 
we start from stating the setting we use in this paper.
\subsection{Setting of complex Grassmannians}
\label{subsection:PreGr}
Fix integers $n_0, k_0$ with $1\leqslant k_0<n_0$, and set $m_0=n_0-k_0$.
Let $N$ be the complex Grassmannian $G_{n_0, k_0}$ defined as the set of 
all $k_0$-dimensional linear subspaces through the origin of 
the complex Euclidean space $\mathbb{C}^{n_0}$. 
With a slight abuse of notation, this can be identified with the 
set of Hermitian rank-$k_0$ projectors:  
\begin{equation}
G_{n_0,k_0}
=
\left\{
A\in H(n_0)\mid 
A^2=A \ \text{and} \ \text{rank}A=k_0 
\right\}, 
\label{eq:Gnk}
\end{equation}
where 
$H(n_0)=\left\{
A\in \mathcal{M}_{n_0\times n_0}
\mid 
A^{*}=A
\right\}$ being the set of Hermitian matrices. 
($\mathcal{M}_{n_0\times n_0}$ denotes the space of 
$n_0\times n_0$ complex-matrices and  
$A^{*}=\overline{A}^{t}$ denotes the conjugate transpose of $A$.)
\par 
Set $U(n_0)=\left\{
B\in \mathcal{M}_{n_0\times n_0}\mid 
B^{*}B=BB^{*}=I
\right\}$ to denote the unitary group of degree $n_0$.  
(In this section, the identity matrix of size $n_0$ is denoted by $I$ and 
the identity matrices of size $k\in \left\{1,\ldots, n_0-1\right\}$ 
are by $I_k$.)
Then $U(n_0)$ is a compact Lie group and  
the Lie algebra consists of the set of skew-Hermitian matrices: 
$$
\mathfrak{u}(n_0):=T_{I}U(n_0)
=
\left\{
\Omega\in \mathcal{M}_{n_0\times n_0}\mid 
\Omega^{*}=-\Omega
\right\}.
$$
Let us define an isometric group action of $U(n_0)$ on $H(n_0)$ by 
$$
\Phi: U(n_0)\times H(n_0)\to H(n_0), 
\quad (B,H)\mapsto BHB^{*}, 
$$
and take 
$$
A_0:=\begin{pmatrix}
I_{k_0} & 0 \\
0 & 0 \\
\end{pmatrix}
\in G_{n_0,k_0}
$$
as the origin of $G_{n_0,k_0}$. 
(Here and hereafter, all matrices in $\mathcal{M}_{n_0\times n_0}$ 
are written as a block form where the submatrix in the upper left corner is of 
order $k_0\times k_0$, and the four zero-matrices as the submatrix are simply 
denoted by $0$.)   
The same argument as that in \cite[Section~2.1]{BZA} shows 
$G_{n_0,k_0}=\Phi(U(n_0),A_0)$ being the orbit of $A_0$ under $\Phi$. 
Moreover, $\Phi$ is a transitive action of $U(n_0)$ on $G_{n_0,k_0}$ 
and the isotropy group at $A_0$ is  
$$
\left\{
\begin{pmatrix}
K_1 & 0 \\
0 & K_2 \\
\end{pmatrix} \middle|
K_1\in U(k_0), K_2\in U(m_0)
\right\}
\cong 
U(k_0)\times U(m_0). 
$$
In fact, 
$G_{n_0,k_0}$ is diffeomorphic to $U(n_0)/(U(k_0)\times U(m_0))$ 
via the canonical map
$G_{n_0,k_0}\ni \Phi(H,A_0)
\mapsto 
H(U(k_0)\times U(m_0))
\in U(n_0)/(U(k_0)\times U(m_0))
$
and is an embedded submanifold of $H(n_0)$ satisfying  
$$
\dim_{\RR}G_{n_0,k_0}
=
\dim_{\RR}U(n_0)
-\dim_{\RR}(U(k_0)\times U(m_0))=2k_0m_0
$$
which implies $n=\dim_{\mathbb{C}}G_{n_0,k_0}=k_0m_0$. 
In addition, the involution which is given by 
$$
\sigma:U(n_0)\to U(n_0),
\quad 
B\mapsto 
\begin{pmatrix}
I_{k_0} & 0 \\
0 & -I_{m_0} \\
\end{pmatrix}
B
\begin{pmatrix}
I_{k_0} & 0 \\
0 & -I_{m_0} \\
\end{pmatrix}^{-1}
$$
makes $U(n_0)/(U(k_0)\times U(m_0))$  
symmetric.  
\par 
Next,  
set $\pi=\Phi(\cdot,A_0)$, that is, 
$$
\pi:U(n_0)\to G_{n_0,k_0}, 
\quad 
B\mapsto BA_0B^{*}.
$$  
For any $A\in G_{n_0,k_0}$, 
there exists $B\in U(n_0)$ such that $A=\pi(B)$ and  
the tangent space of $G_{n_0,k_0}$ at $A\in G_{n_0,k_0}$
can be expressed by     
\begin{align}
T_AG_{n_0,k_0}
&=
\left\{
B
\begin{pmatrix}
0 & V \\
V^{*} & 0 \\
\end{pmatrix} 
B^{*}
\middle| 
V\in \mathcal{M}_{k_0\times m_0}
\right\}. 
\label{eq:TGr1}
\end{align}
This follows from the same argument as that in \cite[Section~2.1]{BZA}. 
To see this more concretely, note that 
the tangent space of $U(n_0)$ at an arbitrary $B\in U(n_0)$ 
is given by the left translation of $\mathfrak{u}(n_0)$, 
\begin{equation}
T_{B}U(n_0)
=
\left\{
B\Omega\in \mathcal{M}_{n_0\times n_0}
\mid 
\Omega\in \mathfrak{u}(n_0)
\right\}.
\label{eq:TU}
\end{equation}
It turns out that the differential $(d\pi)_B: T_BU(n_0)\to T_{\pi(B)}G_{n_0,k_0}$ 
at $B\in U(n_0)$ 
is given by
\begin{align}
&(d\pi)_B
\left(
B\begin{pmatrix}
\omega_{11} & -\omega_{12} \\
(\omega_{12})^{*} & \omega_{22} \\
\end{pmatrix} 
\right)
=
B\begin{pmatrix}
0 & \omega_{12} \\
(\omega_{12})^{*} & 0 \\
\end{pmatrix} 
B^{*}
\label{eq:237011}
\end{align}
for all 
$\begin{pmatrix}
\omega_{11} & -\omega_{12} \\
(\omega_{12})^{*} & \omega_{22} \\
\end{pmatrix}\in \mathfrak{u}(n_0)$.
Since $\pi$ is submersion, \eqref{eq:TGr1} is obtained.
\par 
The complex structure $J_A$ at the point $A=\pi(B)\in G_{n_0,k_0}$ is given by 
\begin{equation}
J_A:T_AG_{n_0,k_0}\to T_AG_{n_0,k_0}, 
\quad 
B
\begin{pmatrix}
0 & V \\
V^{*} & 0 \\
\end{pmatrix} 
B^{*}
\mapsto 
B
\begin{pmatrix}
0 & \sqrt{-1}V \\
(\sqrt{-1}V)^{*} & 0 \\
\end{pmatrix} 
B^{*}. 
\label{eq:comst}
\end{equation}
\par 
The Riemannian metric on $G_{n_0,k_0}$ is 
taken to be $U(n_0)$-invariant by the following 
standard manner: 
We take an Ad-invariant metric $\langle\cdot,\cdot\rangle$ on $\mathfrak{u}(n_0)$ 
which is defined by 
\begin{equation}
\langle \Omega_1,\Omega_2\rangle
=\dfrac{1}{2}\mathrm{tr}\left(
\Omega_1 (\Omega_2)^{*}
\right), 
\label{eq:Ad}
\end{equation}
and define $\langle \cdot, \cdot \rangle_B$ for each $B\in U(n_0)$ by
\begin{equation}
\langle
B\Omega_1, B\Omega_2
\rangle_B
=
\langle \Omega_1,\Omega_2\rangle
\label{eq:metU}
\end{equation}
for any $B\Omega_1, B\Omega_2\in T_BU(n_0)$, 
which gives a bi-invariant Riemannian metric on $U(n_0)$. 
Let $A=\pi(B)\in G_{n_0,k_0}$ 
and let $\Delta_i\in T_AG_{n_0,k_0}$ ($i=1, 2$). 
By \eqref{eq:TGr1} and \eqref{eq:237011},  
there exist $\Omega_i\in \mathfrak{u}(n_0)$ ($i=1, 2$) such that
$$
\Delta_i
=
(d\pi)_B\left(
B\Omega_i \right), 
\quad 
\Omega_i=\begin{pmatrix}
0 & -\omega_i \\
(\omega_i)^{*} & 0 \\
\end{pmatrix},
\quad 
\omega_i\in 
\mathcal{M}_{k_0\times m_0}
\quad 
(i=1,2).
$$ 
We define $h_A(\Delta_1,\Delta_2)$ by 
\begin{align}
h_A(\Delta_1,\Delta_2)
&=
\langle
B\Omega_1, B\Omega_2
\rangle_B
\left(
=
\frac{1}{2}\mathrm{tr}\left(\Omega_1(\Omega_2)^{*}\right)
\right).
\nonumber
\end{align}
Then, $h=\left\{h_A\right\}$ is a $U(n_0)$-invariant Riemannian metric on $G_{n_0,k_0}$.
Furthermore, by the fundamental properties of the trace for complex-component matrices,  
\begin{align}
h_A(\Delta_1,\Delta_2)
&=
\frac{1}{2}\mathrm{tr}\left(\begin{pmatrix}
0 & -\omega_1 \\
(\omega_1)^{*} & 0 \\
\end{pmatrix}
\begin{pmatrix}
0 & -\omega_2 \\
(\omega_2)^{*} & 0 \\
\end{pmatrix}^{*}\right)
=
\Re\left[
\mathrm{tr}(\omega_1(\omega_2)^{*})
\right].
\label{eq:metN}
\end{align}
\begin{remark}
The metric $h$ is the same as that used in, e.g., 
\cite{KN,oneill,petersen}. 
It is also the same as that in \cite{DWW2003,SW2011,DW2018}
up to a constant multiple. 
To investigate the expression of \eqref{eq:pde}, 
we do not need to be aware of the difference of the constant multiple, 
because the Levi-Civita connection and $R$ as a $(1,3) $-tensor 
used to formulate \eqref{eq:pde}
are invariant  under the homothetic change $h\to ch$($c$ is a positive constant). 
In other words,   
although the holomorphic sectional curvature is multiplied by $1/c$, 
the derived final expression \eqref{eq:r51251} is not changed. 
\end{remark}
\begin{remark}
\label{remark:decom}
For each $B\in U(n_0)$, 
$T_BU(n_0)$ can be decomposed into the kernel of the differential $(d\pi)_B$ and 
the orthogonal complement with respect to $\langle\cdot,\cdot\rangle_B$: 
\begin{align}
T_BU(n_0)&=
\mathfrak{k}_B+\mathfrak{m}_B, 
\label{eq:236271}
\end{align}
where 
\begin{align}
\mathfrak{k}_B&:=\mathrm{Ker}((d\pi)_B)
=\left\{
B\begin{pmatrix}
\omega_{11} & 0 \\
0 & \omega_{22} \\
\end{pmatrix} 
\middle|
\omega_{11}\in \mathfrak{u}(k_0), 
\omega_{22}\in \mathfrak{u}(m_0)
\right\},
\label{eq:236272}
\\
\mathfrak{m}_B&:=(\mathfrak{k}_B)^{\perp}
=
\left\{
B\begin{pmatrix}
0 & \omega_{12} \\
-(\omega_{12})^{*} & 0 \\
\end{pmatrix} 
\middle|
\omega_{12}\in 
\mathcal{M}_{k_0\times m_0}
\right\}.
\label{eq:236273}
\end{align}
Comparing \eqref{eq:TGr1} and  \eqref{eq:236273}, 
we see that 
the tangent space $T_AG_{n_0,k_0}$ at $A=\pi(B)$ can be identified with 
$\mathfrak{m}_B$ by the map 
$((d\pi)_B)\lvert_{\mathfrak{m}_B}:\mathfrak{m}_B\to T_AG_{n_0,k_0}$.  
In addition, a direct computation using \eqref{eq:236272} and \eqref{eq:236273} 
easily shows that $\mathfrak{u}(n_0)$ is a symmetric Lie algebra, that is,  
\begin{align}
&\left[
\mathfrak{k}_I, 
\mathfrak{k}_I
\right]
\subset 
\mathfrak{k}_I, 
\quad
\left[
\mathfrak{k}_I, 
\mathfrak{m}_I
\right]
\subset 
\mathfrak{m}_I, 
\quad
\left[
\mathfrak{m}_I, 
\mathfrak{m}_I
\right]
\subset 
\mathfrak{k}_I.
\nonumber
\end{align}
This does not conflict with the fact that $U(n_0)/(U(k_0)\times U(m_0))$ is a symmetric space 
with involution $\sigma$(see, e,g, \cite[Proposition~6.4]{Arvanitoyeorgos}). 
\end{remark}
\begin{remark}
Some other equivalent expressions of the tangent space are known.  
For example, as is used in \cite{Dimitric}, 
the following implicit expression also holds:
\begin{align}
T_AG_{n_0,k_0}
&=
\left\{
H\in H(n_0)
\middle|
HA+AH=H
\right\}.
\label{eq:TGr2}
\end{align}
Indeed, by definition of $H(n_0)$ and \eqref{eq:Gnk}, 
the right hand side of \eqref{eq:TGr2} 
turns out to coincide with that of \eqref{eq:TGr1}.  
\end{remark}
The Riemann curvature tensor $R$ at $A\in G_{n_0,k_0}$ 
is given by the following: 
\begin{align}
(R(X,Y)Z)(A)
&=
\left[
\left[X,Y\right], Z
\right]
\quad 
(X,Y,Z\in T_AG_{n_0,k_0}), 
\label{eq:Rgr}
\end{align}   
where $\left[\cdot,\cdot\right]$ is the bracket of the matrices 
defined by $\left[X_1,X_2\right]=X_1X_2-X_2X_1$. 
The above expression is derived in \cite{Dimitric} by using \eqref{eq:TGr2}. 
(As is commented in \cite{Dimitric}, 
the expression of $R$ differs by a sign from the familiar one
 (e.g., \cite{Arvanitoyeorgos,FK,helgason,jost,KN,petersen}), 
since our tangent vectors are Hermitian rather than skew-Hermitian. 
See Remark~\ref{remark:signR} also.)   
Let $A=\pi(B)\in G_{n_0,k_0}$, 
and let $X,Y,Z,W\in T_AG_{n_0,k_0}$ be expressed by 
\begin{align}
X&=
B\begin{pmatrix}
0 & x \\
x^{*} & 0 \\
\end{pmatrix}B^{*}, 
Y=
B\begin{pmatrix}
0 & y \\
y^{*} & 0 \\
\end{pmatrix}B^{*}, 
Z=
B\begin{pmatrix}
0 & z \\
z^{*} & 0 \\
\end{pmatrix}B^{*}, 
W=
B\begin{pmatrix}
0 & w \\
w^{*} & 0 \\
\end{pmatrix}B^{*}, 
\label{eq:XYZ}
\end{align}
where $x,y,z,w\in \mathcal{M}_{k_0\times m_0}$. (The notation $x$ is not a 
variable of functions only here.) 
Then, the substitution of them into \eqref{eq:Rgr} shows
\begin{align}
(R(X,Y)Z)(A)
&=
B
\left[
\left[
\begin{pmatrix}
0 & x \\
x^{*} & 0 \\
\end{pmatrix}, 
\begin{pmatrix}
0 & y \\
y^{*} & 0 \\
\end{pmatrix}
\right], 
\begin{pmatrix}
0 & z \\
z^{*} & 0 \\
\end{pmatrix}
\right]
B^{*}
=
B
S
B^{*}, 
\label{eq:Rgr2}
\end{align}
where $S\in T_{A_0}G_{n_0,k_0}$ is determined by
\begin{equation}
S
=\begin{pmatrix}
0 & s \\
s^{*} & 0 \\
\end{pmatrix}, 
\quad
s=xy^{*}z-yx^{*}z-zx^{*}y+zy^{*}x
\ (\in \mathcal{M}_{k_0\times m_0}).
\label{eq:Rgr3}
\end{equation} 
This combined with \eqref{eq:metN} gives 
\begin{align}
h_A(R(X,Y)Z,W)
&=
\operatorname{Re}
(\mathrm{tr}(sw^{*}))
\nonumber
\\
&=
\operatorname{Re}
(\mathrm{tr}(xy^{*}zw^{*}-yx^{*}zw^{*}-zx^{*}yw^{*}+zy^{*}xw^{*})).
\label{eq:hR}
\end{align} 
\begin{remark}
\label{remark:signR}
The Riemann curvature tensor 
at $A_0$ can be rewritten  
via the identification 
$$
((d\pi)_{I})\lvert_{\mathfrak{m}_I}:\mathfrak{m}_I
\to T_{A_0}G_{n_0,k_0}, 
\quad 
\begin{pmatrix}
0 & -\omega \\
\omega^{*} & 0 \\
\end{pmatrix}
\mapsto 
\begin{pmatrix}
0 & \omega \\
\omega^{*} & 0 \\
\end{pmatrix}.
$$
To see this,  let $\iota:T_{A_0}G_{n_0,k_0}\to \mathfrak{m}_I$ be the inverse of $((d\pi)_{I})\lvert_{\mathfrak{m}_I}$. 
Then, for any  
$$
X=\begin{pmatrix}
0 & x \\
x^{*} & 0 \\
\end{pmatrix}, 
Y=\begin{pmatrix}
0 & y \\
y^{*} & 0 \\
\end{pmatrix}, 
Z=\begin{pmatrix}
0 & z \\
z^{*} & 0 \\
\end{pmatrix}\in T_{A_0}G_{n_0,k_0}, 
$$
the following relation holds:
\begin{align}
(R(X,Y)Z)(A_0)
&=
-((d\pi)_{I})\lvert_{\mathfrak{m}_I}
\left[
\left[\iota(X), \iota(Y)
\right], 
\iota(Z)
\right]. 
\label{eq:Rgr4}
\end{align}
This does not conflict with \cite{Arvanitoyeorgos,FK,helgason,jost,KN,petersen} 
where $R$ is expressed by 
``$R(X,Y)Z=-[[X,Y],Z]$" 
in the context of the right hand side of \eqref{eq:Rgr4} via the above identification.  
\end{remark}
\subsection{Computation of \eqref{eq:r51251}}
\label{subsection:comGr}
Let $N=G_{n_0,k_0}$ as above.
We compute \eqref{eq:r51251} in Theorem~\ref{theorem:momo}. 
Recall that $n=k_0m_0$ is the complex dimension of $G_{n_0,k_0}$.
For any $j\in \left\{1,\ldots,n\right\}$, 
there exists a unique pair of integers $j_1\in \left\{1, \ldots, k_0\right\}$ and 
$j_2\in \left\{1, \ldots, m_0\right\}$ such that 
$j=(j_2-1)k_0+j_1$. 
In what follows, we denote it by $j=(j_1,j_2)$.  
\par 
Since $u\in C_{u^{\infty}}((-T,T)\times \RR;N)$ in Theorem~\ref{theorem:momo}, 
there exist $B^{\infty}\in U(n_0)$ and 
$B=B(t,x): (-T,T)\times \RR \to U(n_0)$ such that
$u^{\infty}=B^{\infty}A_0(B_{\infty})^{*}$ and 
$u(t,x)=B(t,x)A_0(B(t,x))^{*}$. 
We take $e_j^{\infty}\in T_{u^{\infty}}N$ 
for each $j=(j_1,j_2)\in \left\{1,\ldots,n\right\}$
to satisfy 
\begin{equation}
e_j^{\infty}
=B_{\infty}
\begin{pmatrix}
0 & E_j \\
(E_j)^{*} & 0 \\
\end{pmatrix}
(B_{\infty})^{*}, 
\label{eq:base_j}
\end{equation}
where 
$E_j=E_{(j_1,j_2)}\in \mathcal{M}_{k_0\times m_0}$ denotes 
a constant matrix with entry $1$ where the $j_1$-th row and the 
$j_2$-th column meet, and all other entries being $0$. 
It is easy to see 
$E_{(j_1,j_2)}(E_{(\ell_1, \ell_2)})^{*}
=\delta_{j_2 \ell_2} E_{j_1,\ell_1}^{(k_0)}$ 
where $E_{j_1,\ell_1}^{(k_0)}\in 
\mathcal{M}_{k_0\times k_0}$ denotes
a square matrix with entry $1$ where the $j_1$-th row and the 
$\ell_1$-th column meet, and all other entries being $0$.
This combined with \eqref{eq:metN} shows 
$$
h_{u^{\infty}}(e_j^{\infty},e_{\ell}^{\infty})
=\operatorname{Re}\left[
\mathrm{tr}(E_j(E_{\ell})^{*})
\right]
=
\delta_{j_2 \ell_2}
\delta_{j_1 \ell_1}
=\delta_{j\ell}
\quad 
(\forall  j,\ell\in\left\{1,\ldots,n\right\}).
$$
Therefore, $\left\{e_j^{\infty}, Je_j^{\infty}\right\}_{j=1}^n$ 
is actually an orthonormal basis for $T_{u^{\infty}}N$. 
Let $\left\{e_j, Je_j\right\}_{j=1}^n$ be the associated orthonormal frame 
for $u^{-1}TN$ that satisfies \eqref{eq:mf01}-\eqref{eq:mfini}. 
The parallelity of $R$ and $J$ with respect to $\nabla$ 
shows that $e_j$, $J_ue_j$, $R(e_p,e_q)e_r$, and $R(e_p,J_ue_q)e_r$ 
are respectively the parallel displacement of 
$e_j^{\infty}$, $J_ue_j^{\infty}$, $R(e_p^{\infty},e_q^{\infty})e_r^{\infty}$, 
and $R(e_p^{\infty}, J_ue_q^{\infty})e_r^{\infty}$. 
In addition, $h$ is invariant under the parallel displacement.  
Therefore, 
the expression \eqref{eq:add11171} reduces to 
\begin{align}
\,S_{p,q,r}^j
&=
\frac{1}{2}
\left\{h(R(e_p^{\infty},e_q^{\infty})e_r^{\infty},e_j^{\infty})
+\sqrt{-1}h(R(e_p^{\infty},e_q^{\infty})e_r^{\infty},J_ue_j^{\infty})
\right\}
\nonumber
\\
&\quad
+
\frac{\sqrt{-1}}{2}
\left\{
h(R(e_p^{\infty},J_ue_q^{\infty})e_r^{\infty},e_j^{\infty})
+\sqrt{-1}h(R(e_p^{\infty},J_ue_q^{\infty})e_r^{\infty},J_ue_j^{\infty})
\right\}. 
\nonumber
\end{align} 
Here, we apply \eqref{eq:hR} for $A=u^{\infty}$
to deduce
\begin{align}
h(R(e_p^{\infty},e_q^{\infty})e_r^{\infty},e_j^{\infty})
&=
\operatorname{Re}
(\mathrm{tr}(\Xi_1)),
\nonumber
\end{align}
where
\begin{align}
\Xi_1&=
E_p (E_q)^{*}E_r(E_j)^{*}-E_q(E_p)^{*}E_r(E_j)^{*}
-E_r(E_p)^{*}E_q(E_j)^{*}+E_r(E_q)^{*}E_p(E_j)^{*}).
\nonumber
\end{align}
Moreover, noting 
\begin{equation}
J_ue_j^{\infty}
=B_{\infty}
\begin{pmatrix}
0 & \sqrt{-1}E_j \\
(\sqrt{-1}E_j)^{*} & 0 \\
\end{pmatrix}
(B_{\infty})^{*}, 
\nonumber
\end{equation}
we repeat the above computation replacing $E_j$ with $\sqrt{-1}E_j$, 
which provides
\begin{align}
h(R(e_p^{\infty},e_q^{\infty})e_r^{\infty}, J_ue_j^{\infty})
&=
\operatorname{Re}
(-\sqrt{-1}\mathrm{tr}(\Xi_1))
=
\operatorname{Im}
(\mathrm{tr}(\Xi_1)).
\nonumber
\end{align}
In the same way as above, we deduce 
\begin{align}
h(R(e_p^{\infty},J_ue_q^{\infty})e_r^{\infty},e_j^{\infty})
&=\operatorname{Im}
(\mathrm{tr}(\Xi_2)), 
\nonumber
\end{align}
where 
\begin{align}
\Xi_2&=
E_p (E_q)^{*}E_r(E_j)^{*}+E_q(E_p)^{*}E_r(E_j)^{*}
+E_r(E_p)^{*}E_q(E_j)^{*}+E_r(E_q)^{*}E_p(E_j)^{*}), 
\nonumber
\end{align}
and 
\begin{align}
h(R(e_p^{\infty},J_ue_q^{\infty})e_r^{\infty}, J_ue_j^{\infty})
&=
\operatorname{Im}
(-\sqrt{-1}\mathrm{tr}(\Xi_2))
=
-\operatorname{Re}
(\mathrm{tr}(\Xi_2)).
\nonumber
\end{align}
Combining them, we obtain 
\begin{align}
S_{p,q,r}^j
&=
\frac{1}{2}(\mathrm{tr}(\Xi_1)+\mathrm{tr}(\Xi_2))
=
\mathrm{tr}\left(
E_p (E_q)^{*}E_r(E_j)^{*}+E_r(E_q)^{*}E_p(E_j)^{*}
\right). 
\label{eq:Smomo}
\end{align}
Furthermore, set  
$p=(p_1,p_2)$, $q=(q_1,q_2)$, and $r=(r_1,r_2)$
where $p_1,q_1,r_1\in \left\{1,\ldots,k_0\right\}$ and 
$p_2,q_2,r_2\in \left\{1,\ldots,m_0\right\}$. 
A simple computation yields 
\begin{align}
E_p (E_q)^{*}E_r(E_j)^{*}
&=
\delta_{p_2q_2}E^{(k_0)}_{p_1,q_1}
\delta_{r_2j_2}E^{(k_0)}_{r_1,j_1}
=
\delta_{p_2q_2}\delta_{r_2j_2}
\delta_{q_1r_1}E^{(k_0)}_{p_1,j_1}, 
\end{align}
and thus
$$
\mathrm{tr}\left(E_p (E_q)^{*}E_r(E_j)^{*}\right)
=
\delta_{p_2q_2}\delta_{r_2j_2}
\delta_{q_1r_1}
\delta_{p_1j_1}.
$$
Hence, 
for any $p=(p_1,p_2), q=(q_1,q_2), r=(r_1,r_2), j=(j_1,j_2)\in \left\{1,\ldots,n\right\}$, 
\begin{align}
S_{p,q,r}^j
&=
\delta_{p_2q_2}\delta_{r_2j_2}
\delta_{q_1r_1}
\delta_{p_1j_1}
+
\delta_{r_2q_2}\delta_{p_2j_2}
\delta_{q_1p_1}
\delta_{r_1j_1}. 
\label{eq:Smomo2}
\end{align}
\par 
Based on \eqref{eq:Smomo2}, we proceed the computation of \eqref{eq:r51251}. 
First, we compute the form 
\begin{align}
\sum_{p, q,r=1}^n
S_{p,q,r}^jX_p\overline{Y_q}Z_r
&=
\sum_{p_1, q_1,r_1=1}^{k_0}\sum_{p_2, q_2,r_2=1}^{m_0}
S_{(p_1,p_2),(q_1,q_2),(r_1,r_2)}^{(j_1,j_2)}
X_{(p_1,p_2)}\overline{Y_{(q_1,q_2)}}Z_{(r_1,r_2)}, 
\nonumber
\end{align}
where $X_p, Y_q,Z_r$ for $p,q,r\in \left\{1,\ldots,n\right\}$ 
denote complex-valued functions of $(t,x)$. 
The right hand side of above is divided by the sum of $L_1$ and $L_2$: 
\begin{align}
L_1&=
\sum_{p_1, q_1,r_1=1}^{k_0}\sum_{p_2, q_2,r_2=1}^{m_0}
\delta_{p_2q_2}\delta_{r_2j_2}
\delta_{q_1r_1}
\delta_{p_1j_1}
X_{(p_1,p_2)}\overline{Y_{(q_1,q_2)}}Z_{(r_1,r_2)}, 
\nonumber
\\
L_2&=
\sum_{p_1, q_1,r_1=1}^{k_0}\sum_{p_2, q_2,r_2=1}^{m_0}
\delta_{r_2q_2}\delta_{p_2j_2}
\delta_{q_1p_1}
\delta_{r_1j_1}
X_{(p_1,p_2)}\overline{Y_{(q_1,q_2)}}Z_{(r_1,r_2)}.
\end{align}
By a simple computation, 
\begin{align}
L_1&=
\sum_{q_1,r_1=1}^{k_0}
\sum_{p_2, q_2=1}^{m_0}
\delta_{p_2q_2}
\delta_{q_1r_1}
X_{(j_1,p_2)}\overline{Y_{(q_1,q_2)}}Z_{(r_1,j_2)}
\nonumber
\\
&=
\sum_{r_1=1}^{k_0}
\sum_{p_2=1}^{m_0}
X_{(j_1,p_2)}\overline{Y_{(r_1,p_2)}}Z_{(r_1,j_2)},
\nonumber
\end{align}
\begin{align}
L_2&=
\sum_{p_1, q_1=1}^{k_0}
\sum_{q_2,r_2=1}^{m_0}
\delta_{r_2q_2}
\delta_{q_1p_1}
X_{(p_1,j_2)}\overline{Y_{(q_1,q_2)}}Z_{(j_1,r_2)}
\nonumber
\\
&=
\sum_{p_1=1}^{k_0}
\sum_{r_2=1}^{m_0}
Z_{(j_1,r_2)}\overline{Y_{(p_1,r_2)}}X_{(p_1,j_2)}. 
\nonumber
\end{align}
Combining them, we obtain 
\begin{align}
&\sum_{p, q,r=1}^n
S_{p,q,r}^jX_p\overline{Y_q}Z_r
\nonumber
\\
&=
\sum_{s_2=1}^{k_0}
\sum_{s_1=1}^{m_0}
X_{(j_1,s_1)}\overline{Y_{(s_2,s_1)}}Z_{(s_2,j_2)}
+
\sum_{s_2=1}^{k_0}
\sum_{s_1=1}^{m_0}
Z_{(j_1,s_1)}\overline{Y_{(s_2,s_1)}}X_{(s_2,j_2)}. 
\label{eq:37071}
\end{align}
Applying \eqref{eq:37071}, we have 
\begin{align}
&\sum_{p, q,r=1}^n
S_{p,q,r}^j
\p_x^2Q_p\overline{Q_q}Q_r
\nonumber
\\
&=
\sum_{s_2=1}^{k_0}
\sum_{s_1=1}^{m_0}
\p_x^2Q_{(j_1,s_1)}\overline{Q_{(s_2,s_1)}}Q_{(s_2,j_2)}
+
\sum_{s_2=1}^{k_0}
\sum_{s_1=1}^{m_0}
Q_{(j_1,s_1)}
\overline{Q_{(s_2,s_1)}}
\p_x^2Q_{(s_2,j_2)}, 
\label{eq:37072}
\\
&\sum_{p, q,r=1}^n
S_{p,q,r}^j
Q_p\overline{\p_x^2Q_q}Q_r
\nonumber
\\
&=
\sum_{s_2=1}^{k_0}
\sum_{s_1=1}^{m_0}
Q_{(j_1,s_1)}\overline{\p_x^2Q_{(s_2,s_1)}}Q_{(s_2,j_2)}
+
\sum_{s_2=1}^{k_0}
\sum_{s_1=1}^{m_0}
Q_{(j_1,s_1)}
\overline{\p_x^2Q_{(s_2,s_1)}}Q_{(s_2,j_2)}
\nonumber
\\
&=
2\sum_{s_2=1}^{k_0}
\sum_{s_1=1}^{m_0}
Q_{(j_1,s_1)}\overline{\p_x^2Q_{(s_2,s_1)}}Q_{(s_2,j_2)}, 
\label{eq:37073}
\\
&\sum_{p, q,r=1}^n
S_{p,q,r}^j
\p_xQ_p\overline{\p_xQ_q}Q_r
\nonumber
\\
&=
\sum_{s_2=1}^{k_0}
\sum_{s_1=1}^{m_0}
\p_xQ_{(j_1,s_1)}\overline{\p_xQ_{(s_2,s_1)}}Q_{(s_2,j_2)}
+
\sum_{s_2=1}^{k_0}
\sum_{s_1=1}^{m_0}
Q_{(j_1,s_1)}
\overline{\p_xQ_{(s_2,s_1)}}\p_xQ_{(s_2,j_2)}, 
\label{eq:37074}
\\
&\sum_{p, q,r=1}^n
S_{p,q,r}^j
\p_xQ_p\overline{Q_q}\p_xQ_r
=
2
\sum_{s_2=1}^{k_0}
\sum_{s_1=1}^{m_0}
\p_xQ_{(j_1,s_1)}\overline{Q_{(s_2,s_1)}}\p_xQ_{(s_2,j_2)}, 
\label{eq:37075}
\\
&\sum_{p, q,r=1}^n
S_{p,q,r}^j
Q_p\overline{Q_q}Q_r
=
2\sum_{s_2=1}^{k_0}
\sum_{s_1=1}^{m_0}
Q_{(j_1,s_1)}\overline{Q_{(s_2,s_1)}}Q_{(s_2,j_2)}. 
\label{eq:37076}
\end{align}
Next, we compute 
$$
\sum_{r=1}^{n}
\left(
\int_{-\infty}^x
f_{j,r}^1(Q,\p_xQ)(t,y)\,dy
\right)Q_r, 
$$
where $f_{j,r}^1(Q,\p_xQ)=-(b+2c)(S_1+S_2)+b(S_3+S_4)$ 
and 
$S_1,\ldots S_4$ are defined by \eqref{eq:S1}-\eqref{eq:S4}. 
By \eqref{eq:Smomo2}, we see
\begin{align}
&S_{p,q,r}^j\overline{S_{\alpha,\beta,\gamma}^q}
=S_{p,q,r}^j S_{\alpha,\beta,\gamma}^q
=
S_{(p_1,p_2),(q_1,q_2),(r_1,r_2)}^{(j_1,j_2)}
S_{(\alpha_1,\alpha_2),(\beta_1,\beta_2),(\gamma_1,\gamma_2)}^{(q_1,q_2)}
\nonumber
\\
&=
\delta_{p_2q_2}
\delta_{r_2j_2}
\delta_{q_1r_1}
\delta_{p_1j_1}
\delta_{\alpha_2\beta_2}
\delta_{\gamma_2q_2}
\delta_{\beta_1\gamma_1}
\delta_{\alpha_1q_1}
+\delta_{p_2q_2}
\delta_{r_2j_2}
\delta_{q_1r_1}
\delta_{p_1j_1}
\delta_{\gamma_2\beta_2}
\delta_{\alpha_2q_2}
\delta_{\beta_1\alpha_1}
\delta_{\gamma_1q_1}
\nonumber
\\
&\quad
+
\delta_{r_2q_2}
\delta_{p_2j_2}
\delta_{q_1p_1}
\delta_{r_1j_1}
\delta_{\alpha_2\beta_2}
\delta_{\gamma_2q_2}
\delta_{\beta_1\gamma_1}
\delta_{\alpha_1q_1}
+
\delta_{r_2q_2}
\delta_{p_2j_2}
\delta_{q_1p_1}
\delta_{r_1j_1}
\delta_{\gamma_2\beta_2}
\delta_{\alpha_2q_2}
\delta_{\beta_1\alpha_1}
\delta_{\gamma_1q_1}.
\nonumber
\end{align}
Let $X=\p_xQ$ and $Y=Z=W=Q$. 
It follows that 
\begin{align}
S_1&=\sum_{p,q,\alpha,\beta,\gamma=1}^n 
     S_{p,q,r}^{j}\overline{S_{\alpha,\beta,\gamma}^{q}}
     \overline{X_{\alpha}}Y_{\beta}\overline{Z_{\gamma}}
     W_p
=:L_3+L_4+L_5+L_6, 
\label{eq:long1}
\\
S_2
&=\sum_{p,q,\alpha,\beta,\gamma=1}^n 
     S_{p,q,r}^{j}S_{\alpha,\beta,\gamma}^{p}
     X_{\alpha}\overline{Y_{\beta}}Z_{\gamma}\overline{W_q}
=:L_7+L_8+L_9+L_{10},
\label{eq:long2} 
\end{align}
where  
\begin{align}
L_3
&=
\sum_{\sharp_1}^{k_0}
\sum_{\sharp_2}^{m_0}
\delta_{p_2q_2}
\delta_{r_2j_2}
\delta_{q_1r_1}
\delta_{p_1j_1}
\delta_{\alpha_2\beta_2}
\delta_{\gamma_2q_2}
\delta_{\beta_1\gamma_1}
\delta_{\alpha_1q_1}
\overline{X_{\alpha}}Y_{\beta}\overline{Z_{\gamma}}W_p,
\nonumber
\\
L_4&=
\sum_{\sharp_1}^{k_0}
\sum_{\sharp_2}^{m_0}
\delta_{p_2q_2}
\delta_{r_2j_2}
\delta_{q_1r_1}
\delta_{p_1j_1}
\delta_{\gamma_2\beta_2}
\delta_{\alpha_2q_2}
\delta_{\beta_1\alpha_1}
\delta_{\gamma_1q_1}
\overline{X_{\alpha}}Y_{\beta}\overline{Z_{\gamma}}W_p,
\nonumber
\\
L_5&=
\sum_{\sharp_1}^{k_0}
\sum_{\sharp_2}^{m_0}
\delta_{r_2q_2}
\delta_{p_2j_2}
\delta_{q_1p_1}
\delta_{r_1j_1}
\delta_{\alpha_2\beta_2}
\delta_{\gamma_2q_2}
\delta_{\beta_1\gamma_1}
\delta_{\alpha_1q_1}
\overline{X_{\alpha}}Y_{\beta}\overline{Z_{\gamma}}W_p,
\nonumber
\\
L_6&=
\sum_{\sharp_1}^{k_0}
\sum_{\sharp_2}^{m_0}
\delta_{r_2q_2}
\delta_{p_2j_2}
\delta_{q_1p_1}
\delta_{r_1j_1}
\delta_{\gamma_2\beta_2}
\delta_{\alpha_2q_2}
\delta_{\beta_1\alpha_1}
\delta_{\gamma_1q_1}
\overline{X_{\alpha}}Y_{\beta}\overline{Z_{\gamma}}W_p, 
\nonumber
\end{align}
and 
\begin{align}
L_7
&=
\sum_{\sharp_1}^{k_0}
\sum_{\sharp_2}^{m_0}
\delta_{p_2q_2}
\delta_{r_2j_2}
\delta_{q_1r_1}
\delta_{p_1j_1}
\delta_{\alpha_2\beta_2}
\delta_{\gamma_2p_2}
\delta_{\beta_1\gamma_1}
\delta_{\alpha_1p_1}
X_{\alpha}\overline{Y_{\beta}}Z_{\gamma}\overline{W_q},
\nonumber
\\
L_8&=
\sum_{\sharp_1}^{k_0}
\sum_{\sharp_2}^{m_0}
\delta_{p_2q_2}
\delta_{r_2j_2}
\delta_{q_1r_1}
\delta_{p_1j_1}
\delta_{\gamma_2\beta_2}
\delta_{\alpha_2p_2}
\delta_{\beta_1\alpha_1}
\delta_{\gamma_1p_1}
X_{\alpha}\overline{Y_{\beta}}Z_{\gamma}\overline{W_q},
\nonumber
\\
L_9&=
\sum_{\sharp_1}^{k_0}
\sum_{\sharp_2}^{m_0}
\delta_{r_2q_2}
\delta_{p_2j_2}
\delta_{q_1p_1}
\delta_{r_1j_1}
\delta_{\alpha_2\beta_2}
\delta_{\gamma_2p_2}
\delta_{\beta_1\gamma_1}
\delta_{\alpha_1p_1}
X_{\alpha}\overline{Y_{\beta}}Z_{\gamma}\overline{W_q},
\nonumber
\\
L_{10}&=
\sum_{\sharp_1}^{k_0}
\sum_{\sharp_2}^{m_0}
\delta_{r_2q_2}
\delta_{p_2j_2}
\delta_{q_1p_1}
\delta_{r_1j_1}
\delta_{\gamma_2\beta_2}
\delta_{\alpha_2p_2}
\delta_{\beta_1\alpha_1}
\delta_{\gamma_1p_1}
X_{\alpha}\overline{Y_{\beta}}Z_{\gamma}\overline{W_q}, 
\nonumber
\end{align}
and we use the notation  $\displaystyle\sum_{\sharp_1}^{k_0}:=\sum_{p_1,q_1,\alpha_1,\beta_1,\gamma_1=1}^{k_0}$
and $\displaystyle\sum_{\sharp_2}^{m_0}:=\sum_{p_2,q_2,\alpha_2,\beta_2,\gamma_2=1}^{m_0}$.
Note that $r=(r_1,r_2)$ and $j=(j_1,j_2)$ are fixed here. 
By a simple but a bit careful computation, we deduce 
\begin{align}
L_3
&=
\delta_{r_2j_2}
\sum_{q_1,\alpha_1,\beta_1,\gamma_1=1}^{k_0}
\sum_{p_2,q_2,\alpha_2,\beta_2,\gamma_2=1}^{m_0}
\nonumber
\\
&\qquad \times 
\delta_{p_2q_2}
\delta_{q_1r_1}
\delta_{\alpha_2\beta_2}
\delta_{\gamma_2q_2}
\delta_{\beta_1\gamma_1}
\delta_{\alpha_1q_1}
\overline{X_{(\alpha_1,\alpha_2)}}Y_{(\beta_1,\beta_2)}\overline{Z_{(\gamma_1,\gamma_2)}}
W_{(j_1,p_2)}
\nonumber
\\
&=
\delta_{r_2j_2}
\sum_{q_1,\beta_1,\gamma_1=1}^{k_0}
\sum_{p_2,q_2,\alpha_2,\beta_2=1}^{m_0}
\delta_{p_2q_2}
\delta_{q_1r_1}
\delta_{\alpha_2\beta_2}
\delta_{\beta_1\gamma_1}
\overline{X_{(q_1,\alpha_2)}}Y_{(\beta_1,\beta_2)}\overline{Z_{(\gamma_1,q_2)}}
W_{(j_1,p_2)}
\nonumber
\\
&=
\delta_{r_2j_2}
\sum_{\beta_1=1}^{k_0}
\sum_{q_2,\beta_2=1}^{m_0}
\overline{X_{(r_1,\beta_2)}}Y_{(\beta_1,\beta_2)}\overline{Z_{(\beta_1,q_2)}}
W_{(j_1,q_2)}, 
\label{eq:37077}
\end{align}
\begin{align}
L_4&=
\delta_{r_2j_2}
\sum_{q_1,\alpha_1,\beta_1,\gamma_1=1}^{k_0}
\sum_{p_2,q_2,\alpha_2,\beta_2,\gamma_2=1}^{m_0}
\nonumber
\\
&\qquad
\times
\delta_{p_2q_2}
\delta_{q_1r_1}
\delta_{\gamma_2\beta_2}
\delta_{\alpha_2q_2}
\delta_{\beta_1\alpha_1}
\delta_{\gamma_1q_1}
\overline{X_{(\alpha_1,\alpha_2)}}Y_{(\beta_1,\beta_2)}\overline{Z_{(\gamma_1,\gamma_2)}}
W_{(j_1,p_2)}
\nonumber
\\
&=
\delta_{r_2j_2}
\sum_{q_1,\alpha_1,\beta_1=1}^{k_0}
\sum_{p_2,q_2,\beta_2,\gamma_2=1}^{m_0}
\delta_{p_2q_2}
\delta_{q_1r_1}
\delta_{\gamma_2\beta_2}
\delta_{\beta_1\alpha_1}
\overline{X_{(\alpha_1,q_2)}}Y_{(\beta_1,\beta_2)}\overline{Z_{(q_1,\gamma_2)}}
W_{(j_1,p_2)}
\nonumber
\\
&=
\delta_{r_2j_2}
\sum_{\beta_1=1}^{k_0}
\sum_{q_2,\beta_2=1}^{m_0}
\overline{X_{(\beta_1,q_2)}}Y_{(\beta_1,\beta_2)}\overline{Z_{(r_1,\beta_2)}}
W_{(j_1,q_2)}, 
\label{eq:37078}
\end{align}
\begin{align}
L_5&=
\delta_{r_1j_1}
\sum_{p_1,q_1,\alpha_1,\beta_1,\gamma_1=1}^{k_0}
\sum_{q_2,\alpha_2,\beta_2,\gamma_2=1}^{m_0}
\nonumber
\\
&\qquad 
\times
\delta_{r_2q_2}
\delta_{q_1p_1}
\delta_{\alpha_2\beta_2}
\delta_{\gamma_2q_2}
\delta_{\beta_1\gamma_1}
\delta_{\alpha_1q_1}
\overline{X_{(\alpha_1,\alpha_2)}}Y_{(\beta_1,\beta_2)}\overline{Z_{(\gamma_1,\gamma_2)}}
W_{(p_1,j_2)}
\nonumber
\\
&=
\delta_{r_1j_1}
\sum_{p_1,q_1,\beta_1,\gamma_1=1}^{k_0}
\sum_{q_2,\alpha_2,\beta_2=1}^{m_0}
\delta_{r_2q_2}
\delta_{q_1p_1}
\delta_{\alpha_2\beta_2}
\delta_{\beta_1\gamma_1}
\overline{X_{(q_1,\alpha_2)}}Y_{(\beta_1,\beta_2)}
\overline{Z_{(\gamma_1,q_2)}}W_{(p_1,j_2)}
\nonumber
\\
&=
\delta_{r_1j_1}
\sum_{q_1,\beta_1=1}^{k_0}
\sum_{\beta_2=1}^{m_0}
\overline{X_{(q_1,\beta_2)}}Y_{(\beta_1,\beta_2)}\overline{Z_{(\beta_1,r_2)}}
W_{(q_1,j_2)},
\label{eq:37079}
\end{align}
\begin{align}
L_6
&=
\delta_{r_1j_1}
\sum_{p_1,q_1,\alpha_1,\beta_1,\gamma_1=1}^{k_0}
\sum_{q_2,\alpha_2,\beta_2,\gamma_2=1}^{m_0}
\nonumber
\\&\qquad \times 
\delta_{r_2q_2}
\delta_{q_1p_1}
\delta_{\gamma_2\beta_2}
\delta_{\alpha_2q_2}
\delta_{\beta_1\alpha_1}
\delta_{\gamma_1q_1}
\overline{X_{(\alpha_1,\alpha_2)}}Y_{(\beta_1,\beta_2)}\overline{Z_{(\gamma_1,\gamma_2)}}
W_{(p_1,j_2)}
\nonumber
\\
&=
\delta_{r_1j_1}
\sum_{p_1,q_1,\alpha_1,\beta_1=1}^{k_0}
\sum_{q_2,\beta_2,\gamma_2=1}^{m_0} 
\delta_{r_2q_2}
\delta_{q_1p_1}
\delta_{\gamma_2\beta_2}
\delta_{\beta_1\alpha_1}
\overline{X_{(\alpha_1,q_2)}}Y_{(\beta_1,\beta_2)}\overline{Z_{(q_1,\gamma_2)}}
W_{(p_1,j_2)}
\nonumber
\\
&=
\delta_{r_1j_1}
\sum_{q_1,\beta_1=1}^{k_0}
\sum_{\beta_2=1}^{m_0}
\overline{X_{(\beta_1,r_2)}}Y_{(\beta_1,\beta_2)}\overline{Z_{(q_1,\beta_2)}}
W_{(q_1,j_2)}. 
\label{eq:370710}
\end{align}
\begin{align}
L_7
&=
\delta_{r_2j_2}
\sum_{p_1,\alpha_1,\beta_1,\gamma_1=1}^{k_0}
\sum_{p_2,q_2,\alpha_2,\beta_2,\gamma_2=1}^{m_0}
\nonumber
\\
&\qquad\times
\delta_{p_2q_2}
\delta_{p_1j_1}
\delta_{\alpha_2\beta_2}
\delta_{\gamma_2p_2}
\delta_{\beta_1\gamma_1}
\delta_{\alpha_1p_1}
X_{(\alpha_1,\alpha_2)}\overline{Y_{(\beta_1,\beta_2)}}
Z_{(\gamma_1,\gamma_2)}\overline{W_{(r_1,q_2)}}
\nonumber
\\
&=
\delta_{r_2j_2}
\sum_{p_1,\beta_1,\gamma_1=1}^{k_0}
\sum_{p_2,q_2,\alpha_2,\beta_2=1}^{m_0}
\delta_{p_2q_2}
\delta_{p_1j_1}
\delta_{\alpha_2\beta_2}
\delta_{\beta_1\gamma_1}
X_{(p_1,\alpha_2)}\overline{Y_{(\beta_1,\beta_2)}}
Z_{(\gamma_1,p_2)}\overline{W_{(r_1,q_2)}}
\nonumber
\\
&=
\delta_{r_2j_2}
\sum_{\beta_1=1}^{k_0}
\sum_{q_2,\beta_2=1}^{m_0}
X_{(j_1,\beta_2)}\overline{Y_{(\beta_1,\beta_2)}}
Z_{(\beta_1,q_2)}\overline{W_{(r_1,q_2)}}, 
\label{eq:37081}
\end{align}
\begin{align}
L_8&=
\delta_{r_2j_2}
\sum_{p_1,\alpha_1,\beta_1,\gamma_1=1}^{k_0}
\sum_{p_2,q_2,\alpha_2,\beta_2,\gamma_2=1}^{m_0}
\nonumber
\\
&\qquad\times
\delta_{p_2q_2}
\delta_{p_1j_1}
\delta_{\gamma_2\beta_2}
\delta_{\alpha_2p_2}
\delta_{\beta_1\alpha_1}
\delta_{\gamma_1p_1}
X_{(\alpha_1,\alpha_2)}\overline{Y_{(\beta_1,\beta_2)}}
Z_{(\gamma_1,\gamma_2)}\overline{W_{(r_1,q_2)}}
\nonumber
\\
&=
\delta_{r_2j_2}
\sum_{p_1,\alpha_1,\beta_1=1}^{k_0}
\sum_{p_2,q_2,\beta_2,\gamma_2=1}^{m_0}
\delta_{p_2q_2}
\delta_{p_1j_1}
\delta_{\gamma_2\beta_2}
\delta_{\beta_1\alpha_1}
X_{(\alpha_1,p_2)}\overline{Y_{(\beta_1,\beta_2)}}
Z_{(p_1,\gamma_2)}\overline{W_{(r_1,q_2)}}
\nonumber
\\
&=
\delta_{r_2j_2}
\sum_{\beta_1=1}^{k_0}
\sum_{q_2,\beta_2=1}^{m_0}
X_{(\beta_1,q_2)}\overline{Y_{(\beta_1,\beta_2)}}
Z_{(j_1,\beta_2)}\overline{W_{(r_1,q_2)}}, 
\label{eq:37082}
\end{align}
\begin{align}
L_9&=
\delta_{r_1j_1}
\sum_{p_1,q_1,\alpha_1,\beta_1,\gamma_1=1}^{k_0}
\sum_{p_2,\alpha_2,\beta_2,\gamma_2=1}^{m_0}
\nonumber
\\
&\qquad\times
\delta_{p_2j_2}
\delta_{q_1p_1}
\delta_{\alpha_2\beta_2}
\delta_{\gamma_2p_2}
\delta_{\beta_1\gamma_1}
\delta_{\alpha_1p_1}
X_{(\alpha_1,\alpha_2)}\overline{Y_{(\beta_1,\beta_2)}}
Z_{(\gamma_1,\gamma_2)}\overline{W_{(q_1,r_2)}}
\nonumber
\\
&=
\delta_{r_1j_1}
\sum_{p_1,q_1,\beta_1,\gamma_1=1}^{k_0}
\sum_{p_2,\alpha_2,\beta_2=1}^{m_0}
\delta_{p_2j_2}
\delta_{q_1p_1}
\delta_{\alpha_2\beta_2}
\delta_{\beta_1\gamma_1}
X_{(p_1,\alpha_2)}\overline{Y_{(\beta_1,\beta_2)}}
Z_{(\gamma_1,p_2)}\overline{W_{(q_1,r_2)}}
\nonumber
\\
&=
\delta_{r_1j_1}
\sum_{q_1,\beta_1=1}^{k_0}
\sum_{\beta_2=1}^{m_0}
X_{(q_1,\beta_2)}\overline{Y_{(\beta_1,\beta_2)}}
Z_{(\beta_1,j_2)}\overline{W_{(q_1,r_2)}}, 
\label{eq:37083}
\end{align}
\begin{align}
L_{10}&=
\delta_{r_1j_1}
\sum_{p_1,q_1,\alpha_1,\beta_1,\gamma_1=1}^{k_0}
\sum_{p_2,\alpha_2,\beta_2,\gamma_2=1}^{m_0}
\nonumber
\\
&\qquad\times
\delta_{p_2j_2}
\delta_{q_1p_1}
\delta_{\gamma_2\beta_2}
\delta_{\alpha_2p_2}
\delta_{\beta_1\alpha_1}
\delta_{\gamma_1p_1}
X_{(\alpha_1,\alpha_2)}\overline{Y_{(\beta_1,\beta_2)}}
Z_{(\gamma_1,\gamma_2)}\overline{W_{(q_1,r_2)}}
\nonumber
\\
&=
\delta_{r_1j_1}
\sum_{p_1,q_1,\alpha_1,\beta_1=1}^{k_0}
\sum_{p_2,\beta_2,\gamma_2=1}^{m_0}
\delta_{p_2j_2}
\delta_{q_1p_1}
\delta_{\gamma_2\beta_2}
\delta_{\beta_1\alpha_1}
X_{(\alpha_1,p_2)}\overline{Y_{(\beta_1,\beta_2)}}
Z_{(p_1,\gamma_2)}\overline{W_{(q_1,r_2)}}
\nonumber
\\
&=
\delta_{r_1j_1}
\sum_{q_1,\beta_1=1}^{k_0}
\sum_{\beta_2=1}^{m_0}
X_{(\beta_1,j_2)}\overline{Y_{(\beta_1,\beta_2)}}
Z_{(q_1,\beta_2)}\overline{W_{(q_1,r_2)}}.
\label{eq:37084}
\end{align}
Summing \eqref{eq:37077}, \eqref{eq:37078}, \eqref{eq:37081} and \eqref{eq:37082}, 
and substituting $X=\p_xQ$, $Y=Z=W=Q$, 
we deduce
\begin{align}
L_3+L_4+L_7+L_8
&=
\delta_{r_2j_2}
\sum_{\beta_1=1}^{k_0}
\sum_{q_2,\beta_2=1}^{m_0}
\overline{\p_xQ_{(r_1,\beta_2)}}Q_{(\beta_1,\beta_2)}
\overline{Q_{(\beta_1,q_2)}}Q_{(j_1,q_2)}
\nonumber
\\
&\quad
+
\delta_{r_2j_2}
\sum_{\beta_1=1}^{k_0}
\sum_{q_2,\beta_2=1}^{m_0}
\overline{\p_xQ_{(\beta_1,q_2)}}Q_{(\beta_1,\beta_2)}
\overline{Q_{(r_1,\beta_2)}}Q_{(j_1,q_2)}
\nonumber
\\
&\quad
+
\delta_{r_2j_2}
\sum_{\beta_1=1}^{k_0}
\sum_{q_2,\beta_2=1}^{m_0}
\p_xQ_{(j_1,\beta_2)}\overline{Q_{(\beta_1,\beta_2)}}
Q_{(\beta_1,q_2)}\overline{Q_{(r_1,q_2)}}
\nonumber
\\&\quad
+
\delta_{r_2j_2}
\sum_{\beta_1=1}^{k_0}
\sum_{q_2,\beta_2=1}^{m_0}
\p_xQ_{(\beta_1,q_2)}\overline{Q_{(\beta_1,\beta_2)}}
Q_{(j_1,\beta_2)}\overline{Q_{(r_1,q_2)}}
\nonumber
\\
&=
\delta_{r_2j_2}
\sum_{\beta_1=1}^{k_0}
\sum_{q_2,\beta_2=1}^{m_0}
\overline{\p_xQ_{(r_1,\beta_2)}}Q_{(\beta_1,\beta_2)}
\overline{Q_{(\beta_1,q_2)}}Q_{(j_1,q_2)}
\nonumber
\\
&\quad
+
\delta_{r_2j_2}
\sum_{\beta_1=1}^{k_0}
\sum_{q_2,\beta_2=1}^{m_0}
\overline{\p_xQ_{(\beta_1,q_2)}}Q_{(\beta_1,\beta_2)}
\overline{Q_{(r_1,\beta_2)}}Q_{(j_1,q_2)}
\nonumber
\\
&\quad
+
\delta_{r_2j_2}
\sum_{\beta_1=1}^{k_0}
\sum_{q_2,\beta_2=1}^{m_0}
\p_xQ_{(j_1,q_2)}\overline{Q_{(\beta_1,q_2)}}
Q_{(\beta_1,\beta_2)}\overline{Q_{(r_1,\beta_2)}}
\nonumber
\\&\quad
+
\delta_{r_2j_2}
\sum_{\beta_1=1}^{k_0}
\sum_{q_2,\beta_2=1}^{m_0}
\p_xQ_{(\beta_1,\beta_2)}\overline{Q_{(\beta_1,q_2)}}
Q_{(j_1,q_2)}\overline{Q_{(r_1,\beta_2)}}
\nonumber
\\
&=
\p_x\left\{
\delta_{r_2j_2}
\sum_{\beta_1=1}^{k_0}
\sum_{q_2,\beta_2=1}^{m_0}
\overline{Q_{(r_1,\beta_2)}}Q_{(\beta_1,\beta_2)}
\overline{Q_{(\beta_1,q_2)}}Q_{(j_1,q_2)}
\right\}.
\label{eq:37085}
\end{align}
In the same way as above, 
summing \eqref{eq:37079}, \eqref{eq:370710}, \eqref{eq:37083}, and 
\eqref{eq:37084}, 
and substituting $X=\p_xQ$, $Y=Z=W=Q$, 
we deduce 
\begin{align}
L_5+L_6+L_9+L_{10}
&=
\delta_{r_1j_1}
\sum_{q_1,\beta_1=1}^{k_0}
\sum_{\beta_2=1}^{m_0}
\overline{\p_xQ_{(q_1,\beta_2)}}Q_{(\beta_1,\beta_2)}
\overline{Q_{(\beta_1,r_2)}}Q_{(q_1,j_2)}
\nonumber
\\
&\quad+
\delta_{r_1j_1}
\sum_{q_1,\beta_1=1}^{k_0}
\sum_{\beta_2=1}^{m_0}
\overline{\p_xQ_{(\beta_1,r_2)}}Q_{(\beta_1,\beta_2)}
\overline{Q_{(q_1,\beta_2)}}Q_{(q_1,j_2)}
\nonumber
\\
&\quad+
\delta_{r_1j_1}
\sum_{q_1,\beta_1=1}^{k_0}
\sum_{\beta_2=1}^{m_0}
\p_xQ_{(q_1,\beta_2)}\overline{Q_{(\beta_1,\beta_2)}}
Q_{(\beta_1,j_2)}\overline{Q_{(q_1,r_2)}}
\nonumber
\\
&\quad+
\delta_{r_1j_1}
\sum_{q_1,\beta_1=1}^{k_0}
\sum_{\beta_2=1}^{m_0}
\p_xQ_{(\beta_1,j_2)}\overline{Q_{(\beta_1,\beta_2)}}
Q_{(q_1,\beta_2)}\overline{Q_{(q_1,r_2)}}
\nonumber
\\
&=
\p_x\left\{
\delta_{r_1j_1}
\sum_{q_1,\beta_1=1}^{k_0}
\sum_{\beta_2=1}^{m_0}
\overline{Q_{(q_1,\beta_2)}}Q_{(\beta_1,\beta_2)}
\overline{Q_{(\beta_1,r_2)}}Q_{(q_1,j_2)}
\right\}. 
\label{eq:37086}
\end{align}
From \eqref{eq:37085} and \eqref{eq:37086}, we get 
\begin{align}
S_1+S_2
&=
\p_x\left\{
\delta_{r_1j_1}
\sum_{q_1,\beta_1=1}^{k_0}
\sum_{\beta_2=1}^{m_0}
\overline{Q_{(q_1,\beta_2)}}Q_{(\beta_1,\beta_2)}
\overline{Q_{(\beta_1,r_2)}}Q_{(q_1,j_2)}
\right\}
\nonumber
\\
&\quad +
\p_x\left\{
\delta_{r_2j_2}
\sum_{\beta_1=1}^{k_0}
\sum_{q_2,\beta_2=1}^{m_0}
\overline{Q_{(r_1,\beta_2)}}Q_{(\beta_1,\beta_2)}
\overline{Q_{(\beta_1,q_2)}}Q_{(j_1,q_2)}
\right\}. 
\label{eq:37087}
\end{align}
Using this and replacing indexes, we deduce  
\begin{align}
&\sum_{r=1}^n
\left(\int_{-\infty}^x
(S_1+S_2)(t,y)\,dy
\right)Q_r
\nonumber
\\
&=
\sum_{r_1=1}^{k_0}
\sum_{r_2=1}^{m_0}
\left(
\delta_{r_1j_1}
\sum_{q_1,\beta_1=1}^{k_0}
\sum_{\beta_2=1}^{m_0}
\overline{Q_{(q_1,\beta_2)}}Q_{(\beta_1,\beta_2)}
\overline{Q_{(\beta_1,r_2)}}Q_{(q_1,j_2)}
\right)
Q_{(r_1,r_2)}
\nonumber
\\
&\quad+
\sum_{r_1=1}^{k_0}
\sum_{r_2=1}^{m_0}
\left(
\delta_{r_2j_2}
\sum_{\beta_1=1}^{k_0}
\sum_{q_2,\beta_2=1}^{m_0}
\overline{Q_{(r_1,\beta_2)}}Q_{(\beta_1,\beta_2)}
\overline{Q_{(\beta_1,q_2)}}Q_{(j_1,q_2)}
\right)
Q_{(r_1,r_2)}
\nonumber
\\
&=
\sum_{q_1,\beta_1=1}^{k_0}
\sum_{r_2,\beta_2=1}^{m_0}
\overline{Q_{(q_1,\beta_2)}}Q_{(\beta_1,\beta_2)}
\overline{Q_{(\beta_1,r_2)}}Q_{(q_1,j_2)}
Q_{(j_1,r_2)}
\nonumber
\\
&\quad+
\sum_{r_1,\beta_1=1}^{k_0}
\sum_{q_2,\beta_2=1}^{m_0}
\overline{Q_{(r_1,\beta_2)}}Q_{(\beta_1,\beta_2)}
\overline{Q_{(\beta_1,q_2)}}Q_{(j_1,q_2)}
Q_{(r_1,j_2)}
\nonumber
\\
&=
\sum_{q_1,\beta_1=1}^{k_0}
\sum_{r_2,\beta_2=1}^{m_0}
Q_{(j_1,r_2)}
\overline{Q_{(\beta_1,r_2)}}
Q_{(\beta_1,\beta_2)}
\overline{Q_{(q_1,\beta_2)}}
Q_{(q_1,j_2)}
\nonumber
\\
&\quad +
\sum_{r_1,\beta_1=1}^{k_0}
\sum_{q_2,\beta_2=1}^{m_0}
Q_{(j_1,q_2)}
\overline{Q_{(\beta_1,q_2)}}
Q_{(\beta_1,\beta_2)}
\overline{Q_{(r_1,\beta_2)}}
Q_{(r_1,j_2)}
\nonumber
\\
&=
2
\sum_{s_2,s_4=1}^{k_0}
\sum_{s_1,s_3=1}^{m_0}
Q_{(j_1,s_1)}
\overline{Q_{(s_2,s_1)}}
Q_{(s_2, s_3)}
\overline{Q_{(s_4, s_3)}}
Q_{(s_4,j_2)}. 
\label{eq:S1S2}
\end{align}
Although the expression of $S_3+S_4$ can be obtained 
in the same way as above, 
$S_3+S_4$ is not expressed as an image of $\p_x$.  
Indeed, by applying \eqref{eq:37077}-\eqref{eq:37084} 
for $X=Z=W=Q$ and $Y=\p_xQ$, 
we see $S_3+S_4=(L_3+L_4+L_7+L_8)+(L_5+L_6+L_9+L_{10})$ where
\begin{align}
&L_3+L_4+L_7+L_8
\nonumber
\\
&=
\delta_{r_2j_2}
\sum_{\beta_1=1}^{k_0}
\sum_{q_2,\beta_2=1}^{m_0}
\overline{Q_{(r_1,\beta_2)}}\p_xQ_{(\beta_1,\beta_2)}
\overline{Q_{(\beta_1,q_2)}}Q_{(j_1,q_2)}
\nonumber
\\
&\quad
+
\delta_{r_2j_2}
\sum_{\beta_1=1}^{k_0}
\sum_{q_2,\beta_2=1}^{m_0}
\overline{Q_{(\beta_1,q_2)}}\p_xQ_{(\beta_1,\beta_2)}
\overline{Q_{(r_1,\beta_2)}}Q_{(j_1,q_2)}
\nonumber
\\
&\quad
+
\delta_{r_2j_2}
\sum_{\beta_1=1}^{k_0}
\sum_{q_2,\beta_2=1}^{m_0}
Q_{(j_1,q_2)}\overline{\p_xQ_{(\beta_1,q_2)}}
Q_{(\beta_1,\beta_2)}\overline{Q_{(r_1,\beta_2)}}
\nonumber
\\&\quad
+
\delta_{r_2j_2}
\sum_{\beta_1=1}^{k_0}
\sum_{q_2,\beta_2=1}^{m_0}
Q_{(\beta_1,\beta_2)}\overline{\p_xQ_{(\beta_1,q_2)}}
Q_{(j_1,q_2)}\overline{Q_{(r_1,\beta_2)}}
\nonumber
\\
&=
2
\delta_{r_2j_2}
\sum_{\beta_1=1}^{k_0}
\sum_{q_2,\beta_2=1}^{m_0}
Q_{(j_1,q_2)}
\p_x\left\{
\overline{Q_{(\beta_1,q_2)}}
Q_{(\beta_1,\beta_2)}
\right\}
\overline{Q_{(r_1,\beta_2)}},
\nonumber
\end{align}
\begin{align}
&L_5+L_6+L_9+L_{10}
\nonumber
\\
&=
\delta_{r_1j_1}
\sum_{q_1,\beta_1=1}^{k_0}
\sum_{\beta_2=1}^{m_0}
\overline{Q_{(q_1,\beta_2)}}\p_xQ_{(\beta_1,\beta_2)}
\overline{Q_{(\beta_1,r_2)}}Q_{(q_1,j_2)}
\nonumber
\\
&\quad+
\delta_{r_1j_1}
\sum_{q_1,\beta_1=1}^{k_0}
\sum_{\beta_2=1}^{m_0}
\overline{Q_{(\beta_1,r_2)}}\p_xQ_{(\beta_1,\beta_2)}
\overline{Q_{(q_1,\beta_2)}}Q_{(q_1,j_2)}
\nonumber
\\
&\quad+
\delta_{r_1j_1}
\sum_{q_1,\beta_1=1}^{k_0}
\sum_{\beta_2=1}^{m_0}
Q_{(q_1,\beta_2)}\overline{\p_xQ_{(\beta_1,\beta_2)}}
Q_{(\beta_1,j_2)}\overline{Q_{(q_1,r_2)}}
\nonumber
\\
&\quad+
\delta_{r_1j_1}
\sum_{q_1,\beta_1=1}^{k_0}
\sum_{\beta_2=1}^{m_0}
Q_{(\beta_1,j_2)}\overline{\p_xQ_{(\beta_1,\beta_2)}}
Q_{(q_1,\beta_2)}\overline{Q_{(q_1,r_2)}}
\nonumber
\\
&=
2\delta_{r_1j_1}
\sum_{q_1,\beta_1=1}^{k_0}
\sum_{\beta_2=1}^{m_0}
\p_x\left\{
\overline{Q_{(q_1,\beta_2)}}Q_{(\beta_1,\beta_2)}
\right\}
\overline{Q_{(\beta_1,r_2)}}Q_{(q_1,j_2)}. 
\nonumber
\end{align}
Hence we obtain 
\begin{align}
S_3+S_4
&=
2\delta_{r_1j_1}
\sum_{q_1,\beta_1=1}^{k_0}
\sum_{\beta_2=1}^{m_0}
\p_x\left\{
\overline{Q_{(q_1,\beta_2)}}Q_{(\beta_1,\beta_2)}
\right\}
\overline{Q_{(\beta_1,r_2)}}Q_{(q_1,j_2)}
\nonumber
\\
&\quad+
2
\delta_{r_2j_2}
\sum_{\beta_1=1}^{k_0}
\sum_{q_2,\beta_2=1}^{m_0}
Q_{(j_1,q_2)}
\p_x\left\{
\overline{Q_{(\beta_1,q_2)}}
Q_{(\beta_1,\beta_2)}
\right\}
\overline{Q_{(r_1,\beta_2)}}.
\label{eq:37088}
\end{align}
Using this and replacing indexes, 
we obtain 
\begin{align}
&\sum_{r=1}^n
\left(\int_{-\infty}^x
(S_3+S_4)(t,y)\,dy
\right)Q_r
\nonumber
\\
&=
2
\sum_{r_1, q_1,\beta_1=1}^{k_0}
\sum_{r_2,\beta_2=1}^{m_0}
\delta_{r_1j_1}
\left(
\int_{-\infty}^x
\p_x\left\{
\overline{Q_{(q_1,\beta_2)}}Q_{(\beta_1,\beta_2)}
\right\}
\overline{Q_{(\beta_1,r_2)}}Q_{(q_1,j_2)}\,dy
\right)
Q_{(r_1,r_2)}
\nonumber
\\
&\quad
+2\sum_{r_1,\beta_1=1}^{k_0}
\sum_{r_2, q_2,\beta_2=1}^{m_0}
\delta_{r_2j_2}
\left(
\int_{-\infty}^x
Q_{(j_1,q_2)}
\p_x\left\{
\overline{Q_{(\beta_1,q_2)}}
Q_{(\beta_1,\beta_2)}
\right\}
\overline{Q_{(r_1,\beta_2)}}
\,dy\right)
Q_{(r_1,r_2)}
\nonumber
\\
&=
2\sum_{q_1,\beta_1=1}^{k_0}
\sum_{r_2,\beta_2=1}^{m_0}
Q_{(j_1,r_2)}
\left(
\int_{-\infty}^x
\p_x\left\{
\overline{Q_{(q_1,\beta_2)}}Q_{(\beta_1,\beta_2)}
\right\}
\overline{Q_{(\beta_1,r_2)}}Q_{(q_1,j_2)}\,dy
\right)
\nonumber
\\
&\quad
+2\sum_{r_1,\beta_1=1}^{k_0}
\sum_{q_2,\beta_2=1}^{m_0}
\left(
\int_{-\infty}^x
Q_{(j_1,q_2)}
\p_x\left\{
\overline{Q_{(\beta_1,q_2)}}
Q_{(\beta_1,\beta_2)}
\right\}
\overline{Q_{(r_1,\beta_2)}}
\,dy\right)
Q_{(r_1,j_2)}
\nonumber
\\
&=
2\sum_{s_2,s_4=1}^{k_0}
\sum_{s_1, s_3=1}^{m_0}
Q_{(j_1,s_3)}
\left(
\int_{-\infty}^x
\overline{Q_{(s_4, s_3)}}
\p_x\left\{
Q_{(s_4,s_1)}
\overline{Q_{(s_2,s_1)}}
\right\}
Q_{(s_2,j_2)}\,dy
\right)
\nonumber
\\
&\quad
+2\sum_{s_2,s_4=1}^{k_0}
\sum_{s_1,s_3=1}^{m_0}
\left(
\int_{-\infty}^x
Q_{(j_1,s_1)}
\p_x\left\{
\overline{Q_{(s_2,s_1)}}
Q_{(s_2,s_3)}
\right\}
\overline{Q_{(s_4,s_3)}}
\,dy\right)
Q_{(s_4,j_2)}. 
\label{eq:S3S4}
\end{align}
Combining \eqref{eq:S1S2} and \eqref{eq:S3S4},   
we derive 
\begin{align}
&\sum_{r=1}^n
\left(
\int_{-\infty}^x
f_{j,r}^1(Q,\p_xQ)(t,y)\,dy
\right)Q_r
\nonumber
\\
&=
-2(b+2c)
\sum_{s_2,s_4=1}^{k_0}
\sum_{s_1,s_3=1}^{m_0}
Q_{(j_1,s_1)}
\overline{Q_{(s_2,s_1)}}
Q_{(s_2, s_3)}
\overline{Q_{(s_4, s_3)}}
Q_{(s_4,j_2)}
\nonumber
\\
&\quad
+2b
\sum_{s_2,s_4=1}^{k_0}
\sum_{s_1, s_3=1}^{m_0}
Q_{(j_1,s_3)}
\left(
\int_{-\infty}^x
\overline{Q_{(s_4, s_3)}}
\p_x\left\{
Q_{(s_4,s_1)}
\overline{Q_{(s_2,s_1)}}
\right\}
Q_{(s_2,j_2)}\,dy
\right)
\nonumber
\\
&\quad 
+2b
\sum_{s_2,s_4=1}^{k_0}
\sum_{s_1,s_3=1}^{m_0}
\left(
\int_{-\infty}^x
Q_{(j_1,s_1)}
\p_x\left\{
\overline{Q_{(s_2,s_1)}}
Q_{(s_2,s_3)}
\right\}
\overline{Q_{(s_4,s_3)}}
\,dy\right)
Q_{(s_4,j_2)}. 
\label{eq:37089}
\end{align}
On the other hand, since $G_{n_0,k_0}$ in this setting is 
Hermitian symmetric($\nabla R=0$),  we see 
$f_{j,r}^2(Q,\p_xQ)=0$ for any $j,r\in\left\{1,\ldots,n\right\}$ and hence 
\begin{align}
&\sum_{r=1}^n
\left(
\int_{-\infty}^x
f_{j,r}^2(Q,\p_xQ)(t,y)\,dy
\right)Q_r
=0. 
\nonumber
\end{align}
Finally, substituting 
\eqref{eq:37072}-\eqref{eq:37076}
and 
\eqref{eq:37089}
into \eqref{eq:r51251}, 
we obtain
\begin{align}
&\sqrt{-1}\p_tQ_j+(a\p_x^4+\lambda \p_x^2)Q_j
\nonumber
\\
&=
d_1
\sum_{s_2=1}^{k_0}
\sum_{s_1=1}^{m_0}
\p_x^2Q_{(j_1,s_1)}\overline{Q_{(s_2,s_1)}}Q_{(s_2,j_2)}
+
d_1
\sum_{s_2=1}^{k_0}
\sum_{s_1=1}^{m_0}
Q_{(j_1,s_1)}
\overline{Q_{(s_2,s_1)}}
\p_x^2Q_{(s_2,j_2)}
\nonumber
\\
&\quad
    +2d_2
\sum_{s_2=1}^{k_0}
\sum_{s_1=1}^{m_0}
Q_{(j_1,s_1)}\overline{\p_x^2Q_{(s_2,s_1)}}Q_{(s_2,j_2)}    
    \nonumber
    \\
    &\quad
    +d_3
\sum_{s_2=1}^{k_0}
\sum_{s_1=1}^{m_0}
\p_xQ_{(j_1,s_1)}\overline{\p_xQ_{(s_2,s_1)}}Q_{(s_2,j_2)}
+d_3
\sum_{s_2=1}^{k_0}
\sum_{s_1=1}^{m_0}
Q_{(j_1,s_1)}
\overline{\p_xQ_{(s_2,s_1)}}
\p_xQ_{(s_2,j_2)} 
        \nonumber
        \\
    &\quad
       +2d_4
 \sum_{s_2=1}^{k_0}
 \sum_{s_1=1}^{m_0}
 \p_xQ_{(j_1,s_1)}\overline{Q_{(s_2,s_1)}}\p_xQ_{(s_2,j_2)}      
    -2\lambda
    \sum_{s_2=1}^{k_0}
    \sum_{s_1=1}^{m_0}
    Q_{(j_1,s_1)}\overline{Q_{(s_2,s_1)}}Q_{(s_2,j_2)}
\nonumber
\\
&\quad
+(-2b-4c)
\sum_{s_2,s_4=1}^{k_0}
\sum_{s_1,s_3=1}^{m_0}
Q_{(j_1,s_1)}
\overline{Q_{(s_2,s_1)}}
Q_{(s_2, s_3)}
\overline{Q_{(s_4, s_3)}}
Q_{(s_4,j_2)}
\nonumber
\\
&\quad
+2b
\sum_{s_2,s_4=1}^{k_0}
\sum_{s_1, s_3=1}^{m_0}
Q_{(j_1,s_3)}
\left(
\int_{-\infty}^x
\overline{Q_{(s_4, s_3)}}
\p_x\left\{
Q_{(s_4,s_1)}
\overline{Q_{(s_2,s_1)}}
\right\}
Q_{(s_2,j_2)}\,dy
\right)
\nonumber
\\
&\quad 
+2b
\sum_{s_2,s_4=1}^{k_0}
\sum_{s_1,s_3=1}^{m_0}
\left(
\int_{-\infty}^x
Q_{(j_1,s_1)}
\p_x\left\{
\overline{Q_{(s_2,s_1)}}
Q_{(s_2,s_3)}
\right\}
\overline{Q_{(s_4,s_3)}}
\,dy\right)
Q_{(s_4,j_2)}. 
    \label{eq:Smomo3}
\end{align}
\begin{proof}[Proof of Corollary~\ref{cor:coro}] 
Under the setting \eqref{eq:11021}, it follows that 
$d_1=-2\beta+16\gamma$, 
$d_2=8\gamma$, 
$d_3=2\beta+32\gamma$, 
$d_4=-\beta+16\gamma$, 
$-2b-4c=-2\beta+32\gamma$, 
$2b=2\beta+16\gamma$, 
and thus the system of \eqref{eq:Smomo3} for 
$Q_1,\ldots,Q_n$   
is rewritten as 
\begin{align}
\sqrt{-1}q_t
&=
-\beta\,q_{xxxx}+\alpha\,q_{xx}
+(-2\beta+16\gamma)
(q_{xx}q^{*}q+qq^{*}q_{xx})
+16\gamma\,qq^{*}_{xx}q
\nonumber
\\
&\quad
+(2\beta+32\gamma)
(q_{x}q^{*}_xq+qq^{*}_xq_x)
+(-2\beta+32\gamma)
q_xq^{*}q_x
+2\alpha\,qq^{*}q
\nonumber
\\
&\quad
+(-2\beta+32\gamma)qq^{*}qq^{*}q
\nonumber
\\
&\quad
+(2\beta+16\gamma)
\left\{
q\left(\int_{-\infty}^x
q^{*}(qq^{*})_sq\,ds
\right)
+
\left(
\int_{-\infty}^x
q(q^{*}q)_sq^{*}\,ds
\right)q
\right\}
\label{eq:Ondq}
\end{align}
for $q=\left(Q_{(j_1,j_2)}\right)$ being an 
$\mathcal{M}_{k_0\times m_0}$-valued function 
whose $(j_1,j_2)$-component is $Q_{(j_1,j_2)}=Q_{(j_2-1)k_0+j_1}$. 
Furthermore, \eqref{eq:Ondq} is also formulated as follows:
\begin{align}
\sqrt{-1}q_t
&=
\alpha
\biggl\{
q_{xx}+2qq^{*}q\biggr\}
-\beta
\biggl\{
q_{xxxx}+4q_{xx}q^{*}q+2qq^{*}_{xx}q+4qq^{*}q_{xx}
\nonumber
\\
&\quad
+2q_xq^{*}_xq+6q_xq^{*}q_x+2qq^{*}_xq_x+6qq^{*}qq^{*}q
\biggr\}
\nonumber
\\
&\quad
+2(\beta+8\gamma)
\biggl\{
(qq^{*}q)_{xx}+2qq^{*}qq^{*}q
\nonumber
\\
&\qquad
+q\left(\int_{-\infty}^x
q^{*}(qq^{*})_sq\,ds
\right)
+
\left(
\int_{-\infty}^x
q(q^{*}q)_sq^{*}\,ds
\right)q
\biggr\}. 
\label{eq:Ondq2}
\end{align}
In addition, let $A(t):(-T,T)\to \mathfrak{u}(m_0)$ and 
$B(t):(-T,T)\to \mathfrak{u}(k_0)$ be defined by 
\begin{align}
A(t)&=2(\beta+8\gamma)\sqrt{-1}\left(\int_{-\infty}^0
q^{*}(qq^{*})_sq\,ds
\right), 
\nonumber
\\
B(t)&=2(\beta+8\gamma)\sqrt{-1}
\left(
\int_{-\infty}^0
q(q^{*}q)_sq^{*}\,ds
\right). 
\nonumber
\end{align}
Noting $(A(t))^{*}=-A(t)$ and $(B(t))^{*}=-B(t)$, 
we see there exist $y=y(t):(-T,T)\to U(m_0)$
and $z=z(t): (-T,T)\to U(k_0)$ 
such that 
$$
\frac{dy}{dt}=A(t)y, \quad y(0)=I_{m_0}, \quad 
\frac{dz}{dt}=zB(t), \quad z(0)=I_{k_0}. 
$$
It is easy to find that \eqref{eq:Ondq2} is transformed to \eqref{eq:mns}
by $q(t,x)\mapsto z(t)q(t,x)y(t)$. 
We omit the detail. 
\end{proof}
\begin{remark}
\label{remark:projectivecase} 
It is known that $G_{n+1,1}$ where $k_0=1$, $m_0=n_0-k_0=n$ is 
identified with the complex projective space $P_n(\mathbb{C})$ with the Fubini-Study metric, 
and is a complex K\"ahler manifold of complex dimension $n$ 
with constant holomorphic sectional curvature $K=4$ in our setting of $h$. 
Hence, \eqref{eq:Smomo2} and  \eqref{eq:Smomo3}   
should coincide with \eqref{eq:Spqrj2} and \eqref{eq:512132} with $K=4$ respectively. 
This may not be obvious immediately from the expressions of \eqref{eq:Smomo3} and \eqref{eq:512132}, but 
actually holds. See Appendix for the reason.
\end{remark}
\subsection{Relationship with the method in \cite{DW2018}}
\label{subsection:revisit}
Let $N=G_{n_0,k_0}$ be as above. 
Corollary~\ref{cor:coro} reveals that the system of 
\eqref{eq:r51251} for $Q_1, \ldots,Q_n$ with \eqref{eq:11021}
is essentially the same as \eqref{eq:mns} derived in \cite{DW2018}.  
However, one may want a more theoretical reason why they coincide with each other, 
since our method using the parallel orthonormal frame 
and that used to derive \eqref{eq:mns} in \cite{DW2018} are seemingly different. 
Hence, we here try to make a more convincing explanation of the reason 
by comparing the two methods. 
\par 
To begin with, we review the outline of the method in \cite{DW2018} briefly: 
The authors in  \cite{DW2018} started from identifying $G_{n_0,k_0}$ 
with 
$\left\{
E^{-1}\sigma_3E
\mid 
E\in U(n_0)
\right\}$ which 
is the adjoint orbit embedded 
in $\mathfrak{u}(n_0)$ at   
$
\sigma_3=
\dfrac{\sqrt{-1}}{2}
\begin{pmatrix}
I_{k_0} & 0 \\
0 & -I_{m_0} \\
\end{pmatrix}
\in \mathfrak{u}(n_0). 
$
We can see that the identification is verified by the one-to-one corresponding 
$
\Psi: 
G_{n_0,k_0}
\to 
\left\{
E^{-1}\sigma_3E
\mid 
E\in U(n_0)
\right\}
$
such that 
\begin{align}
\Psi(BA_0B^{*})
&=B\sigma_3B^{*}
\label{eq:iden}
\end{align}
for $A=BA_0B^{*}\in G_{n_0,k_0}$ with $B\in U(n_0)$. 
They next expressed the solution to the generalized bi-Schr\"odinger flow equation 
by $\varphi(t,x)=(E(t,x))^{-1}\sigma_3E(t,x)$. 
Here, $E=E(t,x):(-T,T)\times \RR\to U(n_0)$ and 
satisfies $E_x=PE$ for some 
$P=P(t,x):(-T,T)\times \RR\to \mathfrak{m}_{I}$, 
where $\mathfrak{m}_{I}$ is defined by \eqref{eq:236273}.
Based on the setting, they showed that 
the generalized bi-Schr\"odinger flow equation for $\varphi$ is equivalent to the 
fourth-order matrix-nonlinear Schr\"odinger-like equation for 
$P=E_xE^{-1}=E_xE^{*}$ up to a gauge transformation. 
Since $P$ takes values in $\mathfrak{m}_{I}$, 
\begin{equation}
P=
\begin{pmatrix}
0 & \mathbf{q} \\
-\mathbf{q}^{*} & 0 \\
\end{pmatrix}, 
\label{eq:P1018}
\end{equation} 
for some $\mathbf{q}=\mathbf{q}(t,x):(-T,T)\times \RR\to \mathcal{M}_{k_0\times m_0}$. 
(Equations satisfied by $P$ and $\mathbf{q}$ are respectively 
given by (42) and (62) in \cite{DW2018}, and (62) is just \eqref{eq:mns} for $q$ up to a gauge transformation.) 
Their proof is  based on the geometric concept of PDEs 
with given (non-zero) curvature representation. 
See \cite[Theorem~3]{DW2018} for more detail on their proof. 
\begin{remark}
\label{remark:TU} 
The embedding of $G_{n_0,k_0}$ in $\mathfrak{u}(n_0)$ was adopted also by 
Terng and Uhlenbeck \cite{TU} to show the equivalence of the Schr\"odinger flow 
equation for maps with values in $G_{n_0,k_0}$ and the matrix-nonlinear Schr\"odinger 
equation.  
In fact, from their results in \cite{TU}, 
it turns out that the above $E=E(t,x)$ exists uniquely under some assumptions on the 
map. For example, assume that $\varphi=\varphi(t,x):(-T,T)\times \RR\to 
\left\{
E^{-1}\sigma_3E
\mid 
E\in U(n_0)
\right\}$
is a smooth map such that 
$\displaystyle\lim_{x\to -\infty}\varphi(t,x)=\sigma_3$ and 
$\varphi_x(t,\cdot)$ 
is in the Schwartz class for any $t\in (-T,T)$. 
(The assumption on $\varphi$ 
is equivalently to $u\in C_{A_0}((-T,T)\times \RR; G_{n_0,k_0})$ 
for $u=\Psi^{-1}(\varphi):(-T,T)\times \RR\to G_{n_0,k_0}$.)
Then Corollary~3.3 in \cite{TU} shows that there exists a unique 
$E=E(t,x):(-T,T)\times \RR\to U(n_0)$ such that 
$\varphi(t,x)=(E(t,x))^{-1}\sigma_3E(t,x)$, 
$\displaystyle\lim_{x\to -\infty}E(t,x)=I$, and $EE^{*}_x$ takes values in $\mathfrak{m}_I$. 
\end{remark}
Next, we observe our method to derive \eqref{eq:Ondq}: 
Let  $u\in C_{u^{\infty}}((-T,T)\times \RR;G_{n_0,k_0})$ 
be a solution to \eqref{eq:pde} with \eqref{eq:11021} 
where $u^{\infty}=B^{\infty}A_0(B^{\infty})^{*}$ 
and $B^{\infty}\in U(n_0)$.
We can assume $B^{\infty}=I$ 
without loss of generality, by retaking a $G_{n_0,k_0}$-valued map $(B^{\infty})^{*}u(t,x)B^{\infty}$ 
as $u(t,x)$.  
Let $\left\{e_j,Je_j\right\}_{j=1}^n$ be the 
orthonormal frame for $u^{-1}TG_{n_0,k_0}$ 
constructed in Section~\ref{subsection:frame}, 
and let $Q_j:(-T,T)\times \RR\to \mathbb{C}$ for 
$j\in \left\{1,\ldots,n(=k_0m_0)\right\}$ 
be the functions defined by \eqref{eq:req10181} in Section~\ref{subsection:reduction}. 
We continue to denote $j=(j_1,j_2)$ for $j\in \left\{1,\ldots,n(=k_0m_0)\right\}$ 
if there exist $j_1\in \left\{1,\ldots,k_0\right\}$ and 
$j_2\in \left\{1,\ldots,m_0\right\}$ such that 
$j=(j_2-1)k_0+j_1$. 
As used in the proof of Corollary~\ref{cor:coro}, 
let $\left(Q_{(j_1,j_2)}\right)$ denote the $\mathcal{M}_{k_0\times m_0}$-valued 
function whose $(j_1,j_2)$-components are $Q_{(j_1,j_2)}$.  
\par 
The aim of this subsection is to verify the following:
\begin{proposition}
\label{proposition:revisit}
Under the assumption as above for $u\in C_{A_0}((-T,T)\times \RR;G_{n_0,k_0})$, 
the relation
\begin{equation}
P=
\begin{pmatrix}
0 & \left(Q_{(j_1,j_2)}\right) \\
-\left(Q_{(j_1,j_2)}\right) ^{*} & 0 \\
\end{pmatrix},
\label{eq:1019n} 
\end{equation}
holds, where $P$ is given by \eqref{eq:P1018}. 
\end{proposition}
This shows that the two $\mathcal{M}_{k_0\times m_0}$-valued functions 
$\mathbf{q}$ obtained in \cite{DW2018} and 
$\left(Q_{(j_1,j_2)}\right)$ constructed by our method 
(and thus the equations satisfied by them)  
essentially coincide with each other. 
Therefore, we can be convinced that Corollary~\ref{cor:coro} holds without 
doing the computation in Section~\ref{subsection:comGr}.
\par  
The key of the proof is that both  
$\left\{e_j,Je_j\right\}_{j=1}^n$ and 
$\left(Q_{(j_1,j_2)}\right)$
can be expressed explicitly with the aid of the co-diagonal lifting
(or the horizontal lifting) of $u(t,\cdot):\RR\to G_{n_0,k_0}$.   
\begin{proof}[Proof of Proposition~\ref{proposition:revisit}]
Recall that $\displaystyle\lim_{x\to -\infty}u(t,x)=u^{\infty}=A_0$ here. 
By following the argument in \cite{Andruchow}, 
a map $C=C(t,\cdot):\RR\to U(n_0)$ is called a co-diagonal lifting of $u=u(t,\cdot):\RR\to 
G_{n_0,k_0}$ for each fixed $t\in(-T,T)$, if 
\begin{equation}
C(t,x)A_0(C(t,x))^{*}=u(t,x) 
\ \text{and} \  
\sqrt{-1}(C(t,x))^{*}C_x(t,x)\in T_{u(t,x)}G_{n_0,k_0}
\label{eq:lifting}
\end{equation}
hold for any $x\in \RR$. 
We pick a co-diagonal lifting $C=C(t,x)$ of $u$ satisfying 
$C(t,-\infty):=\displaystyle\lim_{x\to -\infty}C(t,x)=I$ for any $t$. 
By the argument to show Lemma~2.6 in \cite{Andruchow}, 
such a co-diagonal lifting exists uniquely and  
is characterized as the unique solution to 
$C_x=[u_x,u]C$ satisfying $C(t,-\infty)=I$, 
where $[u_x,u]=u_xu-uu_x$. 
\par 
We investigate the expression of $C^{*}C_x$.  
It is immediate to see $C^{*}C=I$ and thus   
$C^{*}C_x+(C_x)^{*}C=O$ holds, 
since $C$ is $U(n_0)$-valued. 
This implies $C^{*}C_x=-(C_x)^{*}C=-(C^{*}C_x)^{*}$ 
and hence $C^{*}C_x$ is $\mathfrak{u}(n_0)$-valued. 
Therefore, we can write 
\begin{equation}
C^{*}C_x=
\begin{pmatrix}
C_{11} & C_{12} \\
-C_{12} ^{*} & C_{22} \\
\end{pmatrix},
\label{eq:10183}
\end{equation}
where $C_{11}=C_{11}(t,x):(-T,T)\times \RR\to \mathfrak{u}(k_0)$, 
$C_{22}=C_{22}(t,x):(-T,T)\times \RR\to \mathfrak{u}(m_0)$, and 
$C_{12}=C_{12}(t,x):(-T,T)\times \RR\to \mathcal{M}_{k_0\times m_0}$.  
On the other hand, $u=CA_0C^{*}$ in \eqref{eq:lifting} and 
$CC^{*}=I$ yields 
\begin{align}
u_x&=C_xA_0C^{*}+CA_0(C_x)^{*}
=
C
\left\{
C^{*}C_xA_0+A_0(C_x)^{*}C
\right\}
C^{*}. 
\label{eq:10181}
\end{align}
Using this, $C^{*}C=I$ and $A_0^2=A_0$, we deduce   
\begin{align}
[u_x,u]
&=
C
\left\{
C^{*}C_xA_0+A_0(C_x)^{*}C
\right\}
C^{*}
CA_0C^{*}
\nonumber
\\
&\quad
-CA_0C^{*}
C
\left\{
C^{*}C_xA_0+A_0(C_x)^{*}C
\right\}
C^{*}
\nonumber
\\
&=
C
\left\{
C^{*}C_xA_0
+
A_0(C_x)^{*}CA_0
-A_0C^{*}C_xA_0
-A_0(C_x)^{*}C
\right\}
C^{*}. 
\nonumber
\end{align}
Substituting this into 
$C^{*}C_x=C^{*}[u_x,u]C$ which follows from $C_x=[u_x,u]C$,  
and using $CC^{*}=C^{*}C=I$
and $(C_x)^{*}C=-C^{*}C_x$ which follows from $C^{*}C=I$, we deduce 
\begin{align}
C^{*}C_x&=
C^{*}C_xA_0
+
A_0(C_x)^{*}CA_0
-A_0C^{*}C_xA_0
-A_0(C_x)^{*}C
\nonumber
\\
&=
C^{*}C_xA_0+A_0C^{*}C_x
-2A_0C^{*}C_xA_0. 
\label{eq:10182}
\end{align}
Therefore, substituting \eqref{eq:10183} into the right hand side of \eqref{eq:10182}, 
we obtain
\begin{align}
C^{*}C_x
&=
\begin{pmatrix}
C_{11} & C_{12} \\
-C_{12} ^{*} & C_{22} \\
\end{pmatrix}
\begin{pmatrix}
I_{k_0} & 0 \\
0 & 0 \\
\end{pmatrix}
+
\begin{pmatrix}
I_{k_0} & 0 \\
0 & 0 \\
\end{pmatrix}
\begin{pmatrix}
C_{11} & C_{12} \\
-C_{12} ^{*} & C_{22} \\
\end{pmatrix}
\nonumber
\\
&\quad 
-2
\begin{pmatrix}
I_{k_0} & 0 \\
0 & 0 \\
\end{pmatrix}
\begin{pmatrix}
C_{11} & C_{12} \\
-C_{12} ^{*} & C_{22} \\
\end{pmatrix}
\begin{pmatrix}
I_{k_0} & 0 \\
0 & 0 \\
\end{pmatrix}
\nonumber
\\
&=
\begin{pmatrix}
0 & C_{12} \\
-C_{12} ^{*} & 0 \\
\end{pmatrix}, 
\label{eq:10184}
\end{align} 
which implies $C^{*}C_x$ takes values in $\mathfrak{m}_I$.
If we adopt the identification \eqref{eq:iden}, 
then what we obtained here is interpreted as 
$\Psi(u)(t,x)=C(t,x)\sigma_3(C(t,x))^{*}$, 
where $C=C(t,x):(-T,T)\times \RR\to U(n_0)$, $C^{*}(t,-\infty)=I$, and 
$C^{*}C_x$ takes values in $\mathfrak{m}_I$.
Hence, the uniqueness result stated in Remark~\ref{remark:TU} verifies 
$C=E^{-1}(=E^{*})$. From this, we see
\begin{align}
P&=E_xE^{*}=(C_x)^{*}C=-C^{*}C_x
=-\begin{pmatrix}
0 & C_{12} \\
-C_{12} ^{*} & 0 \\
\end{pmatrix}.
\label{eq:10185}
\end{align}
\par 
We next investigate the expression of $Q_j$, 
which by \eqref{eq:uxt}-\eqref{eq:req10181} is represented as  
$$
Q_j=h(u_x,e_j)+\sqrt{-1}h(u_x,J_ue_j).
$$
Set $j=(j_1,j_2)\in \left\{1,\ldots,n\right\}$. 
Recall that 
$e_j$ is constructed as the parallel transport of $e_j^{\infty}\in T_{u^{\infty}}G_{n_0,k_0}$ along 
$u(t,\cdot):\RR\to G_{n_0,k_0}$, 
and $e_j^{\infty}$ is defined by \eqref{eq:base_j}. 
In the present setting,  
$B^{\infty}=I$ and $u^{\infty}=A_0$, and hence we can write  
$$
e_{j}^{\infty}
=
\begin{pmatrix}
0 & E_{(j_1,j_2)} \\
(E_{(j_1,j_2)})^{*} & 0 \\
\end{pmatrix}. 
$$  
Note that the parallel transport of tangent vectors on $G_{n_0,k_0}$ along 
$u(t,\cdot)=C(t,\cdot)A_0(C(t,\cdot))^{*}$ 
can be represented by using the co-diagonal lifting $C(t,\cdot)$. 
Indeed, following the argument in \cite{Andruchow},  
we see  
$e_j(t,\cdot)=C(t,\cdot)e_j^{\infty}(C(t,\cdot))^{*}$, that is, 
\begin{align}
e_j(t,x)&=C(t,x)
\begin{pmatrix}
0 & E_{(j_1,j_2)} \\
(E_{(j_1,j_2)})^{*} & 0 \\
\end{pmatrix} 
(C(t,x))^{*}.
\label{eq:10191}
\end{align} 
From the expression, it follows that 
\begin{align}
J_ue_j(t,x)&=C(t,x)
\begin{pmatrix}
0 & \sqrt{-1}E_{(j_1,j_2)} \\
(\sqrt{-1}E_{(j_1,j_2)})^{*} & 0 \\
\end{pmatrix} 
(C(t,x))^{*}.
\label{eq:10192}
\end{align}
In addition, using $(C_x)^*C=-C^*C_x$ and substituting \eqref{eq:10184} into \eqref{eq:10181}, 
we deduce  
\begin{align}
u_x
&=
C
\left\{
C^{*}C_xA_0-A_0C^{*}C_x
\right\}
C^{*}
=
C
\begin{pmatrix}
0 & -C_{12} \\
(-C_{12}) ^{*} & 0 \\
\end{pmatrix}
C^*. 
\label{eq:10193}
\end{align}
Let us also recall \eqref{eq:metN} to see 
$
h(\Delta_1,\Delta_2)=
\operatorname{Re}\left[
\mathrm{tr}(\omega_1(\omega_2)^{*})
\right]
$
for  
$\Delta_k=
C
\begin{pmatrix}
0 & \omega_k \\
(\omega_k) ^{*} & 0 \\
\end{pmatrix}
C^*
\in T_{u}G_{n_0,k_0}
(k=1,2).
$
Then, using \eqref{eq:10191}-\eqref{eq:10193}, we have 
\begin{align}
h(u_x,e_j)
&=\operatorname{Re}\left[
\mathrm{tr}((-C_{12})(E_{(j_1,j_2)})^{*})
\right], 
\nonumber
\\
h(u_x,Je_j)
&=\operatorname{Re}\left[
\mathrm{tr}((-C_{12})(\sqrt{-1}E_{(j_1,j_2)})^{*})
\right]
=
\operatorname{Im}\left[
\mathrm{tr}((-C_{12})(E_{(j_1,j_2)})^{*})
\right]. 
\nonumber
\end{align}
This shows  
$Q_j(t,x)=\mathrm{tr}((-C_{12}(t,x))(E_{(j_1,j_2)})^{*})$. 
Since the right hand side equals to the $(j_1,j_2)$-component of 
$(-C_{12})(t,x)\in \mathcal{M}_{k_0\times m_0}$, we see 
$\left(
Q_{(j_1,j_2)}
\right)=-C_{12}$, that is,  
$$
\begin{pmatrix}
0 & \left(Q_{(j_1,j_2)}\right) \\
-\left(Q_{(j_1,j_2)}\right) ^{*} & 0 \\
\end{pmatrix}
=
-
\begin{pmatrix}
0 & C_{12} \\
-C_{12}^{*} & 0 \\
\end{pmatrix}. 
$$
Comparing this with \eqref{eq:10185}, 
we derive the desired
\eqref{eq:1019n}, which completes the proof.  
\end{proof}
%
%
\section*{Acknowledgments}
This work was supported by 
JSPS Grant-in-Aid for Scientific Research (C) 
Grant Numbers JP16K05235, JP20K03703, JP24K06813.
%
%
%
%
%
 %
 \section{Appendix}
 \label{section:Appendix}
Supplemental comments on Remark~\ref{remark:projectivecase} 
are presented below. 
\par 
Let $N=G_{n+1,1}$ where $k_0=1$, $m_0=n_0-k_0=n$. 
As commented in Remark~\ref{remark:projectivecase}, 
\eqref{eq:Smomo2} and  \eqref{eq:Smomo3} 
turn out to be the same as 
\eqref{eq:Spqrj2} and \eqref{eq:512132} for $K=4$ respectively. 
We here state the outline of how to check it. 
Since $p_1=q_1=r_1=j_1=1$ here,  
it is immediate to see \eqref{eq:Smomo2} becomes 
$$
S_{p,q,r}^j=
\delta_{p_2q_2}
\delta_{r_2j_2}
+
\delta_{r_2q_2}
\delta_{p_2j_2}
=\delta_{pq}
\delta_{rj}
+
\delta_{rq}
\delta_{pj}, 
$$ 
the right hand side of which actually coincides with \eqref{eq:Spqrj2} for $K=4$. 
Second, 
any index $j\in \left\{1,\ldots,n\right\}$ is expressed as $j=(1,j_2)$ by a unique $j_2\in\left\{1,\ldots,n\right\}$, and  
\eqref{eq:Smomo3} for $Q_{(1,j_2)}$ turns out to reduce to \eqref{eq:512132} 
for $Q_j=Q_{(1,j_2)}$.
We omit the detail, except to show that the sum of the final three terms of the right 
hand side of \eqref{eq:Smomo3} actually reduces to that of 
the final two terms of the right hand side of \eqref{eq:512132}, 
because the correspondence of the other terms is obvious. 
To show this, set 
\begin{align}
I_1&=
\sum_{s_2,s_4=1}^{k_0}
\sum_{s_1,s_3=1}^{m_0}
Q_{(j_1,s_1)}
\overline{Q_{(s_2,s_1)}}
Q_{(s_2, s_3)}
\overline{Q_{(s_4, s_3)}}
Q_{(s_4,j_2)}, 
\nonumber
\\
I_2&=
\sum_{s_2,s_4=1}^{k_0}
\sum_{s_1, s_3=1}^{m_0}
Q_{(j_1,s_3)}
\left(
\int_{-\infty}^x
\overline{Q_{(s_4, s_3)}}
\p_x\left\{
Q_{(s_4,s_1)}
\overline{Q_{(s_2,s_1)}}
\right\}
Q_{(s_2,j_2)}\,dy
\right)
\nonumber
\\
I_3&=
\sum_{s_2,s_4=1}^{k_0}
\sum_{s_1,s_3=1}^{m_0}
\left(
\int_{-\infty}^x
Q_{(j_1,s_1)}
\p_x\left\{
\overline{Q_{(s_2,s_1)}}
Q_{(s_2,s_3)}
\right\}
\overline{Q_{(s_4,s_3)}}
\,dy\right)
Q_{(s_4,j_2)}.
\nonumber
\end{align}
We compute $(-2b-4c)I_1+2bI_2+2bI_3$ where $k_0=1$, $m_0=n$ and 
$Q_j=Q_{(1,j_2)}$.
A simple computation shows 
\begin{align}
I_1&=
\sum_{s_1,s_3=1}^{n}
Q_{(1,s_1)}
\overline{Q_{(1,s_1)}}
Q_{(1, s_3)}
\overline{Q_{(1, s_3)}}
Q_{(1,j_2)}
=
\sum_{s_1=1}^{n}
|Q_{(1,s_1)}|^2
\sum_{s_3=1}^n
|Q_{(1, s_3)}|^2
Q_{(1,j_2)}
\nonumber
\\
&=
|Q|^4Q_j, 
\nonumber
\\
I_2
&=
\sum_{s_1, s_3=1}^{n}
Q_{(1,s_3)}
\left(
\int_{-\infty}^x
\overline{Q_{(1, s_3)}}
\p_x\left\{
Q_{(1,s_1)}
\overline{Q_{(1,s_1)}}
\right\}
Q_{(1,j_2)}\,dy
\right)
\nonumber
\\
&=
\sum_{s_3=1}^{n}
Q_{(1,s_3)}
\left(
\int_{-\infty}^x
\overline{Q_{(1, s_3)}}
\p_x\left(
|Q|^2
\right)
Q_{(1,j_2)}\,dy
\right)
\nonumber
\\
&=
\sum_{r=1}^{n}
\left(
\int_{-\infty}^x
\overline{Q_{r}}
Q_j
\p_x\left(
|Q|^2
\right)
\,dy
\right)Q_r.
\nonumber
\end{align}
The fundamental theorem of calculus shows  
\begin{align}
I_3&=
\sum_{s_1,s_3=1}^{n}
\left(
\int_{-\infty}^x
Q_{(1,s_1)}
\p_x\left\{
\overline{Q_{(1,s_1)}}
Q_{(1,s_3)}
\right\}
\overline{Q_{(1,s_3)}}
\,dy\right)
Q_{(1,j_2)}
\nonumber
\\
&=
\sum_{s_1,s_3=1}^{n}
Q_{(1,s_1)}
\overline{Q_{(1,s_1)}}
Q_{(1,s_3)}
\overline{Q_{(1,s_3)}}
Q_{(1,j_2)}
\nonumber
\\
&\quad
-\sum_{s_1,s_3=1}^{n}
\left(
\int_{-\infty}^x
Q_{(1,s_3)}
\p_x\left\{
\overline{Q_{(1,s_3)}}
Q_{(1,s_1)}
\right\}
\overline{Q_{(1,s_1)}}
\,dy\right)
Q_{(1,j_2)}
\nonumber
\\
&=
|Q|^4Q_j-I_3, 
\nonumber
\end{align}
which implies 
$I_3=\dfrac{1}{2}|Q|^4Q_j$. Combining them, we see 
\begin{align}
&(-2b-4c)I_1+2bI_2+2bI_3
\nonumber
\\
&=
(-b-4c)|Q|^4Q_j
+2b
\sum_{r=1}^{n}
\left(
\int_{-\infty}^x
\overline{Q_{r}}
Q_j
\p_x\left(
|Q|^2
\right)
\,dy
\right)Q_r, 
\nonumber
\end{align}
which 
actually 
equals to the sum of the final two terms of the 
right hand side of \eqref{eq:512132}  for $K=4$.


\begin{thebibliography}{0}
%
%
\bibitem{AA}
Anco,~S.~C., Asadi,~E.: 
Hasimoto variables, generalized vortex filament equations, Heisenberg models 
and Schr\"odinger maps arising from group-invariant NLS systems. 
J. Geom. Phys. \textbf{144}, 324--357 (2019) 
%
\bibitem{Andruchow}
Andruchow,~E.: 
The Grassmann manifold of a Hilbert space. 
Proceedings of the XIIth 
``Dr. Antonio A. R. Monteiro'' Congress, 41--55,
Actas Congr. 
``Dr. Antonio A. R. Monteiro'', 
Univ. Nac. del Sur, Bah\'ia Blanca, 2014.
%
\bibitem{Arvanitoyeorgos}
Arvanitoyeorgos,~A.: 
An introduction to Lie groups and the geometry of homogeneous spaces.
Translated from the 1999 Greek original and revised by the author. Student Mathematical Library, 22. American Mathematical Society, Providence, RI, 2003. xvi+141 pp
%
\bibitem{BZA}
Bendokat,~T, Zimmermann,~R, Absil,~P.A.: 
A Grassmann manifold handbook: Basic geometry and computational aspects. 
Adv.~Comput.~Math. \textbf{50}, 6  (2024)
%
%
\bibitem{CSU}
Chang,~N.H., Shatah,~J., Uhlenbeck,~K.:
Schr\"odinger maps.
Comm.\ Pure Appl.\ Math.
\textbf{53}, 590--602  (2000)

\bibitem{chihara2} 
Chihara,~H.: 
Fourth-order dispersive systems on the one-dimensional torus.
J.\ Pseudo-Differ.\ Oper.\ Appl. \textbf{6}, 237--263 (2015)

\bibitem{CO2} 
Chihara,~H, Onodera,~E.:  
A fourth-order dispersive flow into K\"ahler manifolds.
Z.\ Anal.\ Anwend. \textbf{34}, 221--249 (2015) 

\bibitem{DKA}
Daniel,~M, Kavitha,~L., Amuda,~R.: 
Soliton spin excitations in an anisotropic Heisenberg ferromagnet 
with octupole-dipole interaction.
Phys.\ Rev.\ B \textbf{59}, 13774 (1999)

\bibitem{DL}
Daniel,~M., Latha,~M.M.: 
Soliton in discrete and continuum alpha helical proteins with higher-order excitations. 
Physica A: Statistical Mechanics and its Applications 
\textbf{240}, 526--546 (1997)

\bibitem{Dimitric}
Dimitri\'{c},~I.: 
A note on equivariant embeddings of Grassmannians. 
Publ. Inst. Math. \textbf{59}, 131--137 (1996)

\bibitem{DW2018}
Ding,~Q., Wang,~Y.D.: 
Vortex filament on symmetric Lie algebras and 
generalized bi-Schr\"odinger flows.
Math.\ Z. 
\textbf{290}, 167--193 (2018)

\bibitem{DZ2021}
Ding,~Q., Zhong.~S.:  
On the vortex filament in 3-spaces and its generalizations. 
Sci. China Math.
\textbf{64}, 1331--1348 (2021)

\bibitem{DWW2003}
Ding,~W.Y., Wang,~H.Y., Wang,~Y.D.: 
Schr\"odinger flows on compact Hermitian symmetric spaces and related problems. 
Acta Math. Sin. (Engl. Ser.) \textbf{19} 303--312. (2003) 

\bibitem{FK}
Fordy,~A.P., Kulish,~P.P.:  
Nonlinear Schr\"odinger equations and simple Lie algebras.
Comm. Math. Phys.
\textbf{89}, 427--443 (1983)

\bibitem{fukumoto}
Fukumoto,~Y.:
Three-dimensional motion of a vortex filament and its relation 
to the localized induction hierarchy.
Eur.\ Phys.\ J. B {\bf 29}, 167--171 (2002)

\bibitem{FM}
Fukumoto,~Y., Moffatt,~T.K.:  
Motion and expansion of a viscous vortex ring.
Part 1. A higher-order asymptotic formula for the velocity.
J.\ Fluid.\ Mech. {\bf 417}, 1--45 (2000)

\bibitem{GZS} 
Guo,~B., Zeng,~M., Su,~F.:  
Periodic weak solutions for a classical 
one-dimensional isotropic biquadratic Heisenberg spin chain.
J. Math.\ Anal.\ Appl.  \textbf{330}, 729--739 (2007)

\bibitem{Hasimoto}
Hasimoto,~H.:
A soliton on a vortex filament.
J.\ Fluid.\ Mech. {\bf 51}, 477--485 (1972)

\bibitem{helgason}
Helgason,~S.: 
Differential Geometry, Lie Groups, and Symmetric Spaces. 
Corrected reprint of the 1978 original, 
Grad. Stud. Math., 34, 
American Mathematical Society, Providence, RI (2001)

\bibitem{jost}
Jost,~J.: 
Riemannian Geometry and Geometric Analysis.
Seventh edition, Universitext, 
Springer, Cham (2017)

\bibitem{KN}
Kobayashi,~S., Nomizu,~K.: 
Foundations of differential geometry, II. 
A Wiley-Inter science publication (1996)

\bibitem{Koiso1995}
Koiso,~N.:
Vortex filament equation and semilinear Schr\"odinger equation. 
Nonlinear waves (Sapporo, 1995), 231--236,
GAKUTO Internat. Ser. Math. Sci. Appl., 10, Gakk\=otosho, Tokyo, 1997

\bibitem{Koiso1997}
Koiso,~N.: 
The vortex filament equation and a semilinear Schr\"odinger equation 
in a Hermitian symmetric space.
Osaka J. Math. {\bf 34} 199--214 (1997)

\bibitem{LPD}
Lakshmanan,~M., Porsezian,~K., Daniel,~M.:  
Effect of discreteness on the continuum limit of the Heisenberg
spin chain.
Phys.\ Lett.\ A \textbf{133}, 483--488 (1988)

\bibitem{liu}
Liu,~H.F.: 
Periodic Cauchy problem of Heisenberg ferromagnet and its geometric framework.
J. Fixed Point Theory Appl. \textbf{24}, Paper No. 16, 19 pp. (2022)

\bibitem{NSVZ}
Nahmod,~A., Shatah,~J.,  Vega,~L., Zeng,~C.: 
Schr\"odinger maps and their associated frame systems.
Int.\ Math.\ Res.\ Not.\ IMRN. no. 21, Art. ID rnm088, 29 pp. (2007) 

\bibitem{oneill}
O'Neill,~B.:
Semi-Riemannian geometry.  
Pure Appl. Math., 103, Academic Press, New York (1983)

\bibitem{onodera0} 
Onodera,~E.: 
Generalized Hasimoto transform of one-dimensional 
dispersive flows into compact Riemann surfaces.
SIGMA Symmetry Integrability Geom.\ Methods Appl.
\textbf{4}, article No. 044, 10 pages (2008) 

\bibitem{onodera2}
Onodera,~E.:   
The initial value problem for a fourth-order 
dispersive closed curve flow on the two-sphere.
Proc.\ Roy.\ Soc.\ Edinburgh Sect.\ A
\textbf{147}, 1243--1277 (2017) 

\bibitem{onodera3}
Onodera,~E.:   
A fourth-order dispersive flow equation for 
closed curves on compact Riemann surfaces.
J.\ Geom.\ Anal. 
\textbf{27}, 3339--3403 (2017) 

\bibitem{onodera4}
Onodera,~E.: 
Local existence of a fourth-order dispersive 
curve flow on locally Hermitian symmetric spaces and its application.
Differential\ Geom.\ Appl. 
\textbf{67}, 101560, 26pp (2019)

\bibitem{onodera5}
Onodera,~E.:
Uniqueness of 1D Generalized Bi-Schr\"odinger Flow.
J.~Geom.~Anal.
\textbf{32}, Article number: 47, 41pp. (2022)

\bibitem{PDL} 
Porsezian,~K., Daniel,~M., Lakshmanan,~M.:  
On the integrability aspects of the one-dimensional classical 
continuum isotropic biquadratic Heisenberg spin chain.
J. Math.\ Phys. \textbf{33}, 1--10 (1992)

\bibitem{petersen}
Petersen,~P.:
Riemannian geometry.
Third edition
Grad. Texts in Math., 171
Springer, Cham (2016)

\bibitem{RRS}
Rodnianski,~I., Rubinstein,~Y.A., Staffilani,~G.:   
On the global well-posedness of the one-dimensional 
Schr\"odinger map flow.
Analysis and PDE \textbf{2}, 187--209 (2009)

\bibitem{SW2011}
Sun,~X.W., Wang,~Y.D.:  
KdV geometric flows on K\"ahler manifolds.
Internat.\ J.\ Math.  \textbf{22}, 1439--1500 (2011)

\bibitem{SW2013}
Sun,~X.W., Wang,~Y.D.: 
Geometric Schr\"odinger-Airy Flows on K\"ahler Manifolds.
Acta Math.\ Sin. (Engl. Ser.) \textbf{29}, 209--240 (2013)

\bibitem{TT}
Terng,~C.L., Thorbergsson,~G.: 
Completely integrable curve flows on adjoint orbits. 
Dedicated to Shiing-Shen Chern on his 90th birthday.
Results Math. \textbf{40}, 286--309 (2001) 

\bibitem{TU}
Terng,~C.L., Uhlenbeck,~K.: 
Schr\"odinger flows on Grassmannians. 
Integrable systems, geometry, and topology 
235--256,
AMS/IP Stud. Adv. Math., 36, Amer. Math. Soc., Providence, RI (2006)

\bibitem{ZT}
Zakharov,~V.E., Takhtadzhyan,~L.A.:
Equivalence of the nonlinear Schr\"odinger equation and the equation of a Heisenberg ferromagnet.
Teoret. Mat. Fiz. \textbf{38}, 26--35 (1979)
 
\end{thebibliography}
\end{document}